\documentclass[
    12pt,
     final
    ]{amsart}

    \usepackage[margin=1in]{geometry}
    \usepackage{diagbox} %for diagonal slash in table, see https://tex.stackexchange.com/a/11786/241569

    \usepackage{float}
        \newfloat{diag}{htp}{diag}
        \floatname{diag}{Diagram}

        \usepackage{stackengine} % https://tex.stackexchange.com/a/451796 so that we can define our action symbols below
        \usepackage{stix} % for the triangles, https://texdoc.org/serve/symbols-a4.pdf/0 p 145:
    \usepackage{amsmath,%amssymb, %%% amssymb is called by stix
    amsxtra,amsthm,hyperref}
    \usepackage{multicol}
    \usepackage{mathtools}
    \usepackage[final]{graphicx}
    \usepackage{tikz}
    \usetikzlibrary{calc,decorations.markings,arrows.meta}
    \usepackage{enumitem} % To change enumerate environment
    \usepackage{tensor} % To have left-subscripts
    \usepackage{tikz-cd} % For some diagrams
    \usepackage{upref}  % to textup-up all reference numbers

    \newcommand{\assumption}[1]{#1}
    
    \usepackage[normalem]{ulem}

    \newcommand{\CitationFromZStwo}{
        Proposition 6.13
    }

        \newcommand{\wtbeta}{\beta}
        \newcommand{\wtcB}{\cH\backslash \cB}
    
    \newcommand{\usc}{up\-per se\-mi-con\-ti\-nuous}
    \newcommand{\USCBb}{USC Ba\-nach bun\-dle}
    \newcommand{\ZS}{Zappa--Sz\'{e}p}
    \newcommand{\DR}{Deaconu--Renault}
    \newcommand{\CK}{Cuntz-Krieger}
    \newcommand{\etale}{\'{e}tale}
    \newcommand{\LCH}{locally compact Haus\-dorff}
    \newcommand{\ib}{im\-prim\-i\-tiv\-ity bi\-mod\-u\-le}
    \newcommand{\Endo}{\operatorname{End}}
    \renewcommand{\ss}{self-similar}
    \newcommand{\Ss}{Self-similar}
    \newcommand{\ssla}{\ss\ left action}
    \newcommand{\ssra}{\ss\ right action}
    \newcommand{\ssp}{\ss\ product}
    \newcommand{\sspa}{\ss\ para-equivalence}
    \newcommand{\Sspa}{\Ss\ Para-Equivalence}
    \newcommand{\intune}{in tune}% as an adverb, it should not be hyphenated
    \newcommand{\intuneadj}{in-tune}% as an adjective, it should be hyphenated
      
    \newtheorem{theorem}{Theorem}[section]
    \newtheorem{corollary}[theorem]{Corollary}
    \newtheorem{proposition}[theorem]{Proposition}   
    \newtheorem{conjecture}[theorem]{Conjecture} 
    \newtheorem{lemma}[theorem]{Lemma}
    
    \theoremstyle{definition}
    \newtheorem{notation}[theorem]{Notation}
    \newtheorem{definition}[theorem]{Definition}
    \newtheorem{assumptionthm}[theorem]{Assumption}
    
    \newtheorem{example}[theorem]{Example}
    \newtheorem{remark}[theorem]{Remark}
        \numberwithin{equation}{section}

    \let\ipscriptstyle=\scriptscriptstyle
    \def\lipsqueeze{{\mskip -2.0mu}}
    \def\ripsqueeze{{\mskip -2.0mu}}
    \def\ipcomma{\nobreak \mid\nobreak} %% https://tex.stackexchange.com/a/94220/241569: "\nobreak ... is an infinitely bad place to break (a line if in horizontal mode or a page if in vertical mode)."
    \newbox\ipstrutbox
    \setbox\ipstrutbox=\hbox{\vrule
        height5.5pt% height8.5pt%%% Makes sure that a superscript-*, for example, is pushed a bit further up
        depth 2pt% depth 3.5pt%%% This gives it buffer above and below
        width 0pt
    }

    \newcommand{\norm}[1]{\left\| #1 \right\|}
    \newcommand{\abs}[1]{\left| #1 \right|}
    
    \newcommand{\linner}[4][]{
        {%%% Needed to avoid double subscript error
        \relax_{\ipscriptstyle% very small font
                    #2
                    }^{\ipscriptstyle% very small font
                    #1}
                    \lipsqueeze% moves the sub- and super-scripts closer to the inner product
                    \langle #3 \ipcomma #4 \rangle
        }
    }
        \def\lip#1<#2,#3>{\linner{#1}{#2}{#3}}

    \newcommand{\rinner}[4][]{
        {%%% Needed to avoid double superscript error
         \langle #3
            \ipcomma
            #4 \rangle_{\ripsqueeze% moves the sub-script closer to the inner product
            \ipscriptstyle% very small font
            #2
            }^{\ripsqueeze% moves the super-script closer to the inner product
            \ipscriptstyle% very small font
            #1}
        }
    }
        \def\rip#1<#2,#3>{\rinner{#1}{#2}{#3}}

    \newcommand{\HleftX}{
        \mathchoice
              {\displaystyle\mathbin\smalltriangleright}
              {\textstyle\mathbin\smalltriangleright}
              {\scriptstyle\mathbin\smalltriangleright}
              {\scriptscriptstyle\mathbin\smalltriangleright}
        }

    \newcommand{\HrightX}{
        \mathchoice
              {\displaystyle\mathbin\smalltriangleleft}
              {\textstyle\mathbin\smalltriangleleft}
              {\scriptstyle\mathbin\smalltriangleleft}
              {\scriptscriptstyle\mathbin\smalltriangleleft}
        }

    \newcommand{\XleftG}{
        \mathchoice
              {\displaystyle\mathbin\smallblacktriangleright}
              {\textstyle\mathbin\smallblacktriangleright}
              {\scriptstyle\mathbin\smallblacktriangleright}
              {\scriptscriptstyle\mathbin\smallblacktriangleright}
        }

    \newcommand{\XrightG}{
        \mathchoice
              {\displaystyle\mathbin\smallblacktriangleleft}
              {\textstyle\mathbin\smallblacktriangleleft}
              {\scriptstyle\mathbin\smallblacktriangleleft}
              {\scriptscriptstyle\mathbin\smallblacktriangleleft}
        }
    
            \newcommand{\qHrightX}{
                \mathchoice
                  {\stackMath\mathbin{\stackinset{c}{-0.1ex}{c}{0ex}{\displaystyle\HrightX}{\displaystyle\bigcirc}}}
                  {\stackMath\mathbin{\stackinset{c}{-0.1ex}{c}{0ex}{\textstyle\HrightX}{\textstyle\bigcirc}}}
                  {\stackMath\mathbin{\stackinset{c}{-0.1ex}{c}{0ex}{\scriptstyle\HrightX}{\scriptstyle\bigcirc}}}
                  {\stackMath\mathbin{\stackinset{c}{-0.1ex}{c}{0ex}{\scriptscriptstyle\HrightX}{\scriptscriptstyle\bigcirc}}}
                }
            \newcommand{\qHleftX}{
                \mathchoice
                  {\stackMath\mathbin{\stackinset{c}{.1ex}{c}{0ex}{\displaystyle\HleftX}{\displaystyle\bigcirc}}}
                  {\stackMath\mathbin{\stackinset{c}{.1ex}{c}{0ex}{\textstyle\HleftX}{\textstyle\bigcirc}}}
                  {\stackMath\mathbin{\stackinset{c}{.1ex}{c}{0ex}{\scriptstyle\HleftX}{\scriptstyle\bigcirc}}}
                  {\stackMath\mathbin{\stackinset{c}{.1ex}{c}{0ex}{\scriptscriptstyle\HleftX}{\scriptscriptstyle\bigcirc}}}
                }
            \newcommand{\qXleftG}{
                \mathchoice
                    {\stackMath\mathbin{\stackinset{c}{0.1ex}{c}{0ex}{\displaystyle\XleftG}{\displaystyle\bigcirc}}}
                    {\stackMath\mathbin{\stackinset{c}{0.1ex}{c}{0ex}{\textstyle\XleftG}{\textstyle\bigcirc}}}
                    {\stackMath\mathbin{\stackinset{c}{0.1ex}{c}{0ex}{\scriptstyle\XleftG}{\scriptstyle\bigcirc}}}
                    {\stackMath\mathbin{\stackinset{c}{0.1ex}{c}{0ex}{\scriptscriptstyle\XleftG}{\scriptscriptstyle\bigcirc}}}
                }
            \newcommand{\qXrightG}{
                \mathchoice
                  {\stackMath\mathbin{\stackinset{c}{-0.1ex}{c}{0ex}{\displaystyle\XrightG}{\displaystyle\bigcirc}}}
                  {\stackMath\mathbin{\stackinset{c}{-0.1ex}{c}{0ex}{\textstyle\XrightG}{\textstyle\bigcirc}}}
                  {\stackMath\mathbin{\stackinset{c}{-0.1ex}{c}{0ex}{\scriptstyle\XrightG}{\scriptstyle\bigcirc}}}
                  {\stackMath\mathbin{\stackinset{c}{-0.1ex}{c}{0ex}{\scriptscriptstyle\XrightG}{\scriptscriptstyle\bigcirc}}}
                }
        \newcommand{\HleftB}{
                    \mathchoice
                          {\mathbin{\raisebox{-.9pt}{$\includegraphics[scale=1.3]{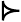}$}}}%display
                          {\mathbin{\raisebox{-.9pt}{$\includegraphics[scale=1.3]{semithickHleftB.pdf}$}}}%text
                          {\mathbin{\raisebox{-0.6pt}{$\includegraphics[scale=.8]{semithickHleftB.pdf}$}}}%script
                          {\mathbin{\raisebox{-.2pt}{$\includegraphics[scale=.6]{semithickHleftB.pdf}$}}}%scriptscript
                    }
                    
        \newcommand{\HrightB}{
                    \mathchoice
                          {\mathbin{\raisebox{-.9pt}{$\includegraphics[scale=1.3,angle=180,origin=c]{semithickHleftB.pdf}$}}}%display
                          {\mathbin{\raisebox{-.9pt}{$\includegraphics[scale=1.3,angle=180,origin=c]{semithickHleftB.pdf}$}}}%text
                          {\mathbin{\raisebox{-0.6pt}{$\includegraphics[scale=.8,angle=180,origin=c]{semithickHleftB.pdf}$}}}%script
                          {\mathbin{\raisebox{-.2pt}{$\includegraphics[scale=.6,angle=180,origin=c]{semithickHleftB.pdf}$}}}%scriptscript
                    }
                    
        \newcommand{\BrightG}{
                    \mathchoice
                          {\mathbin{\raisebox{-.9pt}{$\includegraphics[scale=1.3]{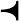}$}}}
                          {\mathbin{\raisebox{-.9pt}{$\includegraphics[scale=1.3]{semithickBrightG.pdf}$}}}
                          {\mathbin{\raisebox{-0.6pt}{$\includegraphics[scale=.8]{semithickBrightG.pdf}$}}}
                          {\mathbin{\raisebox{-.4pt}{$\includegraphics[scale=.6]{semithickBrightG.pdf}$}}}
                    }
        
        \newcommand{\BleftG}{
                    \mathchoice
                          {\mathbin{\raisebox{-.9pt}{$\includegraphics[scale=1.3,angle=180,origin=c]{semithickBrightG.pdf}$}}}
                          {\mathbin{\raisebox{-.9pt}{$\includegraphics[scale=1.3,angle=180,origin=c]{semithickBrightG.pdf}$}}}
                          {\mathbin{\raisebox{-0.6pt}{$\includegraphics[scale=.8,angle=180,origin=c]{semithickBrightG.pdf}$}}}
                          {\mathbin{\raisebox{-.4pt}{$\includegraphics[scale=.6,angle=180,origin=c]{semithickBrightG.pdf}$}}}
                    }

        \newcommand{\BbowtieHplaceholder}[2]{\includegraphics[scale= #1]{semithickHleftB.pdf} \mkern-#2mu \includegraphics[scale=#1 ,angle=180,origin=c]{semithickHleftB.pdf}}
        
        \newcommand{\BbowtieH}{%
            \mathchoice
                {\mathbin{\raisebox{-.9pt}{$\BbowtieHplaceholder{1.3}{6}$}}}%display
                {\mathbin{\raisebox{-.9pt}{$\BbowtieHplaceholder{1.3}{6}$}}}%text
                {\mathbin{\raisebox{-0.6pt}{$\BbowtieHplaceholder{.8}{4}$}}}%script
                {\mathbin{\raisebox{-.2pt}{$\BbowtieHplaceholder{.6}{3}$}}}%scriptscript
            }

        \newcommand{\GbowtieBplaceholder}[2]{\includegraphics[scale= #1,angle=180,origin=c]{semithickBrightG.pdf} \mkern-#2mu \includegraphics[scale=#1 ]{semithickBrightG.pdf}}
        
        \newcommand{\GbowtieB}{%
            \mathchoice
                {\mathbin{\raisebox{-.9pt}{$\GbowtieBplaceholder{1.3}{6}$}}}%display
                {\mathbin{\raisebox{-.9pt}{$\GbowtieBplaceholder{1.3}{6}$}}}%text
                {\mathbin{\raisebox{-0.6pt}{$\GbowtieBplaceholder{.8}{4}$}}}%script
                {\mathbin{\raisebox{-.2pt}{$\GbowtieBplaceholder{.6}{3}$}}}%scriptscript
            }

            \newcommand{\qHleftB}{
                \mathchoice
                  {\stackMath\mathbin{\stackinset{c}{.1ex}{c}{0ex}{\displaystyle\HleftB}{\displaystyle\bigcirc}}}
                  {\stackMath\mathbin{\stackinset{c}{.1ex}{c}{0ex}{\textstyle\HleftB}{\textstyle\bigcirc}}}
                  {\stackMath\mathbin{\stackinset{c}{.1ex}{c}{0ex}{\scriptstyle\HleftB}{\scriptstyle\bigcirc}}}
                  {\stackMath\mathbin{\stackinset{c}{.1ex}{c}{0ex}{\scriptscriptstyle\HleftB}{\scriptscriptstyle\bigcirc}}}
                }
            \newcommand{\qBrightG}{
                \mathchoice
                  {\stackMath\mathbin{\stackinset{c}{-0.1ex}{c}{0ex}{\displaystyle\BrightG}{\displaystyle\bigcirc}}}
                  {\stackMath\mathbin{\stackinset{c}{-0.1ex}{c}{0ex}{\textstyle\BrightG}{\textstyle\bigcirc}}}
                  {\stackMath\mathbin{\stackinset{c}{-0.1ex}{c}{0ex}{\scriptstyle\BrightG}{\scriptstyle\bigcirc}}}
                  {\stackMath\mathbin{\stackinset{c}{-0.1ex}{c}{0ex}{\scriptscriptstyle\BrightG}{\scriptscriptstyle\bigcirc}}}
                }

    \newcommand{\calt}{\curvearrowleft  }%CurveArrowLeftTop
    \newcommand{\cart}{\curvearrowright }%CurveArrowRightTop
    \newcommand{\calb}{\mathbin{\rotatebox[origin=c]{180}{$\cart$}}}%CurveArrowLeftBottom
    \newcommand{\carb}{\mathbin{\rotatebox[origin=c]{180}{$\calt$}}}%CurveArrowRightBottom
    \newcommand{\arrowlssa}{\mathbin{\substack{\textstyle \cart  \\ \calb }}}%%%% for a lssa
    \newcommand{\arrowrssa}{\mathbin{\substack{\textstyle \calt  \\ \carb }}}%%%% for a rssa

    \newcommand{\cA}{\mathscr{A}}
    \newcommand{\cB}{\mathscr{B}}
    \newcommand{\cC}{\mathscr{C}}
    \newcommand{\cM}{\mathscr{M}}
    
    \newcommand{\cG}{\mathcal{G}}
    \newcommand{\cK}{\mathcal{K}}
    \newcommand{\cH}{\mathcal{H}}
    \newcommand{\cX}{\mathcal{X}}

    \newcommand{\bfp}[2]{\lipsqueeze\tensor*[_{\ipscriptstyle #1}]{\ast}{_{\ipscriptstyle #2}}\ripsqueeze} 

    \newcommand{\inv}{^{-1}}

    \newcommand{\z}{^{(0)}}

    \newcommand*\dif{\mathop{}\nobreak    \mskip-\thinmuskip\nobreak    \mathrm{d} }%for d in integrals, see https://tex.stackexchange.com/a/95681

    \newcommand{\sme}{\,\mathord{\mathop{\text{--}}\nolimits_{\relax}}\,}

    \DeclareMathOperator{\supp}{supp}
    
    \newcommand\mvisiblespace[1][.7em]{%
        	\makebox[#1]{%
        		\kern.07em
        		\vrule height.3ex
        		\hrulefill
        		\vrule height.3ex
        		\kern.07em
        	}% <-- don't forget this one!
        }

\makeatletter
    \newcommand{\repeatable}[2]{%
        \label{#1}\global\@namedef{repeatable@#1}{#2}#2%
    }

    \newcommand{\txtrepeat}[1]{%
        \@ifundefined{repeatable@#1}{NOT FOUND}{\@nameuse{repeatable@#1}}%
    }

    \newcommand{\eqrepeat}[1]{%
        \@ifundefined{repeatable@#1}{NOT FOUND}{\begin{align*}\@nameuse{repeatable@#1}\tag{\ref{#1}}\end{align*}}%
    }
    
    \newcommand{\eqrepeatnn}[1]{%
        \@ifundefined{repeatable@#1}{NOT FOUND}{\begin{align*}\@nameuse{repeatable@#1}\end{align*}}%
    }
\makeatother
\title{Imprimitivity theorems and self-similar actions on Fell bundles}

\author{Anna Duwenig}
\address{KU Leuven, Department of Mathematics, Leuven (Belgium)}
\email{anna.duwenig@kuleuven.be}

\author{Boyu Li}
\address{Department of Mathematical Sciences, New Mexico State University, Las Cruces, New Mexico, 88003, USA}
\email{boyuli@nmsu.edu}
\date{\today}


\begin{document}

\begin{abstract} We introduce the notion of \ss\ actions of groupoids on other groupoids and Fell bundles. This leads to a new im\-prim\-i\-tiv\-ity theorem arising from such dynamics, generalizing many earlier im\-prim\-i\-tiv\-ity theorems involving group and groupoid actions. 
\end{abstract}

\thanks{The authors would like to thank Alex Mundey for helping with coding various symbols. The first-named author was supported by a RITA Investigator grant \mbox{(IV017)}
    from the University of Wollongong.
The second-named author was partially supported by Prof.\ Dilian Yang from University of Windsor.}

\maketitle
\setcounter{tocdepth}{1}
\tableofcontents

 \section{Introduction}

The dynamics between groups and operator algebras encompass a vast literature in the study of operator algebras. 
They trace back to the pioneering work of Murray and von Neumann \cite{MurrayVN1936}
where they encode group dynamics as operators on Hilbert spaces.
In its simplest form, a C*-dynamical system arises from 
a group acting by $*$-automorphism on a C*-algebra.
This system is then encoded by the C*-crossed product, where both the group and the C*-algebra are represented as operators on a Hilbert space. 
One may refer to Williams's book \cite{WilliamsBook} for a thorough discussion of the subject.

The C*-crossed product construction bears a strong resemblance to the semi-direct product of groups, in which one group $H$ acts on another group $G$ by automorphisms. 
Their semi-direct product $G\rtimes H$ is a group that encodes both groups and their interaction. 
But what happens if the group $G$ also acts on $H$? 
This leads to a more general construction called the {\em \ZS\ product} of groups (also known as {\em bicrossed product} or {\em knit product}), which encodes a two-way action between two groups. 
Such a two-way action may arise when a group $K$ contains two subgroups $H,G$ such that every element $k\in K$ decomposes uniquely as a product $k=gh$ where $g\in G, h\in H$ (equivalently, $K=G\cdot H$ and $G\cap H=\{e\}$). 
In this case, for each $g\in G$ and $h\in H$, there exists unique $g'\in G$ and $h'\in H$ such that $hg=g'h'$. 
This leads to an $H$-action on $G$ via $(h,g)\mapsto g'$ and a $G$-action on $H$ via $(h,g)\mapsto h'$. 
These two actions need to satisfy certain compatibility conditions, and one may recover the enveloping group $K$ as the \ZS\ product group $G\bowtie H$ from these compatible actions.

In the realm of operator algebras, the  analogous study of \ZS\ products is scarce.
Representations of \ZS\ products of matched pair of groupoids were 
studied in \cite{AA2005}. The \ZS\ product of \etale\ groupoids and their C*-algebras were first studied in \cite{BPRRW2017}. Recently, we defined and studied 
an operator algebraic analogue of such products \cite{DL2021}. Just like the C*-crossed product $A\rtimes H$ 
is an operator algebraic analogue of the semi-direct product of two groups $G\rtimes H$, so is our construction an analogue of the \ZS\ product $\cG\bowtie\cH$ of two groupoids. 
To achieve this, the operator algebraic data has to `act' on the groupoid $\cH$; this is achieved by replacing the C*-algebra $A$ by a Fell bundle $\cB\to\cG$ 
on which the groupoid $\cH$ acts in an appropriate sense to form the Fell bundle $\cB\bowtie\cH\to \cG\bowtie\cH$.
The resulting Fell bundle C*-algebra of these \ZS\ dynamics is a generalization of the classical C*-crossed product, and we proved that several properties of the C*-crossed product hold similarly in the \ZS\ construction.

Given the vast literature on C*-dynamical systems, our study unlocks a trove of intriguing questions on what properties of C*-crossed products can be generalized to the \ZS\ product context. 
In this paper, we prove a \ZS\ analogue of the im\-prim\-i\-tiv\-ity theorems arising from groupoid actions.
Imprimitivity theorems originated from Mackey's study on inducing representations of a locally compact group $G$ from its closed subgroups and giving criteria
to identify such representations, known as {\em Mackey's machine} \cite{Mackey:MackeyMachine}.
Along with the rapid development of
the C*-algebra theory, Mackey's im\-prim\-i\-tiv\-ity theorems were soon recast in terms of C*-algebras in the early 1970s by Rieffel \cite{Rieffel:Mackey1, Rieffel:Mackey2}, where he introduced the notion of Morita equivalence for
C*-algebras \cite{Rieffel1974}. 
One may refer to Rosenberg's survey paper \cite{Rosenberg:Mackey} on the rich history of this subject. 
Since then, the theory of im\-prim\-i\-tiv\-ity theorems and Morita equivalence among C*-algebras has been further developed. For im\-prim\-i\-tiv\-ity theorems arising from group dynamics, notable works includes Green's \cite{Green1978} and Raeburn's \cite{Raeburn1988} symmetric im\-prim\-i\-tiv\-ity theorems. One may refer to \cite[Chapter 4]{WilliamsBook} for various versions and applications of these results. In \cite{MRW:Grpd}, Muhly, Renault, and Williams introduced the notion of {\em equivalent groupoids} which implies the existence of a Morita equivalence between their C*-algebras. 
This was generalized to Fell bundle C*-algebras by Muhly and Williams in \cite{MW2008} (see also \cite{Yamagami:preprint}). Applying the technique developed by Muhly and Williams, Kaliszewski et.\ al.\ \cite{KMQW2013} recovered and extended \emph{``all known im\-prim\-i\-tiv\-ity theorems involving groups''} by using a semi-direct product construction of Fell bundles by locally compact groups. 

The main theorem of this paper (Theorem~\ref{thm.imprimitivity.main}) further generalizes the im\-prim\-i\-tiv\-ity theorem of Kaliszewski et.\ al.\ beyond the realm of semi-direct products and to 
the realm of \ZS\ products. This opens a new world of study on the \ZS-type two-way interactions between
groupoids and Fell bundles.

\smallskip 

We briefly outline the key ideas and constructions of this paper. 
We first introduce the notion of \emph{\ss\ actions} of a groupoid $\cH$ on another groupoid $\cX$ in Section~\ref{sec.ssa} and construct their \ssp\ groupoid $\cX\bowtie \cH$. We adopted this terminology in order to
differentiate our new construction from earlier, more restrictive \ZS\ product constructions \cite{AA2005, BPRRW2017}: 
we no longer require the groupoids to have the same unit space. 
Rather, the groupoids are connected using a momentum map, similar to the idea of a semi-direct product of groupoids in \cite{HKQW2021}. 
This allows us to study many interesting examples such as group actions on groupoids. 
We also removed the requirement imposed in our earlier paper \cite{DL2021} that the groupoids be
\etale: unless stated otherwise, all groupoids are merely assumed to be 
{\em \LCH\ and second countable.} 
Consequently, our new construction is an honest generalization of that in \cite{KMQW2013}, and our notion of a \ss\ action is
a generalization of self-similar group actions whose close relationship to \ZS\ products has already been studied \cite{EP2017, LY2019b, Nekrashevych2009}. At the end of Section~\ref{sec.ssa}, we induce Haar systems from $\cX$ and $\cH$ to a Haar system on $\cX\bowtie\cH$ under mild assumptions.

In Section~\ref{sec.orbit}, we start by studying the orbit space $\cH\backslash \cX$ of a \ssla\ of $\cH$ on $\cX$, which is also a groupoid as long as the action is free and proper.
In the setup of most symmetric im\-prim\-i\-tiv\-ity theorems, it is standard to assume that the left $\cH$ action on $\cX$ commutes with a right action of another groupoid, $\cG$, yielding two groupoids of the form $(\cX/\cG)\rtimes \cH$ and $\cG\ltimes(\cH\backslash\cX)$ that are equivalent. This assumption is not quite 
enough in the \ssp\ setting. We therefore introduce the notion of \emph{\intune} actions (Definition~\ref{df.compatible}), and we
call $\cX$ a {\em $(\cH,\cG)$-\sspa} if the $\cH$ and $\cG$ actions are free, proper, and \intune, and if $\cX$ has open source map. 
Under such assumptions, the $\cH$- and $\cG$-actions on $\cX$ factor through the respective opposite quotient: 
$\cH$ naturally has a 
\ssla\ on $\cX/\cG$ and 
$\cG$ a \ssra\ on $\cH\backslash\cX$, allowing
us to build their \ssp\ groupoids $(\cX/\cG)\bowtie \cH$ and $\cG\bowtie(\cH\backslash\cX)$. We prove (Theorem~\ref{thm.groupoid.eq}) that these two groupoids are equivalent in the sense of \cite{MRW:Grpd}. Moreover, 
the existence of a Haar system on $\cX$ that is equivariant in an appropriate sense allows us to build Haar systems for these equivalent groupoids, so that their groupoid C*-algebras are Morita equivalent. 

In Sections~\ref{sec.ssa.Fell} and \ref{sec.orbit.Fell},
we bootstrap our construction to the more operator algebraic setting of {\em \ss\ actions on Fell bundles} $\cB\to \cX$ for $\cX$
a $(\cH,\cG)$-\sspa. We define the notions of \ss\ left and right actions on $\cB$ following similar ideas as in \cite{DL2021}.
This allows two constructions: that of their 
\ssp s $\cB\BbowtieH\cH$ and $\cG\GbowtieB\cB$, where the color of the symbol distinguishes between left- resp.\ right-actions, and  that of the orbit spaces $\cH\backslash\cB$ and $\cB/\cG$. Assuming the actions are free, proper, and \intune, the orbit spaces become Fell bundles themselves. By iterating these constructions, 
we obtain two Fell bundles, $(\cB/\cG)\BbowtieH \cH$ and $\cG\GbowtieB(\cH\backslash\cB)$. 

Our main theorem (Theorem~\ref{thm.imprimitivity.main}) in Section~\ref{sec.main}  states that these two Fell bundles are equivalent in the sense of \cite{MW2008}. Again, under suitable additional assumptions regarding Haar systems,
their Fell bundle C*-algebras are therefore Morita equivalent. 

We note that the im\-prim\-i\-tiv\-ity theorem of Kaliszewski et.\ al.\ can be recovered by requiring that half of our two two-way actions be trivial (namely, that $\cX$ does not act on $\cH$ or $\cG$). There are other examples where the $\cX$ actions on $\cG$ and $\cH$ are non-trivial, some of which are briefly discussed (Examples~\ref{ex.example1 - part 1}, ~\ref{ex.example1 - part 2},~\ref{ex.example1 - part 3},
and~\ref{ex.example.concrete}).

Finally, we apply our result to a certain class of \DR\ groupoids generated by $*$-commuting endomorphisms in Section~\ref{sec.Deaconu}. We specialize this example to a class of 2-graphs and prove that their higher rank graph C*-algebras are all Morita equivalent to $C(\mathbb{T})$ (Corollary~\ref{cor.rank2graph.main}). 

\smallskip 

Due to the sheer number of actions involved, we try our best to assign each action a unique symbol to best avoid confusion. By convention, the arrow of each action symbol will point to the  element of the space that is acted upon.
We include a summary of all actions, together with their notation, assumptions, and key properties, in Appendix~\ref{sec.appendix.cheatsheet} for quick reference. 

 \section{\Ss\ actions}\label{sec.ssa}

\Ss\ groups originated from Grigorchuk's construction of finitely generated groups of intermediate growth \cite{Grigorchuk1984, Grigorchuk1980}. Its application in operator algebra was first explored by Nekrashevych \cite{Nekrashevych2009} where he studied a self-similar group acting
on a set. The distinctive feature that set it apart from other group actions is that the set also acts back on the group; this action is often called the {\em restriction map}.
Such a two-way interaction has since been generalized to various contexts; 
for example to self-similar actions on directed graphs \cite{EP2017}, $k$-graphs \cite{LY2019b}, and semigroups \cite{BRRW, Starling2015}. In this section, we define 
\ss\ groupoid actions on groupoids. Again, the key feature that
sets our definition apart from classical groupoid actions is the two-way interactions recorded in these \ss\ dynamics. 

\subsection{\Ss\ left actions on groupoids}{}

\begin{notation}
    Given continuous maps $f\colon X\to Z$ and $g\colon Y\to Z$ between topological spaces, we write 
    \[
    X\bfp{f}{g}Y
    \coloneq 
    \{(x,y)\in X\times Y \mid f(x)=g(y)
    \}
    \]
    and equip this space with the subspace topology.
\end{notation} 

\begin{definition}\label{df.left.selfsimilar}\label{df.X-H-actions}
Let~$\cH$ and~$\cX$ be two \LCH\ groupoids.
We say~$\cH$ has a \textit{\ssla} on~$\cX$ if there exist a continuous surjection $\rho_{\cX}\z\colon \cX\z\to\cH\z$ and, using the momentum map $\rho_{\cX}\coloneqq \rho_{\cX}\z\circ r_{\cX}$, 
two continuous maps
\begin{align*}
\repeatable{eq:X-H-actions}{
    \cH\cart \cX\colon &&\cH
        \bfp{s_{\cH}}{\rho_{\cX}}
    \cX\ni (h,x) &\mapsto h\HleftX x \in \cX \\
    \cH\calb \cX\colon &&\cH
        \bfp{s_{\cH}}{\rho_{\cX}}
    \cX\ni (h,x) &\mapsto h\HrightX x \in \cH}
\end{align*}
such that the following hold.
\begin{itemize}
    \item
    For any $h\in \cH$ and $x\in \cX$ such that $s_{\cH}(h)=\rho_{\cX}(x)$, we have:
    \begin{equation}\tag{L1}
        \repeatable{item:L1}
        {
            \begin{aligned}
                r_{\cH}(h)&=\rho_{\cX}(h\HleftX x)
                &&&
                s_{\cH}(h\HrightX x)&=\rho_{\cX}(x\inv)
                &&&
                r_{\cH}(h\HrightX x)&= \rho_{\cX}((h\HleftX x)\inv)
            \end{aligned}%
        }
    \end{equation}%
    \item
    
    For all $h\in \cH$ and $v\in \cX\z $ such that $s_{\cH}(h)=\rho_{\cX}(v)$,
    and for all $x\in \cX$, we have:
    \begin{equation}
        \tag{L2}
        \repeatable{item:L2}
        {
            h\HrightX v=h
            \quad\text{and}\quad
            \rho_{\cX}(x)
        \HleftX x=x
        }
    \end{equation}

    \item 
    For all $h\in \cH$ and $(x,y)\in \cX^{(2)}$ such that $s_{\cH}(h)=\rho_{\cX}(x)$, we have 
    $s_{\cX}(h\HleftX x)= r_{\cX}((h\HrightX x)\HleftX y)$ and
    \begin{align}
        \tag{L3}\repeatable{item:L3}
        {
        h\HrightX (xy)
        &
        =(h\HrightX x)\HrightX y
        }
        \\
        \tag{L4}\repeatable{item:L4}
        {        
        h\HleftX (xy)
        &
        =(h\HleftX x)[(h\HrightX x)\HleftX y]
        }
    \end{align}
    \item
    For all $(h,k)\in \cH^{(2)}$ and $x\in \cX$ such that $s_{\cH}(k)=\rho_{\cX}(x)$, we have:
    \begin{align}
        \tag{L5}\repeatable{item:L5}  
        {
            (hk)\HleftX x&=h\HleftX (k\HleftX x)
        }
        \\
        \tag{L6}\repeatable{item:L6}
        {
        (hk)\HrightX x&=[h\HrightX (k\HleftX x)](k\HrightX x)
        }
    \end{align}
\end{itemize}
\end{definition}

We will often write 
    $\cH\bfp{s
    }{\rho
    } \cX $ instead of  $\cH\bfp{s_{\cH}}{\rho_{\cX}}\cX$ when the subscripts are clear from context.

\begin{remark}
    The two maps $\HleftX$ and $\HrightX$ plus the equalities in \eqref{item:L1} can also be summarized by commutativity of the following diagram:
    \[  
    \begin{tikzcd}[ampersand replacement=\&, column sep = small]
    \cH
    \bfp{s}{\rho\circ r}
    \cX
    \ar[rrrr, "{\displaystyle(\mvisiblespace\HleftX\mvisiblespace \ , \ \mvisiblespace\HrightX\mvisiblespace)}"]
\ar[dddrr,  rounded corners, to path={ 
      --([yshift = -20ex]\tikztostart.south) 
      --node[right]{$(r_{\cH} \ ,\ \rho_{\cX}\circ s_{\cX})$} 
      (\tikztostart.south) 
      |- (\tikztotarget.west)
      }
      ]
    \&
    \&
    \&
    \&
    \cX
    \bfp{\rho\circ s}{r}
    \cH
\ar[dddll,  rounded corners, to path={ 
      --([yshift = -20ex]\tikztostart.south) 
      --node[left]{$(\rho_{\cX}\circ r_{\cX} \ , \  s_{\cH})$} 
      (\tikztostart.south) 
      |- (\tikztotarget.east)
      }]
    \\%[-15ex]
    \&
    (h,x)
    \ar[rr, mapsto]
    \ar[d, mapsto]
    \&
    \&
    (h\HleftX x, h\HrightX x)
    \ar[d, mapsto]
    \&
    \\
    \&
    \bigl(r_{\cH}(h), \rho_{\cX}(x\inv)\bigr)
    \ar[rr, equal]
    \&
    \&
    \bigl(\rho_{\cX}(h\HleftX x), s_{\cH}(h\HrightX x)\bigr)
    \&
    \\
    \&\&
    \cH\z
    \times
    \cH\z
    \&\&
    \end{tikzcd}
    \]
\end{remark}

\begin{example}\label{ex:matched pair is ssla}
    Suppose $\cX$ and $\cH$ are groupoids with $\cX\z=\cH\z$. Then $(\cX,\cH)$ is a matched pair of groupoids in the sense of \cite[Definition 1.1]{AA2005} if and only if $\cH$ has a \ssla\ on $\cX$ with $\rho_{\cX}\z= \mathrm{id}_{\cX\z}$, meaning that  $\rho_{\cX}= r_{\cX}$. We point out that this is the reason that inverse elements appear in Condition~\eqref{item:L1}: Here, the condition $s_{\cH}(h\HrightX x)=\rho_{\cX}(x\inv)$ becomes $s_{\cH}(h\HrightX x)=
    s_{\cX}(x)$, which might feel a bit more natural.
\end{example}

\begin{remark}\label{rm.ssla means cH acts on cX}
    If~$\cH$ has a \ssla\ on~$\cX$, then $(h,x)\mapsto h\HleftX x$ is a left action of the groupoid~$\cH$ on the space~$\cX$ 
    with momentum map $\rho_{\cX}$
    in the sense of \cite[Def.~2.1]{Wil2019}. Indeed, the algebraic properties needed for an action are 
    \begin{equation}\label{eq:assumptions on an action}
    \rho_{\cX}(h\HleftX x) = r_{\cH}(h),\quad \rho_{\cX} (x) \HleftX x = x,\quad  \text{ and }\quad (kh)\HleftX x=k\HleftX(h\HleftX x),
    \end{equation}
    which are all assumed in~\eqref{item:L1},~\eqref{item:L2}, resp.~\eqref{item:L5}.
    
    Moreover, if $\cX\z=\cH\z$ and 
        $\rho_{\cX}\z=\mathrm{id}_{\cX\z}$,
    then $(h,x)\mapsto h\HrightX x$ is a right action of the groupoid~$\cX$ on the space~$\cH$ 
    with
    momentum map $s_{\cH}$. 
\end{remark}

\begin{example}\label{ex:ssla ON trivial gpd}
    \repeatable{assumptions of ex:ssla ON trivial gpd}{%
        Suppose $\cH$ acts on a groupoid $\cX$ by automorphisms%
    }%
    , meaning $\cX$ has a
    continuous, surjective
    momentum map $\rho_{\cX}\colon \cX\to\cH\z$ and there is a continuous map $\cH\bfp{s}{\rho}\cX\to \cX$ satisfying not only the conditions in~\eqref{eq:assumptions on an action} but also
    $h\HleftX (xy)=(h\HleftX x)(h\HleftX y)$ where it makes sense.
    Then $\HleftX$ is a \ssla\ of $\cH$ on $\cX$ if and only 
    if we let $\cX$ act trivially on $\cH$ (meaning $h\HrightX x=h$). Note that there is no other choice for $\HrightX$ because of Condition~\eqref{item:L4} in combination with the  assumption that $\HleftX$ is an action by homomorphisms.
\end{example}%

\begin{example}[see {\cite[Example 1.6.]{AA2005}}]\label{ex:ssla OF trivial gpd}
    Suppose we are given a groupoid $\cX$. If we let $\cH=\cX\z$ be the trivial groupoid and let 
    $\rho_{\cX}\z=\mathrm{id}_{\cX\z}$, so that $\rho_{\cX}=r_{\cX}$,
    then we can define for a tuple $(u,x)=(r_{\cX}(x),x)\in \cX\z\bfp{s}{r}\cX$,
    \begin{align*}
        \cX\z\cart \cX\colon &&
        r_{\cX}(x)\HleftX x &= x, \\
        \cX\z\calb \cX\colon &&
        r_{\cX}(x)\HrightX x &= s_{\cX}(x).
    \end{align*}
    One swiftly verifies that these constitute a \ssla\ of $\cX\z$ on $\cX$. (In fact, these groupoids form a matched pair.)
\end{example}

    We point out that, in order for the condition $s_{\cH}(h\HrightX x)=\rho_{\cX}(x\inv)$ in \eqref{item:L1} to be satisfied by the pair in Example~\ref{ex:ssla OF trivial gpd}, we must define $\cX\z\calb \cX$ in the above way and cannot let $\cX$ act trivially on~$\cX\z$.
    For Example~\ref{ex:ssla of trivial gp and gpd:sometimes a something} later, it will therefore be convenient to know that we can also replace the trivial groupoid~$\cX\z$ with the trivial group~$\{e\}$ as follows. This also highlights the advantage of not having forced $\cX$ and $\cH$ to have the same unit space, as was the case in, for example, \cite{AA2005,BPRRW2017}.
\begin{example}\label{ex:ssla of trivial group}
    Suppose we are given a groupoid $\cX$. If we let $\cH=\{e\}$ be the trivial group, so that 
    $\rho_{\cX}\z\colon \cX\z\to \cH\z=\{e\}$
    is constant
    and so that $\HleftX$ and $\HrightX$ must be defined to be trivial, then these constitute a \ssla\ of $\{e\}$ on $\cX$.
\end{example}

    \begin{example}\label{ex.example1 - part 1}
        Suppose a \LCH\ group $K$ acts on the left on a \LCH\ space $X$, denoted by $\ast$.
        Suppose further that $K$ can be written as an (internal) \ZS\ product of two (necessarily closed) subgroups, i.e., $K=G\bowtie H$ with 
        the product topology. 
        This means that, for any $h\in H$ and $t\in G$, there exist unique elements $h|_{t}\in H$ and $h\cdot t\in G$ such that 
        $(e,h)(t,e)=(h\cdot t , h|_t)$, where the product on the left-hand side is the group multiplication of~$K$ and where~$e$ denotes the identity element of each group.
        
        Consider the transformation groupoid 
        $\cX=G\ltimes X = \{(t,x): x\in X, t\in G\}$;
        we choose the convention that its range and source maps are $r(t,x)=t\ast x$ resp.\ $s(t,x)=x$. 
        Then 
        \begin{align}
        \begin{split}\label{eq:actions in example1}
            H \cart \cX\colon \quad h\HleftX (t,x) &= (h\cdot t, h|_t \ast x)    \\        
            H \calb \cX\colon \quad h\HrightX (t,x) &= h|_t 
        \end{split}
        \end{align}
        is a \ssla\ of $H$ on $\cX$. 
    \end{example}

Note that units are not necessarily fixed by \ss\ actions.
Instead, we have the following formulas:
\begin{lemma}\label{lem:acting_on_units}
    For any $x\in \cX$ and for any $(h,v)\in 
    \cH \bfp{s}{\rho} \cX\z
    $, we have
    \begin{align}\tag{L7}
    \repeatable{item:L7}
          	{
          	\rho_{\cX}(x)\HrightX x
           =
            \rho_{\cX}(x\inv)
          	}
          	\\
            \tag{L8}\repeatable{item:L8-new}
          	{
          	 h\HleftX v\in \cX\z
          	}
           %  \tag{L8}\repeatable{item:L8}
          	% {
          	%  \rho_{\cX}(h\HleftX v)&=r_{\cH}(h)
          	%  }
    \end{align}
    Moreover, if 
    $s_{\cH}(h)=\rho_{\cX}(x)$, then
    \begin{align}\tag{L9}\repeatable{item:L9}
    {
    (h\HleftX x)\inv  &= (h\HrightX x) \HleftX x\inv 
    &&\text{and}&&&
    (h\HrightX x)\inv  &= 
    h\inv  \HrightX (h\HleftX x)
    }
    \\
    \tag{L10}\repeatable{item:L10}
	{
	r_{\cX}(h\HleftX x)&=h \HleftX r_{\cX}(x)
    &&\text{and}&&&
    s_{\cX}(h\HleftX x)&=(h\HrightX x)\HleftX s_{\cX}(x)
	}
    \end{align}
\end{lemma}

\begin{proof} Let $e=\rho_{\cX}(x)\in\cH\z$. For \eqref{item:L7},
\[e\HrightX x=(e^2) \HrightX x\overset{\eqref{item:L6}}{=}(e\HrightX (e\HleftX x))(e\HrightX x)\overset{\eqref{item:L2}}{=}(e\HrightX x)^2.\]
Hence, $e\HrightX x\in \cH\z $. Therefore, % Condition~\eqref{item:L1}, 
\[e\HrightX x=s_{\cH}(e\HrightX x)\overset{\eqref{item:L1}}{=}
    \rho_{\cX}(x\inv)
.\]
Condition \eqref{item:L8-new} follows from
%For \eqref{item:L8}, we first observe that
\[h\HleftX v=h\HleftX (v^2) \overset{\eqref{item:L4}}{=} (h\HleftX v)((h\HrightX v)\HleftX v)\overset{\eqref{item:L2}}{=}(h\HleftX v)^2.\]

For \eqref{item:L9}, note that we have just shown that 
$h\HleftX (xx\inv )\in \cX\z $. By Condition~\eqref{item:L4}, 
\[
\cX\z \ni h\HleftX (xx\inv ) = (h\HleftX x)[(h\HrightX x) \HleftX x\inv ] .\]
Therefore, $(h\HleftX x)\inv =(h\HrightX x) \HleftX x\inv $. Similarly, by
what we have proved above,
$(h\inv h)\HrightX x\in\cH\z $. Therefore, by Condition~\eqref{item:L6},
\[\cH\z\ni (h\inv h)\HrightX x = [h\inv  \HrightX (h\HleftX x)] (h\HrightX x) .\]
This proves that $(h\HrightX x)\inv  = 
h\inv  \HrightX (h\HleftX x)$.

 Lastly, for \eqref{item:L10}, we compute
\[h\HleftX x=h\HleftX (r_{\cX}(x)x)\overset{\eqref{item:L4}}{=}(h\HleftX r_{\cX}(x))[(h\HrightX r_{\cX}(x))\HleftX x]\overset{\eqref{item:L2}}{=}(h\HleftX r_{\cX}(x))(h\HleftX x)\]
and
\[
    h\HleftX x
    =
    h\HleftX (x s_{\cX}(x))
    \overset{\eqref{item:L4}}{=}
    (h\HleftX x)[(h\HrightX x)\HleftX s_{\cX}(x)].\qedhere
\]
\end{proof}

\begin{corollary}\label{cor:q(non unit) is non unit}
    If~$\cH$ has a \ssla\ on~$\cX$ and if $h\HleftX x$ is a unit in $\cX$, then $x$ is a unit. 
\end{corollary}
\begin{proof}
    By Lemma~\ref{lem:acting_on_units}, $\cH$ maps units to units. In particular,  $x \overset{\eqref{item:L5}}{=} h\inv \HleftX (h\HleftX x)$ is a unit.
\end{proof}

Since $\HleftX$ is a left groupoid action of $\cH$ on the space $\cX$ (Remark~\ref{rm.ssla means cH acts on cX}), we make the following definitions, which
are standard in the literature.

\begin{definition}\label{df.free, proper}
    If~$\cH$ has a \ssla\ on~$\cX$, we call it
    {\em free} if $\HleftX$ is free, meaning that the equality $h\HleftX x=x$ implies $h\in\cH\z $.    Likewise, we call it {\em proper}  if $\HleftX$ is proper, meaning that the map $\cH \bfp{s}{\rho} \cX
    \to \cX\times\cX$ defined by $(h,x)\mapsto (h\HleftX x, x)$ is a proper map.
\end{definition}

We note that these conditions on $\HleftX$ do not impose conditions on $\HrightX$.

\begin{example}
\label{ex:ssla of trivial gp and gpd:free and proper} 
    Given a groupoid $\cX$, the (trivial) \ssla s of the trivial groupoid $\cX\z$ and of the trivial group $\{e\}$ on $\cX$ (Examples~\ref{ex:ssla OF trivial gpd} and~\ref{ex:ssla of trivial group})  are both free and proper.
\end{example}%

    \begin{example}[%
        continuation of Example~\ref{ex.example1 - part 1}%
        ]\label{ex.example1 - part 2}  
        Suppose again that a \LCH\ group $K=G\bowtie H$ acts on the left on a \LCH\ space $X$, denoted by $\ast$. We define the \ssla\ $\HleftX$ and $\HrightX$  of $H$ on the transformation groupoid $\cX=G\ltimes X$ as in~\eqref{eq:actions in example1}.
        
        Note that, if $\ast$ is free, then so is $\HleftX$:
        suppose $h\HleftX (t,x)=(t,x)$,
        i.e., $h\cdot t=t$ and $h|_t \ast x=x$. By the freeness of the $K$-action on $X$, this forces $h|_t=e$. Recall that the \ZS-structure of $K$ implies that 
        $(e,h)(t,e)=(h\cdot t , h|_t)$. But the right-hand side equals $(t, e)$, which forces $h=e$.

        Likewise, if $\ast$ is proper, then so is $\HleftX$:
        suppose that we have  convergent nets   $(t_{i},x_{i})\to (t,x)$ and $h_{i}\HleftX (t_{i},x_{i})\to (s,y)$ in $\cX$; we must check that $h_{i}$ has a convergent subnet.
        By definition of $\HleftX$, we know in particular that 
        $h_{i}|_{t_{i}}\ast x_{i}\to y$ in $X$. As $x_{i}\to x$ and as $\ast$ is proper, it follows that $h_{i}|_{t_{i}}$ (has a subnet that) converges to, say, $k$ in $K$. Since $H$ is closed in $K$, $k$ is an element of $H$, and so by continuity of the restriction and inversion map, we conclude that $h_{i}= (h_{i}|_{t_{i}})|_{t_{i}\inv} \to k|_{t\inv}$. % The claim follows.
    \end{example}

\begin{lemma}\label{lem:restricted action on cXz}
    If~$\cH$ has a \ssla\ on~$\cX$, then $\HleftX$ 
    restricts to a continuous left action of~$\cH$ on the unit space, $\cX\z$. 
    The action on $\cX$ is free 
        (resp.\ proper)
    if and only the action on $\cX\z$ is free 
        (resp.\ proper)
    . 
\end{lemma}

\begin{proof}
    Notice first that, if $v\in\cX\z$ and $h\in\cH$ are such that $s_{\cH}(h)=\rho_{\cX}(v)$, then $h\HleftX v \in \cX\z$ by
    Lemma~\ref{lem:acting_on_units} \eqref{item:L8-new},
    so the map restricts to
        a continuous action $\cH\bfp{s}{\rho}\cX\z\to \cX\z$ with momentum map $\rho_{\cX}\z\colon \cX\z \to \cH\z$.

    Now suppose the action on $\cX\z$ is free, and assume that $h\HleftX x = x$ for some $x\in \cX$. Then
    \[
        x\inv = (h\HleftX x)\inv \overset{\eqref{item:L9}}{=} (h\HrightX x)\HleftX x\inv, 
    \]
    so that
    \begin{align*}
        h\HleftX (x x\inv) 
        \overset{\eqref{item:L4}}{=}
        (h\HleftX x)\bigl[ (h\HrightX x)\HleftX x\inv\bigr]
        =
        xx\inv.
    \end{align*}
    As $xx\inv\in \cX\z$, our assumption now implies that $h$ is a unit, proving that $\HleftX$ is free. The other direction of the equivalence is trivial.

        Lastly suppose that the action on $\cX\z$ is proper, and assume that the net $\{(h_\lambda\HleftX x_\lambda,x_\lambda)\}_\Lambda$ converges to $ (y,x)$ in $\cX\times\cX$. By \eqref{item:L10} and continuity of $r_{\cX}$, this implies that $(h_\lambda\HleftX r_{\cX}(x_\lambda),r_{\cX}(x_\lambda)) \to (r_{\cX}(y),r_{\cX}(x))$. By properness on $\cX\z$, it follows from \cite[Proposition 2.17]{Wil2019} that $\{h_\lambda\}_\Lambda$ has a convergent subnet. By the same proposition, this implies that $\cH$ acts properly on $\cX$. 
\end{proof}

   The above implies that a non-trivial groupoid $\cH$ cannot admit a free \ssla\ on a group $\cX$, because its action on the unit space $\{e\}$ of $\cX$ is never free.

\begin{lemma}\label{lm.uniqueness}
Let~$\cH$ act on~$\cX$ by a {\em free} \ssla.  
If $x,x'\in \cX$ satisfy $\cH\HleftX x=\cH\HleftX x'$ and if $r_{\cX}(x)=r_{\cX}(x')$, then $x=x'$.
\end{lemma}

\begin{proof} Since $\cH\HleftX x=\cH\HleftX x'$, there exists $h\in \cH$ such that $x'=h\HleftX x$. By \eqref{item:L10} (Lemma~\ref{lem:acting_on_units}), $h\HleftX r_{\cX}(x)=r_{\cX}(h\HleftX x)= r_{\cX}(x')$, which coincides with $r_{\cX}(x)$ by assumption. Since the \ss\ $\cH$-action is free, $h$ must be in $\cH\z $ and thus $x'=x$. 
\end{proof}

\subsection{The \ssp\ groupoid: A generalized \ZS\ product}

Following \cite{BPRRW2017} and \cite[Example 2.4]{DL2021}, we can define a \ZS-type product of~$\cH$ with~$\cX$; the main difference is that we do not require
the unit spaces of the two groupoids to coincide.

\begin{definition}\label{def.ZSProduct.groupoid.left}
Let~$\cH$ be a groupoid that has a (not necessarily free or proper) \ssla\ on~$\cX$ (Definition~\ref{df.left.selfsimilar}). 
The {\em \ssp} of~$\cX$ and~$\cH$ is the set
\[\cX\bowtie \cH=\{(x,h)\in \cX\times\cH: 
    \rho_{\cX}(x\inv)
=r_{\cH}(h)\}\]
with the following structure of a groupoid: the unit space is $$(\cX\bowtie\cH)\z=(\cX\z\times\cH\z)\cap(\cX\bowtie\cH)$$
and its range and source maps are given by
\[
    r_{\cX\bowtie\cH} (x,h) = \bigl(r_{\cX}(x) , r_{\cH}(h)\HrightX x\inv\bigr)
\quad\text{resp.}\quad
    s_{\cX\bowtie\cH} (x,h)
    =\bigl(h\inv \HleftX s_{\cX}(x), s_{\cH}(h)\bigr).
\]
Two elements $(x,h)$ and $(y,k)$ are composable if and only if  
\(
    s_{\cX}( x)
        =h \HleftX 
        r_{\cX}(y)
        ,
\)
in which case their composition is defined by
\[(x,h)(y,k)\coloneqq (x(h\HleftX y), (h\HrightX y)k).
\]
Lastly, the inverse is
\[(x,h)\inv \coloneqq  (h\inv \HleftX x\inv , h\inv \HrightX x\inv ).\]
\end{definition}

\begin{remark}
    Let us do some sanity checks.
    
    \paragraph{\itshape The range map lands in the alleged unit space.} We trivially have that $v\coloneqq r_{\cX}(x)$ is in~$\cX\z$.  Since $r_{\cH}(h)\HrightX x\inv =\rho_{\cX}(v)$ by Lemma~\ref{lem:acting_on_units}, it is an element of $\cH\z$, and 
    \begin{align*}
        \rho_{\cX}(v\inv)
        &=
        \rho_{\cX}(v)=
        r_{\cH}(h)\HrightX x\inv
        =
        r_{\cH}\left(r_{\cH}(h)\HrightX x\inv\right),
    \end{align*}
    which shows that $r_{\cX\bowtie\cH} (x,h)$ is in \mbox{$(\cX\bowtie\cH)\z$}.

    \smallskip\paragraph{\itshape Composability condition.} The elements $(x,h)$ and $(y,k)$ are composable in \mbox{$\cX\bowtie\cH$} if and only if
    \(
        s_{\cX\bowtie\cH} (x,h)
        =
        r_{\cX\bowtie\cH} (y,k)
    \);
    by our definition of the source and range map, that means
    \[  
        h\inv \HleftX s_{\cX}( x)
        =
        r_{\cX}(y)
        \quad\text{and}\quad
        s_{\cH}(h)=
        r_{\cH}(k)\HrightX y\inv.
    \]
    But now notice that the first condition implies the second: 
    \begin{align*}
        r_{\cH}(k)\HrightX y\inv
        &=
        \rho_{\cX}(y)
        &&\text{(by \eqref{item:L7} in Lemma~\ref{lem:acting_on_units})}
        \\
        &
        =\rho_{\cX}\z (h\inv \HleftX s_{\cX}( x))
        &&\text{(by the first condition)}
        \\
        &= r_{\cH}(h\inv)=s_{\cH}(h) 
        &&\text{(by 
            \eqref{item:L1})},
    \end{align*}
    so $(x,h)$ and $(y,k)$ are composable if and only if $h\inv \HleftX s_{\cX}( x)=
        r_{\cX}(y)$, as claimed.
    
    \smallskip\paragraph{\itshape The composition makes sense.}
    By assumption, we have $s_{\cH}(h)=r_{\cH}(k)\HrightX y\inv$. By Lemma~\ref{lem:acting_on_units}~\eqref{item:L7}, the right-hand side is 
        exactly  $\rho_{\cX}(y)$,
so that $  h\HleftX y$ and $h\HrightX y$ are indeed defined.
    We have $r_{\cX}(h\HleftX y) =h\HleftX r_{\cX}( y)$ by \eqref{item:L10} (Lemma~\ref{lem:acting_on_units}); the right-hand side is, by assumption, equal to $h\HleftX [h\inv \HleftX s_{\cX}(x)]$. By~\eqref{item:L5}, that is exactly $s_{\cX}(x)$, so that $x(h\HleftX y)$ is defined.  
    We have 
        $s_{\cH}(h\HrightX y)=\rho_{\cX}(y\inv)$ by~\eqref{item:L1}.
    Since $(y,k)\in\cX\bowtie\cH$, the right-hand side equals $r_{\cH}(k)$, so that $(h\HrightX y)k$ makes sense.
    We have $\rho_{\cX}((x[h\HleftX y])\inv)=\rho_{\cX}((h\HleftX y)\inv)$
    which equals $r_{\cH}(h\HrightX y)=r_{\cH}([h\HrightX y]k)$ by~\eqref{item:L1},
    so the product is an element of \mbox{$\cX\bowtie\cH$}.
\end{remark}

\begin{remark}\label{rm.ssp is LCH gpd} \label{rm.product.groupoid.unit}
    With the algebraic structure from Definition~\ref{def.ZSProduct.groupoid.left} and the subspace topology, \mbox{\mbox{$\cX\bowtie\cH$}} is a \LCH\ groupoid. Indeed, since~$\cX$ and~$\cH$ are both \LCH, and since \mbox{$\cX\bowtie\cH$} is a closed subspace of \mbox{\mbox{$\cX\times\cH$}}, it is clear that \mbox{$\cX\bowtie\cH$} is itself \LCH. Continuity of multiplication and inversion  follow immediately from continuity of $\HleftX$, $\HrightX$, and of multiplication and inversion in~$\cX$ and $\cH$.
\end{remark}

\begin{remark}\label{rm.unit space of ssp}
    Notice that the unit space of the \ssp, 
    $$(\cX\bowtie \cH)\z =\{(u,
    v%e
    ): u\in \cX\z , 
    v%e
    \in\cH\z , \rho_{\cX}(u)=
    v%e
    \},$$
    is homeomorphic to $\cX\z $, since the map $(u,
    v%e
    )\mapsto u$ and its inverse $u\mapsto (u,\rho_{\cX}(u))$ are continuous. Under this identification, we can simply write $r_{\cX\bowtie\cH}(x,h)=r_{\cX}(x)$ and $s_{\cX\bowtie\cH}(x,h)=h\inv \HleftX s_{\cX}(x)$.
\end{remark}

\begin{example}[continuation of Example~\ref{ex:ssla ON trivial gpd}]\label{ex:ssla ON trivial gpd:bowtie}
    \txtrepeat{assumptions of ex:ssla ON trivial gpd}. 
    Then the \ssp\ $\cX\bowtie\cH$ (where $\cX$ acts trivially on $\cH$) is identical to the transformation groupoid $\cX\rtimes \cH$, if we use the convention that $r_{\cX\rtimes \cH}(x,h)=x$ and $s_{\cX\rtimes \cH}(x,h)=h\inv\HleftX x$.
\end{example}

\begin{example}\label{ex:ssla of trivial group and groupoid:bowtie}
    Given a groupoid $\cX$, %\sout{ with open range map}, 
    it is easy to check that the \ssp\  $\cX\bowtie\cX\z$ of~$\cX$ with the trivial groupoid $\cX\z$ (as in Example~\ref{ex:ssla OF trivial gpd}) is isomorphic to $\cX$ via $ (x,s_{\cX}(x))\mapsto x$.
    Likewise, the \ssp\ $\cX\bowtie\{e\}$  of~$\cX$ with the trivial group (as in Example~\ref{ex:ssla of trivial group}) is isomorphic to the groupoid $\cX$ via $(x,e)\mapsto x$.
\end{example}

In~\cite[Section 3]{BPRRW2017}, the construction of the \ZS\ product was only done for \etale\ groupoids. Furthermore, their groupoids were {\em matched}: In addition to the left and right actions, groupoids in a matched pair are assumed to have the same unit space, $\cX\z=\cH\z$,
    and that $\rho_{\cX}\z=\mathrm{id}_{\cX\z}$.
Our above definition of the
    \ssp\
$\cX\bowtie \cH$ does not require~$\cX$ and~$\cH$ to be matched; they  
may have different unit spaces.
However, as pointed out in~\cite[Example 2.4]{DL2021}, we can construct a new transformation groupoid $\widetilde{\cH}$ such that $\widetilde{\cH}$ and~$\cX$ are matched, and such that their
 \ZS\ product $\cX\bowtie\widetilde{\cH}$ 
is isomorphic to the
    \ssp\
\mbox{$\cX\bowtie\cH$}. 
We will now make this more precise.

\begin{lemma}\label{lem.ss.is.matched}
    Suppose a groupoid $\cH$ has a \ssla\ on a groupoid $\cX$, denoted~$\HleftX$ and~$\HrightX$.
By Lemma~\ref{lem:restricted action on cXz}, we get a
left action of~$\cH$ on $\cX\z$ which
gives rise to a transformation groupoid
$\widetilde{\cH}
    =\cH\ltimes \cX\z
$
with unit space $\cX\z$. If we define for $((h,u),x)\in \widetilde{\cH}\bfp{s}{r}\cX$,
\begin{align*}
    \widetilde{\cH}\cart \cX:&&(h,u) \cdot x &
        \coloneqq
    %= 
    h\HleftX x, \\
    \widetilde{\cH}\calb \cX:&&
    (h,u)|_x &
        \coloneqq
    %=
    (h\HrightX x,s_{\cX}(x)),
\end{align*}
then $(\cX,\widetilde{\cH})$ is a matched pair.
\end{lemma}

    Note that the momentum map of $\cX$ for these newly defined actions is not $\rho_{\cX}$ but $r_{\cX}$, as necessary for a matched pair.

\begin{proof}
Recall that $\widetilde{\cH}$ is the set
$\cH\bfp{s}{\rho}\cX\z$ with multiplication and inversion defined by
\begin{align*}
    (k,h\HleftX u)(h,u) = (hk,u)
    &&\text{resp.}&&
    (h,u)\inv  = (h\inv ,h\HleftX u). 
\end{align*}
Its unit space is further identified with $\cX\z$; to be precise, the source of $(h,u)$ is $(h\inv h,u) = (u,u)$, or simply $u$.

Let us
check that the new 
actions are well defined. The actions are only defined for $((h,u),x)$ for which $s_{\widetilde{\cH}}(h,u)=u$ equals $r_{\cX}(x)$. 
Since $(h,u)\in\widetilde{\cH}$, we have 
    $s_{\cH}(h)=\rho_{\cX}\z(u)$, and so $s_{\cH}(h)=\rho_{\cX}(x)$.  
This means that $h\HleftX x$ and $h\HrightX x$ are both defined. Lastly, notice that 
\(
    s_{\cH}(h\HrightX x)
    =
    \rho_{\cX} (x\inv)
    =
    \rho_{\cX} (s_{\cX} (x))
\)
by~\eqref{item:L1},
so that $(h,u)|_x$ is indeed another element of $\widetilde{\cH}$.

The ambitious reader can now verify easily that $(\cX,\widetilde{\cH})$ is a matched pair.
\end{proof}

\begin{proposition}\label{prop.ss.to.matched.left}
    With the assumptions and definitions 
    in
    Lemma~\ref{lem.ss.is.matched}, the
   \ZS\ product $\cX\bowtie\widetilde{\cH}$ of the matched pair is isomorphic to the 
        \ssp\
    $\cX\bowtie \cH$ in the sense of Definition~\ref{def.ZSProduct.groupoid.left}. 
\end{proposition}

\begin{proof}
   By definition of $\cX\bowtie\widetilde{\cH}$, any of its elements $(x, (h,u))$ satisfies $s_{\cX}(x)=r_{\widetilde{\cH}}(h,u)$, which is exactly $h\HleftX u$ by definition of 
   the range map of
   $\widetilde{\cH}$. Thus, $u=h\inv \HleftX s_{\cX}(x)$. Moreover, 
   $\rho_{\cX}(x\inv)
   =\rho_{\cX}(h\HleftX u)=r_{\cH}(h)$ 
   by 
   \eqref{item:L1},
   Lemma~\ref{lem:acting_on_units}~\eqref{item:L8},
   which shows that $(x,h)$ is an element of \mbox{$\cX\bowtie\cH$}. All in all, the maps
   \[
        \varphi\colon
        \cX\bowtie\widetilde{\cH}
        \to 
        \cX\bowtie\cH,
        \quad
        (x,(h,u))\mapsto (x,h),
   \]
   and
   \[
        \cX\bowtie\cH
        \to 
        \cX\bowtie\widetilde{\cH},
        \quad
        (x,h)\mapsto \bigl(x, (h, h\inv \HleftX s_{\cX}(x)
        )
        \bigr),
   \]
   are well defined and mutually inverse.
   Since they are constructed out of continuous maps, they are themselves continuous.	
   Lastly, notice that $\varphi$ is a groupoid homomorphism:
   \begin{align*}
       \varphi
       \bigl( (x,(h,u))\,(y,(k,v))\bigr)
       &=
       \varphi
       \bigl(
        x[(h,u)\cdot y],
        (h,u)|_{y} (k,v)
       \bigr)
       && \text{(def'n of $\cX\bowtie\widetilde{\cH}$)}
       \\&
       =
       \varphi
       \bigl(
        x[h\HleftX y],
        (h\HrightX y,s_{\cX}(y))(k,v)
       \bigr)
       && \text{(def'n of $\cdot$ and $|$)}
       \\
       &=
       \varphi
       \bigl(
        x[h\HleftX y],
        ([h\HrightX y] k,v)
       \bigr)
       && \text{(def'n of $\widetilde{\cH}$)}
       \\
       &=
       \bigl(
        x[h\HleftX y],
        [h\HrightX y] k
       \bigr)
       && \text{(def'n of $\varphi$)}
       \\&=(x,h)
       \,(y,k)
       && \text{(def'n of $\cX\bowtie \cH$)}
       \\&
       =
       \varphi
       ( x,(h,u))
       \,\varphi
       (y,(k,v)).
   \end{align*}
    This proves that $\cX\bowtie\widetilde{\cH}$ is isomorphic to \mbox{$\cX\bowtie\cH$}. 
\end{proof}

\begin{example}[cf.\ {\cite[Section 5.3]{BPRRW2017},
\cite[Definition 3.6]{Deaconu}
}]\label{ex:skew product}
    Suppose $\cG$ is a \LCH\ groupoid and $H$ is a group (neither are assumed to be \etale), and $\mathbf{c}\colon \cG\to H$ is a continuous homomorphism. The {\em skew-product groupoid $\cG(\mathbf{c})$} is the set $\cG\times H$ with the operations given for $(g,g')\in\cG^{(2)}$ and $h\in H$ by
    \[
        (g, h )(g', h \mathbf{c}(g))=(gg', h )
        \quad\text{and}\quad
        (g, h )\inv = (g\inv, h\mathbf{c}(g)  ).
    \]
    Note that $\cG(\mathbf{c})\z=\cG\z\times H$. 
    The formula
    \(
        \varphi_{h}(g, h')\coloneqq(g, h' h\inv)
    \)
    defines a continuous, free action of $H$ on $\cG(\mathbf{c})$ by automorphisms. See \cite[Section 4]{KWR2001:Skew} 
    for more details, but note that their convention for $\cG(\mathbf{c})$ is slightly different from ours.

    In the case where $\cG$ and $H$ are \etale, \cite[Proposition 22]{BPRRW2017} states that the above action induces a left $H$-action on $\cG\z\times H$ and that the corresponding transformation groupoid $$\widetilde{H}\coloneqq H\tensor*[_{\varphi}]{\ltimes}{} \cG(\mathbf{c})\z$$ allows a \ZS\ product with $\cG(\mathbf{c})$. It was pointed out further that this product $\cG(\mathbf{c})\bowtie \widetilde{H}$ {\em ``should be considered as the \ZS\ product of the groupoid $\cG(\mathbf{c})$ with the group $H$''}, since the space $\cG(\mathbf{c})\times H$  is homeomorphic to $\cG(\mathbf{c})\bowtie \widetilde{H}$ via $\bigl((g, h ),h'\bigr)\mapsto \bigl((g, h ),(h',s(g), h \mathbf{c}(g)h')\bigr)$. 
    
    Using our machinery above, this comment can be made  concrete without the need to go via the transformation groupoid $\widetilde{H}$ (and without assuming \etale):
    Since $H\z=\{e\}$, the balanced fiber product $\bfp{s
        }{\rho
        }$ just becomes the Cartesian product,
    and we can define
    \begin{align*}
        H\cart  \cG(\mathbf{c})%\cX
    \colon &&
    % H \times \cG(\mathbf{c})%\cX
    % \ni (h,(g,h')) &\mapsto 
    h\HleftX (g,h') &\coloneqq (g,h'h\inv)
    %\in  \cG(\mathbf{c})%\cX
     \\
        H\calb  \cG(\mathbf{c})%\cX
    \colon &&
    %H \times \cG(\mathbf{c})%\cX
    %\ni (h,(g,h')) &\mapsto 
    h\HrightX (g,h')  &\coloneqq  \mathbf{c}(g)\inv h\mathbf{c}(g)
    %\in H
    \end{align*}

    One  verifies that these give a \ssla\ of $H$ on $\cG(\mathbf{c})$, and so we may construct the \ssp\ $\cG(\mathbf{c})\bowtie H$ as in Definition~\ref{def.ZSProduct.groupoid.left}. By Proposition~\ref{prop.ss.to.matched.left}, $\cG(\mathbf{c})\bowtie H$ is isomorphic to  the \ZS\ product groupoid $\cG(\mathbf{c})\bowtie \widetilde{H}$ from \cite[Proposition 22]{BPRRW2017}.
\end{example}

\begin{remark}
    As the last example highlights, the main distinction between the (old) \ZS\ product and our (new) \ssp\ is that the latter does not require the groupoids with two-way actions to have matching unit spaces. For \ZS\ products, there is no inherent distinction between the roles of the two groupoids $\cH$ and $\cX$ (everything is entirely symmetric), while the \ss-variant makes a clear distinction between them: Besides its range and source maps, the groupoid $\cX$ must also carry a separate momentum map $\rho_{\cX}\colon \cX\to \cH\z$ with respect to which the $\cH$-action is defined. After Proposition~\ref{prop:Htilde}, it is natural to ask whether this added layer of difficulty in Definition~\ref{df.left.selfsimilar} is worth the effort. But  while the \ssp\ $\cX\bowtie\cH$ and the \ZS\ product $\cX\bowtie\widetilde{\cH}$ are isomorphic, there are fundamental differences between the pair $(\cX,\cH)$ and the pair $(\cX,\widetilde{\cH})$, as we will see in Example~\ref{ex:ssla of trivial gp and gpd:sometimes a something} and its subsequent remark.
\end{remark}

\begin{example}[reconciliation]\label{ex:reconciliation}
    Suppose $\cH=\{e\}$ has the trivial \ssla\ on a groupoid $\cX$ (Example~\ref{ex:ssla of trivial group}). The induced action $\cdot$ of the transformation groupoid $\widetilde{\{e\}}=\{e\}\ltimes \cX\z$ on $\cX$ as defined in Lemma~\ref{lem.ss.is.matched} is then likewise trivial, and the induced action $|$ of $\cX$ on $\widetilde{\{e\}}$ is given for $x\in\cX$ and $(e,u)\in\widetilde{\{e\}}$ by
    \[
        (e, u)|_{x}\coloneqq (e, s_{\cX}(x)) \quad\text{ where } \quad u=s_{\widetilde{\cH}}(e,u)=r_{\cX}(x).
    \]
    In other words: If we identify an element $(e,u)$ of $\widetilde{\{e\}}$ with $u$ in $\cX\z$, then 
    the \ssla\ of $\widetilde{\{e\}}$ on $\cX$ that we described in Lemma~\ref{lem.ss.is.matched} is identical to the one of $\cX\z$ on $\cX$ that we described in Example~\ref{ex:ssla OF trivial gpd}. Under this identification, the concatenation of the isomorphisms $\cX\bowtie\{e\}\cong \cX$ and $\cX\cong \cX\bowtie\cX\z$ in Example~\ref{ex:ssla of trivial group and groupoid:bowtie} yields exactly the isomorphism $\cX\bowtie\{e\}\cong\cX\bowtie \widetilde{\{e\}} $ in Proposition~\ref{prop.ss.to.matched.left}.
\end{example}

One can define an analogous notion of a \ss\ action on the right. For the convenience of the reader and to establish notation, we will repeat the main properties in Subsection~\ref{ssec.right ssa}.

\subsection{Haar systems for \ssla s}

\begin{definition}\label{df.lambda H-invariance}
    Suppose $\cH$ and $\cX$ are 
    groupoids and that $\HleftX$ is a left $\cH$-action on $\cX$ with momentum map $
        \rho_{\cX}=\rho_{\cX}\z\circ r_{\cX}
    \colon \cX\to\cH\z$.
    We say that a left Haar system $\{\lambda^{u}\}_{u\in \cX\z }$ on~$\cX$ is {\em $\HleftX$-invariant} if for all $h\in\cH$ and all $u\in \cX\z $ with $s_{\cH}(h)=\rho_{\cX}(u)$, we have
    \[h\HleftX \lambda^{u}=\lambda^{h\HleftX u},\] where $(h\HleftX \lambda^{u}) (E) = \lambda^{u}(h\inv\HleftX E)$. Equivalently, 
    for all $f\in C_c(\cX)$, \begin{equation}\label{eq:HleftX-invariance in integral form}
        \int f(h\HleftX x) \dif\lambda^{u}(x) = \int f(y) \dif\lambda^{h\HleftX u}(y).
    \end{equation}
\end{definition}

\begin{proposition}[cf.\ {\cite[Proposition 6.4]{KMQW2010}}
]\label{prop.ZSProduct.Haar.left} 
    Suppose $\cH$ and $\cX$ are \LCH\ groupoids, that $\cH$ has a 
    \ssla\ on $\cX$, and that $\cX$ has a $\HleftX$-invariant left Haar system~$\lambda$. 
    If $\varepsilon$ is any left Haar system for~$\cH$, then
    we get a left Haar system $\lambda\bowtie\varepsilon$ for \mbox{$\cX\bowtie\cH$}
    defined for $u\in\cX\z$ by
    \[
        \dif (\lambda\bowtie\varepsilon)^{u}
        (y,k)= \dif
        \varepsilon^{\rho(y\inv)}(k)
        \dif\lambda^{u}(y) 
        .
    \]
    Equivalently, 
    for any $f\in C_c(\cX\bowtie\cH)$, 
    \[
    \int f(y,k) \dif (\lambda\bowtie\varepsilon)^{u}
    (y,k) = \int_{\cX} \int_{\cH} f(y,k)  \dif 
        \varepsilon^{\rho(y\inv)}(k)
        \dif\lambda^{u}(y).
    \]
\end{proposition}

In the above, we have used the fact that $(\cX\bowtie\cH)\z\approx \cX\z$ by Remark~\ref{rm.unit space of ssp}.
To prove the above proposition, we need the following:

\begin{lemma}
\label{lem:x2h2}
    Suppose $u,v\in\cX\z$ and $(x,h)\in \cX^{u}_v\times\cH^{\rho(v)}\subseteq \cX\bowtie\cH$ are fixed. If we let
    \(
    	h_{2}=(h\inv \HrightX x\inv)\inv
		\) and \(
		x_{2}=h_{2}\inv\HleftX x,
    \)
    then 
        $h \HleftX s_{\cX}(x_{2})= v$,
    and for all $y\in \cX^{h\inv \HleftX v}$, we have
    $x(h\HleftX y)=h_{2}\HleftX(x_{2}y)$.
\end{lemma}

\begin{proof}
    We compute
        \begin{align}\label{eq:x_2}
        	x_{2}
        	&=
        	(h\inv \HrightX x\inv)\HleftX x
        	\overset{\eqref{item:L9}}{=}
        	(h\inv\HleftX x\inv)\inv,
        \end{align}
        so that
        \begin{align*}
          s_{\cX}(x_{2}) &
          = r_{\cX}(h\inv \HleftX x\inv )
          \overset{\eqref{item:L10}}{=} h\inv \HleftX r_{\cX}(x\inv )
          = h\inv \HleftX s_{\cX}(x) = h\inv \HleftX v,
        \end{align*}
        as claimed.
        By Equation~\eqref{eq:x_2},
        \begin{align*}
        	h_{2}
        	&\overset{\eqref{item:L9}}{=}
        	h  \HrightX (h\inv \HleftX x\inv)
        	=
        	h\HrightX x_{2}\inv, 
        	\quad\text{ so that }\quad
        	h_{2}\HrightX x_{2}
         \overset{\eqref{item:L3}}{=} h.
        \end{align*}
        Now, if $y$ is such that $r_{\cX}(y)=h\inv \HleftX v$, meaning that $x_{2}y$ makes sense by our above computation, then
        \[
            \rho_{\cX}(x_{2}y)
            =
            \rho_{\cX}(x_{2})
            \overset{\eqref{item:L1}}{=} 
            s_{\cH} (h\HrightX x_{2}\inv)
            =
            s_{\cH}(h_{2})
            .
        \]
        Therefore, $h_{2}\HleftX(x_{2}y)$ is likewise defined, and we have:
        \begin{align*}
        	h_{2}\HleftX(x_{2}y)
        	&=
        	(h_{2}\HleftX x_{2})[(h_{2}\HrightX x_{2})\HleftX y]
        	&&\text{(by \eqref{item:L4})}
        	\\
        	&=
        	x[h\HleftX y]
	       	&&\text{(def'n of $h_{2}$  and by the above)}.\qedhere
        \end{align*}
\end{proof}

\begin{corollary}
\label{cor:left-invariance and HleftX}
    Suppose  $u,v\in\cX\z$, $(x,h)\in \cX^{u}_v\times\cH^{\rho(v)}\subseteq \cX\bowtie\cH$, and $\lambda$
    is a $\HleftX$-invariant left Haar system for~$\cX$ in the sense of Definition~\ref{df.lambda H-invariance}. If $G\in C_{c} (\cX)$, then
    \[    \int_{\cX}
            G\big(x[h\HleftX y]\big)
            \dif\lambda^{h\inv \HleftX v}(y)
            =
            \int_{\cX}
            G(y)
            \dif\lambda^{u}(y).
    \]
\end{corollary}

\begin{proof}
Let $x_{2},h_{2}$ be as in Lemma~\ref{lem:x2h2}. Then
\begin{align*}
    \int_{\cX}
    G\big(x[h\HleftX y]\big)
    \dif\lambda^{h\inv \HleftX v}(y)
    &=
    \int_{\cX}
    G\big(h_{2}\HleftX[x_{2}y]\big)
    \dif\lambda^{s(x_{2})}(y)
    .
\intertext{By left invariance of $\lambda$, we have}
    \int_{\cX}
    G\big(h_{2}\HleftX[x_{2}y]\big)
    \dif\lambda^{s%_{\cX}
    (x_{2})}(y)
    &=
    \int_{\cX}
    G(h_{2}\HleftX z )
    \dif\lambda^{r%_{\cX}
    (x_{2})}(z).
\intertext{
Since $r_{\cX}(x_{2})= s_{\cH}(h_{2})$, we can invoke $\HleftX$-invariance of $\lambda$ in the form of Equation~\eqref{eq:HleftX-invariance in integral form} 
to conclude}
    \int_{\cX}
    G(h_{2}\HleftX z )
    \dif\lambda^{r%_{\cX}
    (x_{2})}(z)
    &=
    \int_{\cX}
    G(y)
    \dif\lambda^{h_{2}\HleftX r%_{\cX}
    (x_{2})}(y)
    .
\end{align*}
Since $x_{2}=h_{2}\inv \HleftX x$, it follows from \eqref{item:L10} (Lemma~\ref{lem:acting_on_units}) that $h_{2}\HleftX r_{\cX}(x_{2})=r_{\cX}(x)=u$, so that the above right-hand side is as claimed in the statement.
\end{proof}

\begin{proof}[Proof of Proposition~\ref{prop.ZSProduct.Haar.left}]
   For this proof, let $\rho\coloneqq \rho_{\cX} = \rho_{\cX}\z\circ r_{\cX}$ and $\rho'\coloneqq  \rho_{\cX}\z\circ s_{\cX}$.
   Fix an arbitrary $u\in\cX\z$ and note that $(\lambda\bowtie\varepsilon)^{u}$
   is a Radon measure on $\cX\bowtie \cH$, since
   \[
   (\lambda\bowtie\varepsilon)^{u}
   \colon\quad
        C_c (\cX\bowtie\cH)\to \mathbb{C},
        \quad
        F \mapsto 
        \int_{\cX}
        \int_{\cH}
            F(y,k)
        \dif 
            \varepsilon^{\rho'(y)}(k)
        \dif\lambda^{u}(y),
   \]
   is clearly a positive linear functional on $C_c(\cX\bowtie\cH)$.
   First, we show that  $\supp \,(\lambda\bowtie\varepsilon)^{u}
   =
            (\cX\bowtie\cH)^{u}$. To see~$\supseteq$,  fix any $$\eta=(y,k)\in (\cX\bowtie\cH)^{u}
            =
            \cX^{u}\bowtie\cH
            =
            \sqcup_{v\in\cX\z} \cX^{u}_{v}\times\cH^{ \rho (v)}.$$ For any open neighborhood $N_{\eta}$ around $\eta$, we must show that $(\lambda\bowtie\varepsilon)^{u}
            (N_{\eta})>0$. By monotonicity, 
            it suffices to show this for a {\em basic} open neighborhood, so we may assume that $N_{\eta}=  (N_{y}\times N_{k})\cap \cX\bowtie\cH$ for some 
            neighborhoods $N_{y}$ of $y$ and $N_{k}$ of $k$.
        Thus,
        \begin{align}
            (\lambda\bowtie\varepsilon)^{u}
            (N_{\eta})
            &=
            \int_{\cX\bowtie\cH} 1_{N_{\eta}}(\xi) \dif
                (\lambda\bowtie\varepsilon)
           ^{u}(\xi)
            =
            \int_{\cX}\int_{\cH} 1_{N_{\eta}}(x,h) \dif
            \varepsilon^{\rho'(x)}(h)
            \,\dif\lambda^{u}(x)
            \notag
            \\
            &
            =
            \int_{\cX}\int_{\cH}
                1_{N_{y}}(x)\, 1_{N_{k}}(h)\, 1_{\cX\bowtie\cH} (x,h)
            \dif
            \varepsilon^{\rho'(x)}(h)
            \,\dif\lambda^{u}(x)
            \notag
            \\
            &
            =
            \int_{\cX}1_{N_{y}}(x)\, \left[\int_{\cH}
                1_{N_{k}}(h)
            \dif
            \varepsilon^{\rho'(x)}(h)
        \right]
            \,\dif\lambda^{u}(x)
            .
            \label{eq:measure of N_eta}
        \end{align}

        \assumption{Since~$\cH$ is locally compact}, we may find a precompact neighborhood $M_{k}$ of $k$ for which $\overline{M_{k}}\subseteq N_{k}$. Since $k\in \cH^{ \rho (v)}=\supp \varepsilon^{ \rho (v)}$, we have $\delta\coloneqq \varepsilon^{ \rho (v)} (M_{k})>0$. Let $f\in C_{c}(\cH,[0,1])$ be a function that is constant $1$ on $M_{k}$ and vanishes outside of $N_{k}$, so that for all $w\in \cH\z$,
         \begin{equation}\label{eq:integral-Nk}
            \int_{\cH} 1_{N_{k}}(h) \dif\varepsilon^{w}(h)
            \geq
            \int_{\cH} f(h) \,\dif\varepsilon^w (h)
            \geq
            \int_{\cH} 1_{M_{k}}(h) \dif\varepsilon^{w}(h)
            .
        \end{equation}
        Note that the middle term is exactly $\varepsilon^w (f)$.
        As $\varepsilon$ is a Haar system for~$\cH$, the function $$\varepsilon (f)\colon \cH\z\to\mathbb{C}, w\mapsto \varepsilon^w (f),$$  is continuous,
        where we followed the notation used in~\cite[Remark 1.20]{Wil2019}.
        As the right-most side of \eqref{eq:integral-Nk} equals $\delta$ for $w= \rho (v)$, continuity of $\varepsilon(f)$ implies that $\varepsilon(f)$ is greater than $\frac{\delta}{2}$ in a neighborhood $U$ of $ \rho (v)$; let $V\coloneqq 
            (\rho')\inv 
        (U)\subseteq \cX$. 
        
        Using our computation in \eqref{eq:measure of N_eta}, we see that
        \begin{align*}
            (\lambda\bowtie\varepsilon)^{u}
            (N_{\eta})
            &
            \geq
            \int_{\cX}1_{N_{y}}(x)\, \varepsilon(f)\bigl(
                \rho'(x)
            \bigr)
            \,\dif\lambda^{u}(x)
            \\
            &
            \geq
            \int_{\cX}1_{N_{y}\cap V}(x)\, \varepsilon(f)\bigl(
                \rho'(x)
            \bigr)
            \,\dif\lambda^{u}(x)
            \\
            &
            \geq
            \delta
            \,
            \int_{\cX}1_{N_{y}\cap V}(x)\, \dif\lambda^{u}(x)
            =
            \delta\, \lambda^{u} (N_{y}\cap V).
        \end{align*}
        Note that 
        by choice of $y$, $\rho'(y) = \rho_{\cX}\z(s_{\cX}(y))=\rho_{\cX}\z(v)$ is an element of $U$,
        so $N_{y}\cap V$ is a neighborhood of $y$.  Since  $y\in \cX^{u}=\supp \lambda^{u}$, we must have $\lambda^{u} (N_y\cap V)>0$, and hence $(\lambda\bowtie\varepsilon)^{u}
        (N_{\eta})>0$. Since $y,k, u,v$ were arbitrary, this proves that $\supp \,(\lambda\bowtie\varepsilon)^{u}
        \supseteq \sqcup_{v\in\cX\z} \cX^{u}_{v}\times\cH^{ \rho (v)}$.
        
        \smallskip
        Conversely, assume that $\eta\notin  (\cX\bowtie\cH)^{u}$, i.e., if we write $\eta=(y,k)$, then $r_{\cX}(y)\neq u$. Consider $r_{\cX}\inv (\cX\z\setminus \{u\})=\cX\setminus\cX^{u}$. 
        Since $\cX\z$ is Hausdorff, this is an open neighborhood around $y$. 
        Since $\supp\lambda^{u}=\cX^{u}$, we have $\lambda^{u} (\cX\setminus\cX^{u})=0$. In particular, if we let $N_{\eta}\coloneqq (\cX\setminus\cX^{u})\bowtie \cH$, then we have found a neighborhood of $\eta$ for which $(\lambda\bowtie\varepsilon)^{u}(N_{\eta})=0$. Indeed, using our computation in \eqref{eq:measure of N_eta}, we see that
        \begin{align*}
            (\lambda\bowtie\varepsilon)^{u}
            (N_{\eta})
            =
            \int_{\cX}1_{\cX\setminus\cX^{u}}(x)\, \left[\int_{\cH}
                1_{\cH}(h)
            \dif
            \varepsilon^{\rho'(x)}(h)
            \right]
            \,\dif\lambda^{u}(x)
            =
            0.
        \end{align*}
        This means that $\eta\notin \supp \,(\lambda\bowtie\varepsilon)^{u}$, as claimed.
        
        \medskip
        
        Next, for $F\in C_c (\cX\bowtie\cH)$, we need to show that the map $u\mapsto \int F \dif (\lambda\bowtie\varepsilon)^{u}
        $ is continuous.
        We will first prove the claim for $F=(f\times g)|_{\cX\bowtie\cH}$, where $f \times  g \colon (x,h)\mapsto  f (x) g (h)$ for some $ f \in C_{c}(\cX)$ and $ g \in C_{c}(\cH)$, so that
        \begin{align*}
	        \int_{\cX\bowtie\cH} F(\eta) \dif (\lambda\bowtie\varepsilon)^{u}
	        (\eta)
	        &=
	        \int_{\cX}
	         f (y)\,\int_{\cH}  g (k)
	        \dif
            \varepsilon^{\rho'(y)}(k)
	        \,\dif\lambda^{u}(y).
        \end{align*}
        Since $\varepsilon$ is a Haar system on~$\cH$ and since $ g \in C_c(\cH)$, we know that the function
        \[
        	\cH\z \to \mathbb{C}, \quad u'\mapsto \int_{\cH}  g (k)  \dif \varepsilon^{u'}(k),
        \]
        is continuous. Since $ f \in C_{c}(\cX)$ and since 
         $ \rho'=\rho_{\cX}\z \circ s_{\cX}$ 
        is continuous, it follows that
        \[
        	G\colon \cX\to \mathbb{C},
        	\quad 
        	y\mapsto  f (y)\left(\int_{\cH}  g (k)
        		        \dif 
            \varepsilon^{\rho'(y)}(k)
            \right),
        \]
        is continuous and compactly supported. Since $\lambda$ is a Haar system on~$\cX$, we thus know that 
        \[
        	(\cX\bowtie\cH)\z \cong \cX\z\to \mathbb{C},\quad
        	u\mapsto \int_{\cX}
        		        G(y)
        		        \,\dif\lambda^{u}(y) = \int_{\cX\bowtie\cH} F(\eta) \dif (\lambda\bowtie\varepsilon)^{u}
	        (\eta),
        \]
        is continuous, as needed.

        For general $F\in C_{c}(\cX\bowtie\cH)$, let $K_{\cX}$ and $K_{\cH}$ be the  $\cX$- resp.\ the $\cH$-part of $\supp(F)$, both of which are compact.
        Pick $f\in C_{c}(\cX)$ and $g\in C_{c}(\cH)$ which are constant 1 on $K_{\cX}$ resp.\ $K_{\cH}$, so that for any $v\in\cX\z$
        and for $K_{\cX}\bowtie K_{\cH}\coloneqq (K_{\cX}\times K_{\cH})\cap \cX\bowtie \cH$,
        \[
            (\lambda\bowtie\varepsilon)^{v} (\supp(F))
            \leq 
            (\lambda\bowtie\varepsilon)^{v} (K_{\cX}\bowtie K_{\cH})
            \leq
            \int_{\cX\bowtie\cH} (f\times g) \dif (\lambda\bowtie\varepsilon)^{v}
            .
        \]
       By our earlier argument, the right-hand side is a continuous function in $v$. Therefore, if $K$ is some compact set, then for any $v\in K$,
       \begin{equation}\label{eq:lambda bowtie vareps of suppF}
            (\lambda\bowtie\varepsilon)^{v} (K_{\cX}\bowtie K_{\cH})
            \leq
            \max_{v'\in K}
            \bigl[\int_{\cX\bowtie\cH} (f\times g) \dif (\lambda\bowtie\varepsilon)^{v'}\bigr]
            \rotatebox[origin=c]{180}{$\coloneqq$}
            c_{K}<\infty.
       \end{equation}
       Now, assume we are given a convergent net $u_{i}\to u$ in $\cX\z$ and fix an arbitrary $\epsilon>0$. By \assumption{local compactness of $\cX$}, we may without loss of generality assume that each $u_{i}$ is contained in a compact neighborhood $K$ of $U$, so that \eqref{eq:lambda bowtie vareps of suppF} holds for $v=u_{i}$.
       By Stone--Weierstrass, we can choose finitely  many $f_{j}\in C_{c}(\cX),g_{j}\in C_{c}(\cH)$ such that
            \[
            \norm{F - \sum\nolimits_{j=1}^{k} (f_{j}\times g_{j})|_{\cX\bowtie\cH}}_{\infty}
            < 
            \epsilon/ \big( 3 c_{K} + 1 \big).
            \]
        Without loss of generality, the support of each $f_{j}$ is in $K_{\cX}$ and of each $g_{j}$ is in $K_{\cH}$, so that for all $v\in K$,
        \begin{align}
            &\int 
                \big|F - \sum\nolimits_{j} f_{j}\times g_{j}\big|
            \dif (\lambda\bowtie\varepsilon)^{v}
            \notag
            \\
            \leq&
            (\lambda\bowtie\varepsilon)^{v}
            \bigl(K_{\cX}\bowtie K_{\cH}\bigr)
            \, 
            \norm{F - \sum\nolimits_{j=1}^{k} (f_{j}\times g_{j})|_{\cX\bowtie\cH}}_{\infty}
            \overset{\eqref{eq:lambda bowtie vareps of suppF}}{<} 
            \epsilon/3
            .
            \label{consequence of eq:lambda bowtie vareps of suppF}
        \end{align}
        By our earlier result, we may choose $i_0$ large enough such that for all $i\geq i_0$ and all $1\leq j\leq k$, we have
            \[
                \abs{
            \int f_{j}\times g_{j} \dif (\lambda\bowtie\varepsilon)^{u_{i}}
            -
            \int f_{j}\times g_{j} \dif (\lambda\bowtie\varepsilon)^{u}
            }
            < \epsilon/3k.
            \]
        Combining this with \eqref{consequence of eq:lambda bowtie vareps of suppF}, we get for all $i\geq i_0$ that
        \begin{align*}
            &\abs{
            \int F \dif (\lambda\bowtie\varepsilon)^{u_{i}}
            -
            \int F \dif (\lambda\bowtie\varepsilon)^{u}
            }\notag
            \\%
            %%%%%%%%%%%%%%%%%%%%
            &\quad
            \leq
            \int 
                \big|F - \sum\nolimits_{j} f_{j}\times g_{j}\big|
            \dif (\lambda\bowtie\varepsilon)^{u_{i}}
            \notag
            \\
            &\qquad+
            \sum\nolimits_{j}
            \abs{
            \int f_{j}\times g_{j} \dif (\lambda\bowtie\varepsilon)^{u_{i}}
            -
            \int f_{j}\times g_{j} \dif (\lambda\bowtie\varepsilon)^{u}
            }\notag
            \\
            &\qquad+
            \int
                \big|\bigl(\sum\nolimits_{j} f_{j}\times g_{j}\bigr) - F\big|
            \dif (\lambda\bowtie\varepsilon)^{u}
            <\epsilon
            ,
        \end{align*}
            as needed.

        \medskip
        
        Lastly, we have to show that for any $\xi\in \cX\bowtie\cH$ and any $F\in C_c (\cX\bowtie\cH)$, we have $\int F(\xi\eta) \dif
        (\lambda\bowtie\varepsilon)^{s(\xi)}
        \eta = \int F(\eta) \dif  (\lambda\bowtie\varepsilon)^{r(\xi)}
	        \eta$. Write $\xi=(x,h)\in \cX^{u}_v\times\cH^{ \rho (v)}$, so that $s(\xi)=h\inv \HleftX v$; as above, it suffices to consider the case where 
	        $F$ can be written as $F(\xi)= f (x) g (h)$ for some $ f \in C_c (\cX)$ and some $ g \in C_{c}(\cH)$.
        Then 
        \begin{align*}
            \int_{\cX\bowtie\cH} F(\xi\eta) \dif
        (\lambda\bowtie\varepsilon)^{s(\xi)}
        (\eta)
            &=
            \int_{\cX\bowtie\cH} F\big(x[h\HleftX y],[h\HrightX y]k\big) \dif
            (\lambda\bowtie\varepsilon)^{h\inv \HleftX v}
            (y,k)
            \\
            &=
            \int_{\cX}\int_{\cH}
             f \big(x[h\HleftX y]\big)\, g \big([h\HrightX y]k\big)
            \dif 
            \varepsilon^{\rho'(y)}(k)
            \,
            \dif\lambda^{h\inv \HleftX v}(y)
            \\
            &=
            \int_{\cX}
             f \big(x[h\HleftX y]\big)
            \,\int_{\cH}
             g \big([h\HrightX y]k\big)
            \dif 
            \varepsilon^{\rho'(y)}(k)
            \,
            \dif\lambda^{h\inv \HleftX v}(y)
            ,
        \end{align*}
        where the last equation follows from \eqref{item:L1}, which guarantees that       
        \(
             \rho' (y)
            =
            s_{\cH}(h\HrightX y).
        \)
        Since 
        \(
            r_{\cH}(h\HrightX y)
            =
            \rho'(h\HleftX y)
            ,
        \)
        left-invariance of $\varepsilon$ implies
        \begin{align*}
            \int_{\cX\bowtie\cH} F(\xi\eta) \dif
        (\lambda\bowtie\varepsilon)^{s(\xi)}
            (\eta)
            &=
            \int_{\cX}
             f \big(x[h\HleftX y]\big)
            \,\int_{\cH}
             g (k)
            \dif  
            \varepsilon^{\rho'(h\HleftX y)}(k)
            \,
            \dif\lambda^{h\inv \HleftX v}(y)
            .
        \end{align*}
        For $z\in \cX$, define
        \[
            G(z)
            \coloneqq 
             f (z)
            \,\int_{\cH}
             g (k)
            \dif   
            \varepsilon^{\rho'(z)}(k)
            .
        \]
        Since $\varepsilon$ is a Haar system and  since $ g \in C_c (\cH)$, we know that
                $$
                	\cH\z \to 
                	\mathbb{C},
                	\quad
                	u'\mapsto \int_{\cH}  g  \dif \varepsilon^{u'},$$
        is continuous. Since 
         $ \rho'=\rho_{\cX}\z \circ s_{\cX}$ 
        is continuous and since $ f \in C_c(\cX)$, we conclude that $G$ is a continuous and compactly supported function on~$\cX$. 
        Since $s_{\cX}(h\HleftX y)= s_{\cX}(x[h\HleftX y])$, we conclude that
        \begin{align*}\label{eq:int F(xieta)}
            \int_{\cX\bowtie\cH} F(\xi\eta) \dif
        (\lambda\bowtie\varepsilon)^{s(\xi)}
            (\eta)
            &=
            \int_{\cX}
            G\big(x[h\HleftX y]\big)
            \dif\lambda^{h\inv \HleftX v}(y)
            \\
            &=
            \int_{\cX}
            G(y)
            \dif\lambda^{r%_{\cX}
            (x)}(y)
            &\text{(Corollary~\ref{cor:left-invariance and HleftX})}
            \\
            &=
            \int_{\cX}
             f (y)
            \,\int_{\cH}
             g (k)
            \dif   
            \varepsilon^{\rho'(y)}(k)
            \dif\lambda^{r(\xi)}(y)
            &\text{(def'n of $G$)}
            \\
            &=
            \int_{\cX\bowtie\cH} F(\eta) \dif
            (\lambda\bowtie\varepsilon)^{r(\xi)}
            (\eta).
            \qedhere
        \end{align*}
\end{proof}

\begin{corollary}\label{cor:ssp gpd:r-discrete, etale}
    Suppose $\cH$ and $\cX$ are \LCH\ groupoids and that $\cH$ has a  \ssla\ on $\cX$. 
    \begin{enumerate}
        \item 
        If $\cX$ is \etale, then counting measure on $\cX$ is $\HleftX$-invariant in the sense of Definition~\ref{df.lambda H-invariance}.
        \item If $\cH$ and $\cX$ are both $r$-discrete, then so is $\cH\bowtie\cX$.
        \item  If $\cH$ and $\cX$ are both \etale, then so is $\cH\bowtie\cX$.
    \end{enumerate}
    
\end{corollary}

\begin{proof}
        If $\cX$ is
    \etale, \cite[Prop.\ 1.29]{Wil2019} says that counting measures form 
        a Haar system on $\cX$.
    Now, for any fixed $(h,u)\in \cH\bfp{s}{\rho}\cX\z$, the map $\cX^{h\HleftX u}\to \cX^{u}$, $y\mapsto h\inv\HleftX y$, is a bijection (in fact, a homeomorphism), and thus
    \[
        \sum_{x\in\cX^{u}} f(h\HleftX x)  = \sum_{y\in \cX^{h\HleftX u}} f(y)
    \]
    for all $f\in C_c (\cX)$. In other words, counting measure on $\cX$ is $\HleftX$-invariant.

    Now suppose the groupoids are $r$-discrete. Since $\cX\z\times\cH\z$ is open in $\cX\times\cH$ and since $\cX\bowtie\cH$ has the subspace topology, we have that $(\cX\z\times\cH\z)\cap (\cX\bowtie\cH) = (\cX\bowtie\cH)\z$ is open in $\cX\bowtie \cH$. Thus, $\cX\bowtie\cH$ is also $r$-discrete.

    Now, if 
        both $\cX$ and $\cH$ are \etale, then it follows from
    Proposition~\ref{prop.ZSProduct.Haar.left}  that $\cH\bowtie\cX$ admits a Haar system. According to \cite[Prop.\ 1.23 and 1.29]{Wil2019}, any locally compact and $r$-discrete groupoid that admits a Haar system is necessarily \etale, so our claim follows. 
\end{proof}

\subsection{Rehash (from left to right)}\label{ssec.right ssa}

The definitions we made so far can similarly be made on the right; we have added them here for easy reference.

\begin{definition}[cf.\ Definition~\ref{df.left.selfsimilar}]\label{df.right.selfsimilar}\label{df.G-X-actions}
Let~$\cG$ and~$\cX$ be two \LCH\ groupoids. We say~$\cG$ has a {\em \ssra} on~$\cX$ if there exists 
    a continuous surjection $\sigma_{\cX}\z\colon \cX\z\to\cG\z$ and, using the anchor map $\sigma_{\cX}\coloneqq \sigma_{\cX}\z\circ s_{\cX}$, 
two continuous maps
\begin{align*}
    %\repeatable{eq:G-X-actions}{
    \cX\calt \cG\colon &&\cX 
    \bfp{\sigma_{\cX}    }{r_{\cG}}
    \cG\ni (x,s) &\mapsto x\XrightG s \in \cX\\
        \cX \carb \cG\colon &&\cX
    \bfp{\sigma_{\cX}    }{r_{\cG}}
    \cG\ni (x,s) &\mapsto x\XleftG s \in \cG
    %}
\end{align*}
such that the following hold.
\begin{itemize}
    \item For any $x\in \cX$ and $t\in \cG$
        such that $\sigma_{\cX}(x)=r_{\cH}(t)$, we have
        \begin{equation}\tag{R1}
    \repeatable{item:R1}%former \label{item:R6}
    {
        \begin{aligned}
            \sigma_{\cX}(x\XrightG t)&=s_{\cG}(t)
            &&&
            \sigma_{\cX}(x\inv)&= r_{\cG}(x\XleftG t)
            &&&
            \sigma_{\cX}((x\XrightG t)\inv)&=s_{\cG}(x\XleftG t)
        \end{aligned}
    }
    \end{equation}
    \item
      For all $v\in \cX\z $ and $s\in \cG$
        such that $\sigma_{\cX}(v)=r_{\cG}(s)$
        and for all $x\in \cX$, we have:
        \begin{equation}\tag{R2}
        \repeatable{item:R2} %former \label{item:R1}
        {
            v\XleftG s=s \text{ and } x\XrightG \sigma_{\cX}(x)=x
        }
        \end{equation}
    \item 
    For all $(x,y)\in \cX^{(2)}$ and $s\in \cG$
        such that $\sigma_{\cX}(y)=r_{\cG}(s)$, we have
        $s_{\cX}(x\XrightG (y\XleftG s))=r_{\cX}(y\XrightG s)$ and
        \begin{align}
        \tag{R3}
        \repeatable{item:R3} 
        {
        (xy)\XleftG s&=x\XleftG (y\XleftG s)
        }
        \\
        \tag{R4}
        \repeatable{item:R4} 
        {
        (xy)\XrightG s&=[x\XrightG (y\XleftG s)](y\XrightG s)
        }
    \end{align}
    \item
    For all $x\in \cX$ and $(s,t)\in \cG^{(2)}$ 
        such that $\sigma_{\cX}(x)=r_{\cG}(s)$, we have:
        \begin{align}
            \tag{R5}
        \repeatable{item:R5}
        {
            x\XrightG (st)&=(x\XrightG s)\XrightG t
        }
            \\
            \tag{R6}
        \repeatable{item:R6}
        {
            x\XleftG (st)&=(x\XleftG s)[(x\XrightG s)\XleftG t]
        }
        \end{align}
\end{itemize}
We call the \ssra\ {\em free} (resp.\ {\em proper}) if $\XrightG$ is free (resp.\ proper).
\end{definition}

\begin{remark} Similar to our previous computation for the \ssla s, 
for every $t\in\cG$, $x\in\cX$, and $v\in \cX\z$ with $r_{\cG}(t)=\sigma_{\cX}(x)=\sigma_{\cX}(v)$, we have 
    \begin{align}\tag{R7}
    \repeatable{item:R7}
          	{
          	x\XleftG \sigma_{\cX}(x)
           =
           \sigma_{\cX}(x\inv)
          	}
          	\\
           \tag{R8}\repeatable{item:R8-new}
          	{
          	 v\XrightG t\in \cX\z
          	}
           % \tag{R8}\repeatable{item:R8}
          	% {
          	%  \sigma_{\cX}(v\XrightG t)&=s_{\cG}(t)
          	%  }
          	\\
          	\tag{R9}\repeatable{item:R9}
                {
(x\XrightG t)\inv  = x\inv \XrightG (x\XleftG t)
\qquad&\text{ and }\qquad
(x\XleftG t)\inv  = (x\XrightG t) \XleftG t\inv 
                }
          	\\
          	\tag{R10}\repeatable{item:R10}
	{
	s_{\cX}(x\XrightG t)=s_{\cX}(x)\XrightG t \qquad&\text{ and }\qquad r_{\cX}(x\XrightG t)=r_{\cX}(x)\XrightG (x\XleftG t)
	}
    \end{align}
\end{remark}

In a very similar fashion, we can define the 
    \ssp\
for a right action:

\begin{definition}\label{def.ZSProduct.groupoid.right} Let~$\cG$ be a groupoid that has a \ssra\ on~$\cX$. 
Define their \ssp\ as the set
\[\cG\bowtie \cX=\{(t,x): s_{\cG}(t)=\sigma_{\cX}(r_{\cX}(x))\}\]
with multiplication
\[(s,x)(t,y)\coloneqq (s(x\XleftG t), (x\XrightG t)y), \quad\text{whenever } s_{\cX}(x)=r_{\cX}(y) \XrightG t\inv ,\]
and inverse
\[(t,x)\inv  \coloneqq  (x\inv \XleftG t\inv , x\inv \XrightG t\inv ).\] 
\end{definition}

For a right action, we mimic the construction in Definition~\ref{df.lambda H-invariance} verbatim, only replacing the left Haar system  by a right Haar system:
\begin{definition}\label{df.lambda G-invariance}
        Suppose $\cG$ and $\cX$ are \LCH\ groupoids and that $\XrightG$ is a right $\cG$-action on $\cX$ with momentum map $\sigma_{\cG}\colon \cX\to\cG\z$.
    We say that a right Haar system $\{\lambda_u\}_{u\in \cX\z }$ on~$\cX$ is {\em $\XrightG$-invariant} if for all $t\in\cG$ and all $u\in \cX\z $ with $\sigma_{\cX}(u) = r_{\cG}(t)$, we have
    \[\lambda_{u}\XrightG t=\lambda_{u\XrightG t},\] where $\lambda_{u}\XrightG t (E) = \lambda_{u}( E\XrightG t\inv)$.
\end{definition}

Given a \ssra\
    of $\cG$ on $\cX$, a right Haar system of $\cG$ and a $\XrightG$-invariant right Haar system on $\cX$ yields a
    right
Haar systems on the \ssp\ groupoid $\cG\bowtie\cX$ similarly to the result in Proposition~\ref{prop.ZSProduct.Haar.left}. The details are omitted here.

\section{The Orbit Space}\label{sec.orbit}

If~$\cH$ has a \ssla\ on the groupoid $\cX$, then $(h,x)\mapsto h\HleftX x$ is an $\cH$-action on the space~$\cX$  according to Lemma \ref{rm.ssla means cH acts on cX}. We can therefore construct the quotient space, $\cH\backslash\cX$, whose elements we will denote by $\cH\HleftX x$. 
We will now show that we can equip this space with its own groupoid structure as long as the action is \assumption{free and proper}.

Recall from Lemma~\ref{lem:restricted action on cXz} that $\HleftX$ restricts to an $\cH$-action on $\cX\z$, so we may consider $\cH\backslash \cX\z$.
We
define $s_{\cH\backslash\cX}, r_{\cH\backslash\cX}\colon \cH\backslash\cX \to \cH\backslash\cX\z$ by
    \begin{equation}\label{eq.s and r for quotient}
        s_{\cH\backslash\cX}(\cH\HleftX x)= \cH\HleftX s_{\cX}(x)
        \quad\text{and}\quad
        r_{\cH\backslash\cX}(\cH\HleftX x)= \cH\HleftX r_{\cX}(x).
    \end{equation}
These are well defined by \eqref{item:L10}.

\begin{lemma}\label{lem:s and r of H under X}
        \assumption{If $s_{\cH}$ and $s_{\cX}$ are open,} then the map $s_{\cH\backslash\cX}$ is also open.
\end{lemma}

\begin{proof}  
     Since \assumption{$s_{\cH}$ is open,} the quotient map $q|_{\cX\z}\colon \cX\z\to \cH\backslash\cX\z$ 
     is open by \cite[Proposition 2.12]{Wil2019}. The claim now follows from continuity of $q$ and commutativity of the  diagram below.
     \[
     \begin{tikzcd}
         \cX \ar[r, "q"]\ar[d, "s_{\cX}"]
         & \ar[d, "s_{\cH\backslash\cX}"]
         \cH\backslash\cX
        \\
        \cX\z \ar[r, "q|_{\cX\z}"]& \cH\backslash\cX\z
     \end{tikzcd}
     \qedhere
     \]
\end{proof}

\begin{lemma}\label{lm.orbit multiplication}
    Suppose~$\cH$ has a \ssla\ on $\cX$, and fix two elements $\xi,\eta$
    of $\cH\backslash\cX$
    for which $s_{\cH\backslash\cX}(\xi)= r_{\cH\backslash\cX}(\eta)$.
    Then we can find $x_{1}\in \xi
    $ and $y_{1}\in \eta
    $ such that $s_{\cX}(x_{1})=r_{\cX}(y_{1})$.
    
    Moreover, \assumption{if the action of $\cH$ on $\cX$ is free,} then any two more such elements $x_{2},y_{2}$ satisfy $\cH\HleftX  (x_{2}y_{2})=\cH\HleftX (x_{1}y_{1}) $.
\end{lemma}

\begin{proof}
    For existence,
    start with two arbitrary elements $x\in \xi$ and $y\in \eta$. By construction of $s_{\cH\backslash\cX}$ and $r_{\cH\backslash\cX}$, we have $s_{\cX}(x)\in s_{\cH\backslash\cX}(\xi)$ and $r_{\cX}(y)\in r_{\cH\backslash\cX}(\eta)$. As the two equivalence classes coincide by assumption, there exists $h\in \cH$ such that $s_{\cX}(x) = h\HleftX r_{\cX}(y)$. Since the right-hand side equals $r_{\cX}(h \HleftX y)$ by \eqref{item:L10} (Lemma~\ref{lem:acting_on_units}), we see that we can pick $ x_{1}\coloneqq x$ and $y_{1}\coloneqq h \HleftX y \in \cH\HleftX y = \eta$.
    
    To see the claim about  the product,
    let $ k,l \in \cH$ be such that $x_{2}= k\HleftX x_{1}$ and $y_{2}= l   \HleftX y_{1}$.
    By \eqref{item:L10}, 
    \begin{align*}
        s_{\cX}(x_{2})&=s_{\cX} (k\HleftX x_{1})=( k \HrightX x_{1})\HleftX s_{\cX}(x_{1}), 
        \qquad\text{ and}
        \\
        r_{\cX}(y_{2})&=r_{\cX}(l\HleftX y_{1})= l \HleftX r_{\cX}(y_{1})=l \HleftX s_{\cX}(x_{1}).
    \end{align*}
    Since the left-hand sides of these equations are assumed to be equal and \assumption{since the $\cH$-action is free,} we conclude that $l=k \HrightX x_{1}$. Therefore, by \eqref{item:L6},
    \begin{align*}
        x_{2}y_{2}&= (k \HleftX x_{1})(l \HleftX y_{1}) = (k \HleftX x_{1})([k \HrightX x_{1}] \HleftX y_{1}) = k\HleftX (x_{1}y_{1}),
    \end{align*}
    so $\cH\HleftX (x_{2}y_{2})=\cH\HleftX (x_{1}y_{1})$, as claimed.
\end{proof}

The lemma allows us to make the following definition:
\begin{proposition}\label{prop.gpd structure on orbit space}
    Suppose~$\cH$ has a  \ssla\ on~$\cX$ for which $\HleftX$ is {\em free on $\cX$}.
    For two elements $\xi,\eta$  of the orbit space  $\cH\backslash\cX$ with $s_{\cH\backslash\cX}(\xi)=r_{\cH\backslash\cX}(\eta)$,
    define 
    \[
    \xi\eta
    =
    \cH\HleftX (xy)
    \quad
    \text{where }x\in \xi, y\in \eta \text{ are such that }s_{\cX}(x)=r_{\cX}(y)
    .
    \]
    Further, define
    $$(\cH\HleftX x)\inv =\cH\HleftX x\inv .$$
    With this structure,
    $\cH\backslash\cX$ is a (non-topological) groupoid.
    
    If we further assume that $\HleftX$ is proper and $s_{\cH}$ is open, then $\cH\backslash\cX$ is a \LCH\ groupoid with the quotient topology, and if
    $\cX$ is 
    \etale, then so is $\cH\backslash\cX$.
\end{proposition}

\begin{proof}
    We have seen in Lemma~\ref{lm.orbit multiplication} that, since \assumption{the $\cH$-action is free}, the multiplication is well defined and independent of the choice of $x,y$. To see that the inversion is well defined, suppose that $x_{1}\in \cH\HleftX x =\xi$, i.e., $x_{1}=h\HleftX x$ for some $h$. Then by \eqref{item:L9}, we have $x_{1}\inv = ( h  \HleftX x)\inv= (h\HrightX x)\HleftX x\inv$, so $x_{1}\inv \in \cH\HleftX x\inv$, and hence the definition of $(\cH\HleftX x)\inv$ does not depend on the chosen representative.

    The algebraic properties of a groupoid are now easy to verify and follow from the algebraic properties that~$\cX$ satisfies.

    \smallskip

        Now suppose $\HleftX$ is proper and $s_{\cH}$ is open.
        Since we assume our groupoids 
        \assumption{$\cH$ and~$\cX$ to be}
        \LCH, 
        it follows from \cite[Proposition 2.18]{Wil2019} 
        that the quotient is \LCH.

    To show that the multiplication map $(\cH\backslash\cX)^{(2)}\to\cH\backslash\cX$  
    is continuous, suppose we are given a net $\{(\xi_{i },\eta_{i })\}_{i\in I }$ in $(\cH\backslash\cX)^{(2)}$ that converges to some composable pair $(\xi,\eta)$. Because of Lemma~\ref{lem:exercise in top:cts maps}, it suffices to show that a subnet of $\{\xi_{i }\eta_{i }\}_{i\in I }$ converges to $\xi\eta$.
    
    As $(\cH\backslash\cX)^{(2)}$ has the subspace topology of the product topology on $(\cH\backslash\cX) \times (\cH\backslash\cX)$, convergence implies  that $\xi_{i }\to \xi$ and $\eta_{i }\to \eta$ in $\cH\backslash\cX$.
    \assumption{Since $s_{\cH}$ is open,} the quotient map $q$ is open by \cite[Proposition 2.12]{Wil2019}.
    Thus, if we fix $x\in \xi$, then by 
    Proposition~\ref{Fell's criterion}
    we can find a subnet of $\{\xi_{i }\}_{i\in I }$ that is the image under $q$ of a net in~$\cX$ that converges to $x$; 
    without loss of generality, the subnet is the net itself,  meaning there exist  $x_{i }\in\cX$ such that  $x_{i }\to x$ and $\cH\HleftX x_{i } = \xi_{i }$. Once again by passing to a subnet, we can without loss of generality assume that $\{\eta_{i }\}_{i\in I }$ is the image under $q$ of a convergent net, say of  $y_{i }\to y\in\eta$. In other words, by passing to a subnet of a subnet, we can without loss of generality assume that $\{(\xi_{i },\eta_{i })\}_{i\in I }$ itself can be lifted to a net $\{(x_{i  },y_{i } )\}_{i\in I  }$ that converges to $(x,y)\in \xi\times\eta$ in $\cX\times\cX$. Since $(\xi_{i  },\eta_{i  })\in (\cH\backslash\cX)^{(2)}$, we have 
    \[
        \cH\HleftX s_{\cX}(x_{i  })
        =
        s_{\cH\backslash\cX} (\xi_{i  })
        =
        r_{\cH\backslash\cX} (\eta_{i  })
        =
        \cH\HleftX r_{\cX}(y_{i  }),
    \]
    so we can find  $h_{i  }\in\cH$ such that $s_{\cX}(x_{i  }) = h_{i  }\HleftX r_{\cX}(y_{i  }) 
    $; note that $h_{i }$ is unique \assumption{by freeness}. Similarly, there exists a unique $h$ with $s_{\cX}(x) = h\HleftX r_{\cX}(y)$.  Continuity of $s_{\cX}$ and $r_{\cX}$ implies that 
    \begin{equation}\label{eq:why properness is needed}
        \bigl(
            h_{i  }\HleftX r_{\cX}(y_{i  })
            ,
            r_{\cX}(y_{i  }) 
        \bigr)
        =
        \bigl(
            s_{\cX}(x_{i  }), r_{\cX}(y_{i  })
        \bigr)
        \overset{i  }{\to}
        \bigl(
            s_{\cX}(x), r_{\cX}(y)
        \bigr)
        =
        \bigl(
            h\HleftX r_{\cX}(y)
            ,
            r_{\cX}(y) 
        \bigr)
    \end{equation}
    Since \assumption{$\HleftX$ is proper}, this convergence implies that
    (a subnet of) $\{h_{i  }\}_{i\in I  }$ converges.
        Since~$\cX\z$
    \assumption{%
    is Hausdorff and $\HleftX$ is free,} it must converge to $h$. In particular, continuity if $\HleftX$ implies that $\{(x_{i  }, h_{i } \HleftX y_{i  })\}_{i\in I  }$ is a net in $\cX^{(2)}$ that converges to the composable pair $(x,h\HleftX y)$. Continuity of the multiplication on~$\cX$ implies that $\{x_{i  }[h_{i } \HleftX y_{i  }]\}_{i\in I  }$ converges to $x[h\HleftX y]$.
    Since
    \[
        q(x_{i  }[h_{i } \HleftX y_{i  }])
        =
        \cH \HleftX (x_{i  }[h_{i } \HleftX y_{i  }])
        =
        (\cH \HleftX x_{i  })(\cH \HleftX y_{i  })
        =
        \xi_{i } \, \eta_{i } 
    \] and
    $
        q(x[h\HleftX y])
        =
        \xi\,\eta,
    $
    continuity of $q$ implies that $\{\xi_{i }  \eta_{i } \}_{i\in I  }$ converges to $\xi\eta$. This proves that the multiplication on $\cH\backslash\cX$ is continuous.

    \smallskip
    
    For the inversion map, the argument is similar: if $\xi_{i }\to \xi$ in $\cH\backslash\cX$, then openness of $q$ allows a lift $\{x_j \}_{j\in J} $ of a subnet $\{\xi_j \}_{j\in J} $ which converges to a fixed preimage $x$ of $\xi$. Continuity of the inversion in~$\cX$ implies that $x_j \inv \to x\inv $, and continuity of $q$ implies $\xi_j \inv = \cH\HleftX (x_j \inv)\to\cH\HleftX (x\inv)= \xi\inv$. By Lemma~\ref{lem:exercise in top:cts maps}, this suffices to show that the inversion on $\cH\backslash\cX$ is continuous.
    
    \smallskip

    Lastly, assume that $\cX$ is \etale, so its source map is an open map and its unit space is open. As argued above, the quotient map $q\colon \cX\to \cH\backslash\cX$ is open, and so $(\cH\backslash\cX)\z=\cH\backslash\cX\z = q(\cX\z)$ is open, i.e.,  $\cH\backslash\cX$ is $r$-discrete. Since \assumption{$s_{\cH}$ and $s_{\cX}$ are open}, Lemma~\ref{lem:s and r of H under X} implies that the source map of $\cH\backslash\cX$ is open, and so \cite[Proposition 1.29]{Wil2019} implies that $\cH\backslash\cX$ is \etale.
\end{proof}

\begin{example}\label{ex:ssla of trivial group and groupoid:quotient}
    If we consider the \ssla\ of the trivial groupoid $\cX\z$ on $\cX$ as defined in Example~\ref{ex:ssla OF trivial gpd}, then $\cX\z\backslash\cX\cong \cX$ via $\cX\z\HleftX x\mapsto x$, since $\HleftX$ is trivial.
    
    Likewise, the trivial group $\{e\}$ with its (trivial) \ssla\  on a groupoid $\cX$  as defined in Example~\ref{ex:ssla of trivial group} is (trivially) free and proper. The quotient groupoid $\{e\}\backslash\cX$ is exactly the groupoid $\cX$ if we identify $\{e\}\HleftX x$ with $x$.
\end{example}

\subsection{\Sspa s}

We are now in a position where we can define a generalized notion of compatible actions. 
\begin{definition}\label{df.compatible}
    Suppose the two groupoids $\cH,\cG$ act on the left resp.\ right of a groupoid~$\cX$ by \ss\ actions. We say the actions are {\em\intune} if for any $h\in \cH$, $x\in \cX$, and $s\in \cG$ with $s_{\cH}(h)=\rho_{\cX}(x)$ and $\sigma_{\cX}(x)=r_{\cG}(s)$, we have
    \repeatable{properties of two sets of compatible actions}{
    \begin{enumerate}[leftmargin=1.5cm, label=\textup{(C\arabic*)}, start=0]
    \item 
    \label{cond.compatible.c0} $
        \sigma_{\cX}(h\HleftX x) = \sigma_{\cX}(x)$ in $\cG\z$ and $
        \rho_{\cX}(x)
        =
        \rho_{\cX}(x\XrightG s)$ in $\cH\z$,
        \item\label{cond.compatible.c1} $h\HleftX (x\XrightG s)=(h\HleftX x)\XrightG s$
            in $\cX$,
        \item\label{cond.compatible.c2} $(h\HleftX x)\XleftG s=x\XleftG s$
            in $\cG$, and
        \item\label{cond.compatible.c3} $h\HrightX (x\XrightG s)=h\HrightX x$
            in $\cH$.
    \end{enumerate}
    }
\end{definition}
    Note that Condition~\ref{cond.compatible.c0} ensures that the elements in the other conditions make sense.

\begin{definition}\label{df.something}   
   Suppose the two groupoids $\cH,\cG$ act on the left resp.\ right of a groupoid~$\cX$ by \ss\ actions. If the \ss\ actions are \intune\ and both free and proper, and if $\cH$, $\cG$, and $\cX$ have open source %(and range) 
   maps, then we call~$\cX$ an {\em $(\cH,\cG)$-\sspa}.
\end{definition}

\begin{remark} In case of 
    the semidirect product construction
in \cite[Appendix A.2]{KMQW2013}, 
we have $x\XleftG s=s$ and $h\HrightX x=h$ 
for all $h\in \cH$, $x\in \cX$, and $s\in \cG$. Therefore, Conditions~\ref{cond.compatible.c2} and~\ref{cond.compatible.c3} are 
    trivially satisfied since both sides of the first equation are $s$ and both sides of the second are $h$. 
Thus, in this case, the \intuneadj\ conditions simply reduce to the commuting
    conditions~\ref{cond.compatible.c0} and~\ref{cond.compatible.c1}.
\end{remark}

\begin{example}\label{ex:ssla of trivial gp and gpd:sometimes a something}
Let $\cG$ and  $\cX$ be 
groupoids whose source maps are open, and suppose that $\cG$ has a \ssra\ on~$\cX$ that is free and proper (Definition~\ref{df.right.selfsimilar}). Then $\cX$ is a $(\{e\}, \cG)$-\sspa. Indeed, the trivial actions $\HleftX, \HrightX$ constitute a free and proper \ssla\ of $\{e\}$ on~$\cX$ (see Examples~\ref{ex:ssla of trivial group} and~\ref{ex:ssla of trivial gp and gpd:free and proper}), and the following computations show that the actions of $\cX\z$ and of $\cG$ are \intune, where $x\in\cX$ and $s\in\cG$ are such that $\sigma_{\cX}(x)=r_{\cG}(s)$.
\begin{description}%[leftmargin=1.7cm, label=\textup{Ad (C\arabic*):}, start=0]
    \item[Ad \ref{cond.compatible.c0}] Since $\HleftX$ is trivial, we have $\sigma_{\cX}(h\HleftX x) = \sigma_{\cX}(x)$, and since $\rho_{\cX}\colon \cX\to \{e\}$ is constant, we have $
        \rho_{\cX}(x)=
        \rho_{\cX}(x\XrightG s)$.
        \item[Ad \ref{cond.compatible.c1}, \ref{cond.compatible.c2}] Since $\HleftX$ is trivial, we have $e\HleftX (x\XrightG s)=x\XrightG s=(e\HleftX x)\XrightG s$ and $(e\HleftX x)\XleftG s=x\XleftG s$.
        \item[Ad \ref{cond.compatible.c3}] Since $\HrightX$ is trivial, we have $e\HrightX (x\XrightG s)=e=e\HrightX x$.
    \end{description}

    Note, however, that $\cX$ is {\em not} a $(\cX\z,\cG)$-\sspa\ (even though $\cX$ has a free and proper \ssla\ of $\cX\z$ by Example~\ref{ex:ssla OF trivial gpd}): If there exists one $(x,s)\in \cX\bfp{\sigma}{s}\cG$ with $s \notin \cG\z$, then freeness of the $\cG$-action $\XrightG$ on $\cX\z$ implies that
    \[
        % \rho_{\cX}(x)=
        r_{\cX}(x)\neq r_{\cX}(x)\XrightG(x\XleftG s)
        \overset{\eqref{item:R10}}{=} r_{\cX}(x\XrightG s).
        % =\rho_{\cX}(x\XrightG s),
     \]
     As 
     the momentum map $\rho_{\cX}$ for the $\cX\z$-action on $\cX$ is $r_{\cX}$ in this setting, the above inequality
     conflicts with Condition~\ref{cond.compatible.c0}.
\end{example}
 
\begin{remark}\label{rmk:ill advised}
    In Example~\ref{ex:reconciliation}, we showed that the (trivial) \ssla\ of $\{e\}$ on a groupoid $\cX$ gives rise to the `standard' \ssla\ of $\cX\z$ on $\cX$ as defined in Example~\ref{ex:ssla OF trivial gpd}, 
    and it then also followed that $\cX\bowtie\{e\}\cong \cX\bowtie\cX\z$. This seemed to indicate that the pairs $(\cX,\{e\})$ and $(\cX,\cX\z)$ are `the same' in some sense.

    However, the above example shows that this point of view is ill-advised, since a \sspa\ $\cX$ between $\cH$ and $\cG$ need not be one between $\widetilde{\cH}=\cH\ltimes \cX\z$ and $\cG$. The reason is that, Condition~\ref{cond.compatible.c0} for the pair $(\cH,\cG)$ does not imply the same condition for $(\widetilde{\cH},\cG)$, since the momentum maps on $\cX$ with respect to the left actions do not need to coincide:
    we have $\rho_{\cX}\colon\cX\to\cH\z$ for the left $\cH$-action $\HleftX$, while we have $r_{\cX}\colon\cX\to\cX\z=\widetilde{\cH}\z$ for the left $\widetilde{\cH}$-action $\cdot$ (see Lemma~\ref{lem.ss.is.matched}).
\end{remark}

 Given a  $(\cH,\cG)$-\sspa\  $\cX$, we have shown in Proposition~\ref{prop.gpd structure on orbit space} that the orbit space $\cH\backslash \cX$ and, by extension, $\cX/\cG$ are groupoids. In Proposition~\ref{prop.ssp actions on quotients}, we will establish that~$\cH$ has a \ssla\ on $\cX/\cG$; similarly,~$\cG$ has a \ssra\ on $\cH\backslash\cX$. We can then consider the 
    \ssp\ % \ZS\ product 
groupoids $(\cX/\cG)\bowtie \cH$ and $\cG\bowtie (\cH\backslash \cX)$, as constructed in Definition~\ref{def.ZSProduct.groupoid.left}. Our main result is that these two
    \ssp\ % \ZS\ product 
 groupoids are equivalent via their actions on~$\cX$ 
 in the sense of \cite[Definition 2.1]{MRW:Grpd}
 as summed up in the following theorem; this generalizes \cite[Lemma 3.2]{KMQW2013}.
 
\begin{theorem}[cf.\ {\cite[Lemma 3.2]{KMQW2013}}]\label{thm.groupoid.eq} 
Let $\cH,\cG,\cX$ be 
groupoids, and suppose that~$\cX$ is a $(\cH,\cG)$-\sspa\  in the sense of
    Definition~\ref{df.something},
that is,
\begin{itemize}
    \item $s_{\cH}$, $s_{\cG}$, and $s_{\cX}$ are open maps,
    \item $\cH$ has a \ssla\ on~$\cX$ that is free and proper (Definition~\ref{df.left.selfsimilar}),
    \item $\cG$ has a \ssra\ on~$\cX$ that is free and proper (Definition~\ref{df.right.selfsimilar}), and
    \item the two actions are \intune\ (Definition~\ref{df.compatible}).
\end{itemize}
Then there is a natural way to turn $\cX$ into
a groupoid equivalence from $(\cX/\cG)\bowtie \cH$ to $\cG\bowtie (\cH\backslash \cX)$. 
\end{theorem}

For the description of the equivalence structure on $\cX$, see Proposition~\ref{prop.actions.X}.
    Examples of applications of Theorem~\ref{thm.groupoid.eq} can be found in Subsection~\ref{ssec:Applications of thm.groupoid.eq}.

    \begin{example}\label{ex:recovering KMQW2023, Lemma 3.2}
    Theorem~\ref{thm.groupoid.eq} recovers \cite[Lemma 3.2]{KMQW2013}: when $H$ and $G$ are \LCH\ groups and their free and proper actions on a groupoid 
    $\cX$ %$X$
    are actions by automorphisms, then we may let
    $\cX$ %$X$
    act trivially on both $H$ and $G$, i.e., $h\HrightX x=h$ and $x\XleftG s=s$. This makes
    $\cX$ 
    a $(H,G)$
    -\sspa, and the equivalence structure  alluded to in Theorem~\ref{thm.groupoid.eq} makes
    $\cX$ 
    a groupoid equivalence between $(
    \cX   
    /G)\rtimes H$ and $G\ltimes (H\backslash 
    \cX 
    )$.
    \end{example}

\begin{proposition}\label{prop.ssp actions on quotients}
    Suppose~$\cX$ is a $(\cH,\cG)$-\sspa.
Then~$\cH$ has a \ssla\ on $\cX/\cG$: the momentum map is given by $\tilde\rho(x\XrightG G)=\rho_{\cX}(x)$, and the actions are defined by
\begin{align*}
    %\repeatable{eq:H-X over G-actions}
    %{
        \cH \cart (\cX/\cG)\colon&&\cH \bfp{s}{\tilde{\rho}} (\cX/\cG)
        \ni (h,x\XrightG \cG) &\mapsto h\qHleftX (x\XrightG \cG)
        \coloneqq  (h \HleftX x)\XrightG \cG \in \cX/\cG
        \\
        \cH \calb (\cX/\cG)\colon&&\cH \bfp{s}{\tilde\rho} (\cX/\cG)\ni (h,x\XrightG \cG)
        &\mapsto h\qHrightX (x\XrightG \cG) \coloneqq 
        h\HrightX x \in \cH
    %}
\end{align*}

Likewise,
$\cG$ has a \ssra\ on $\cH\backslash \cX$: the momentum map is given by 
$\tilde\sigma(\cH\HleftX x)=\sigma_{\cX}(x)$, and the actions are defined by
    \begin{align*}
    %\repeatable{eq:H_under_X-G-actions}
    %{
        (\cH\backslash \cX) \calt \cG\colon
        &&
        (\cH\backslash \cX) \bfp{\tilde\sigma}{r} \cG\ni (\cH\HleftX x, s)
        &\mapsto (\cH\HleftX x) \qXrightG s\coloneqq \cH\HleftX (x\XrightG s) \in \cH\backslash \cX
        \\
        (\cH\backslash \cX) \carb \cG\colon
        &&
        (\cH\backslash \cX) \bfp{\tilde\sigma}{r} \cG\ni (\cH\HleftX x, s)
        &
        \mapsto (\cH\HleftX x) \qXleftG s\coloneqq x\XleftG s \in \cG
    %}
    \end{align*}
\end{proposition}

Note that, even though $\HleftX$ and $\XrightG$ are free and proper, the same is not necessarily true for~$\qHleftX$ or~$\qXrightG$. This fact prevents us from turning an iterated quotient such as $\cH\backslash (\cX/\cG)$ or $(\cH\backslash \cX)/\cG$ into a topological groupoid, if we were so inclined. (Luckily, we aren't.)

\begin{proof} The momentum map is well defined by Condition~\ref{cond.compatible.c0} and it is surjective
because $\rho_{\cX}$ is surjective. 
 It remains to check that $\tilde{\rho}$ is continuous.
Since
    \assumption{%
    $r_{\cG}$ is open,
    }%
    we know that the quotient map $\cX\to \cX/\cG$ is open by \cite[Proposition 2.12]{Wil2019}. In particular, if $\{x_{i}\XrightG \cG\}_{i\in I}$ is a net converging to $x\XrightG \cG$ in $\cX/\cG$, then 
    Proposition~\ref{Fell's criterion}
  says that we can find a subnet $\{x_{f(j)}\XrightG \cG\}_{j\in J}$ which allows a convergent lift  in $\cX$, i.e., there exist $y_{j}\in x_{f(j)}\XrightG \cG$ for all $j$ with $y_{j}\to y$ for some $y\in x\XrightG \cG$.
    Continuity of $\rho_{\cX}$ then implies 
    \[
    \tilde{\rho}(x_{f(j)}\XrightG\cG)
    =
    \rho_{\cX}(y_{j})\to
    \rho_{\cX}(y)
    =
    \tilde{\rho}(x\XrightG\cG).
    \]
    Using Lemma~\ref{lem:exercise in top:cts maps}, we conclude that $\tilde{\rho}$ is continuous.

    We next verify that $
    \qHleftX
    $ is well defined. If $x\XrightG \cG=y\XrightG \cG$, there exists a unique $s\in \cG$ such that $x=y\XrightG s$. Now
    by the commuting Condition~\ref{cond.compatible.c1},
    \[h\HleftX x=h\HleftX (y\XrightG s) = (h\HleftX y) \XrightG s.\]
    Therefore, $(h\HleftX x)\XrightG \cG=(h\HleftX y)\XrightG \cG$. Similarly, to show that $
    \qHrightX
    $ is well defined, 
    let $x,y,s$ be as above, and let $h\in\cH$ be such that $s_{\cH}(h)=\rho_{\cX}(x)$.
    By Condition~\ref{cond.compatible.c3}, we have
    \[h\HrightX x=h\HrightX (y\XrightG s) = h\HrightX y.\]
    To see that this $\cH$-action on $\cX/\cG$ is \ss, we observe that the $\cH$-action on $\cX$ passes through the quotient and $(x\XrightG \cG)(y\XrightG \cG)=(xy)\XrightG \cG$ whenever $r_{\cX}(y)=s_{\cX}(x)$. Therefore, the Conditions~\eqref{item:L2} through~\eqref{item:L6} from the \ss\ $\cH$-action on $\cX$ all pass through to the $\cH$-action on $\cX/\cG$, proving that this $\cH$-action on $\cX/\cG$ is also \ss.

    Lastly, we will check that $\qHleftX$ is continuous.
    So assume we are given $h_{i}\in \cH$ such that $s_{\cH}(h_{j})=\rho_{\cX}(x_{i})$ and $h_{i}\to h$. By continuity of $\HleftX$ and $\HrightX$, we have $h_{f(j)}\HleftX y_{j}\to h\HleftX y$ in~$\cX$ and $h_{f(j)}\HrightX y_{j}\to h\HrightX y$ in $\cH$. Continuity of the quotient map $\cX\to\cX/\cG$ then implies that
    \begin{align*}
    h_{f(j)}\qHleftX (x_{f(j)}\XrightG\cG)=(h_{f(j)}\HleftX y_{j})\XrightG \cG \to (h\HleftX y)\XrightG \cG =h\qHleftX (x\XrightG\cG),
    \intertext{and likewise we have}
    h_{f(j)}\qHrightX (x_{f(j)}\XrightG\cG)=h_{f(j)}\HrightX y_{j}  \to h\HrightX y =h\qHrightX (x\XrightG \cG).
    \end{align*}
    Lemma~\ref{lem:exercise in top:cts maps} again implies that $\qHleftX$ is continuous.
    
    The claims for $\qXrightG$ and $\qXleftG$ follow {\em mutatis mutandis}.
\end{proof}

Following Definitions~\ref{def.ZSProduct.groupoid.left} and~\ref{def.ZSProduct.groupoid.right}, we obtain two groupoids, $(\cX/\cG)\bowtie \cH$ and $\cG\bowtie (\cH\backslash \cX)$. By Remark~\ref{rm.unit space of ssp}, 
the unit space of the \ssp\ groupoid $(\cX/\cG)\bowtie\cH$ is homeomorphic to the unit space  of $\cX/\cG$. 
In other words, we have:
\begin{align*}
    ((\cX/\cG)\bowtie \cH)\z  &\approx
    (\cX\z)/\cG
    = \{u\XrightG \cG: u\in \cX\z \}; \\
    (\cG\bowtie (\cH\backslash \cX))\z  & \approx
    \cH \backslash (\cX\z)
    = \{\cH\HleftX u: u\in \cX\z \}.
\end{align*}

The following lemma computes the range and source maps explicitly for these two \ssp\ groupoids. It follows immediately from Remark~\ref{rm.unit space of ssp}. 
\begin{lemma}\label{lm.orbitZS.rs}
Consider $(\xi, h)\in (\cX/\cG)\bowtie \cH$ and $(t, \eta)\in \cG\bowtie (\cH\backslash \cX)$, and let $x\in\xi$ and $y\in\eta$ be arbitrary. We have 
\begin{enumerate}[label=\textup{(\arabic*)}]
    \item $r(\xi, h)=r_{\cX/\cG}(\xi)=r_{\cX}(x)\XrightG \cG$
    \item $s(\xi, h)=h\inv \qHleftX s_{\cX/\cG}(\xi) = (h\inv \HleftX s_{\cX}(x))\XrightG \cG$
    \item $r(t, \eta)=r_{\cH\backslash\cX}(\eta)\qXrightG t\inv =\cH\HleftX (r_{\cX}(y)\XrightG t\inv )$
    \item $s(t, \eta)=s_{\cH\backslash\cX}(\eta)=\cH\HleftX s_{\cX}(y)$
\end{enumerate}
\end{lemma}

We now define left and right actions of these groupoids on $\cX$.

\begin{proposition}\label{prop.actions.X} 
Let $\cX$ be a $(\cH,\cG)$-\sspa. 
Define $\mathfrak{r}\colon \cX\to [(\cX/\cG)\bowtie \cH]\z$ and $\mathfrak{s}\colon \cX\to [\cG\bowtie (\cH\backslash \cX)]\z$ by
\[\mathfrak{r}(x)=r_{\cX}(x)\XrightG \cG \quad \text{and} \quad \mathfrak{s}(x)=\cH\HleftX s_{\cX}(x).
\]
These are well-defined,
surjective, continuous, open maps. Using them as momentum maps, we can define
a left-$(\cX/\cG)\bowtie \cH$ and a right-$\cG\bowtie (\cH\backslash \cX)$ action via:
\begin{align*}
    ((\xi,h),y) & \mapsto (\xi,h)\cdot y \coloneqq  x(h\HleftX y), &&\text{where $x\in \xi$ is such that } 
    (x,h\HleftX y)\in\cX^{(2)}
    ;
    \\
    (y,(t,\eta)) & \mapsto y\cdot (t,\eta) \coloneqq  (y\XrightG t)z,&& \text{where  $z\in \eta$ is such that } (y\XrightG t,z)\in\cX^{(2)}
    .
\end{align*}
Here, $x(h\HleftX y)$ and $ (y\XrightG g)z$ denote composition in the groupoid $\cX$.
These actions are free and proper, and they commute.
\end{proposition}

\begin{proof}
We will do everything for the left-hand side; the claims for the right-hand side will follow {\em mutatis mutandis}.

First, notice that $\mathfrak{r}$ is clearly continuous (resp.\ surjective) since $r_{\cX}$ is continuous (resp.\ surjective). Furthermore, $\mathfrak{r}$ is open as a concatenation of open maps:  \assumption{$r_{\cX}$ is open by assumption,} and the quotient map $\cX\to \cX/\cG$ is open by \cite[Proposition 2.12]{Wil2019} \assumption{since $r_{\cG}$ is open by assumption}.

Next, we
verify that the left-$(\cX/\cG)\bowtie \cH$ action  is well defined. 
Given a pair $((\xi,h),y)$ with $\mathfrak{r}(y)=s(\xi,h)$, it follows from the definition of $\mathfrak{r}$, from Lemma~\ref{lm.orbitZS.rs}, and from \eqref{item:L10} that
$r_{\cX}(h\HleftX y)\XrightG \cG=s_{\cX/\cG}(\xi)$, where $s_{\cX/\cG}\colon \cX/\cG\to (\cX\z)/\cG$ is as in Equation~\eqref{eq.s and r for quotient}. 
Therefore,
there exists
    $x\in \xi$ such that $s_{\cX}(x)=r_{\cX}(h\HleftX y)$.
Since the action on~$\cX$ is \assumption{assumed to be free,} we may invoke a $\XrightG$-version of Lemma~\ref{lm.uniqueness}  to conclude
that such $x$ must be unique. 
Therefore, the left action is well defined.

\smallskip

We now verify that the left action is
free. Pick any $y\in \cX$ and $(\xi,h)\in (\cX/\cG)\bowtie \cH$ such that $(\xi,h)\cdot y=y$, and let $x\in \xi$
satisfy $r_{\cX}(h\HleftX y)=s_{\cX}(x)$. By the definition of the left $(\cX/\cG)\bowtie \cH$-action on $\cX$, our assumption $(\xi,h)\cdot y=y$ implies
$x(h\HleftX y)=y$. In particular, 
\[s_{\cX}(y)=s_{\cX}(x(h\HleftX y))=s_{\cX}(h\HleftX y)=h\HleftX s_{\cX}(y).\] 
Since \assumption{the $\cH$-action on~$\cX$ is free}, we have $h\in\cH\z $ and thus $y=xy$. This only happens when $x=r_{\cX}(y)$ and thus $(\xi,h)=(r_{\cX}(y)\XrightG \cG,h)$ is a unit in $(\cX/\cG)\bowtie \cH$. 

\smallskip

To see that the left action is continuous, assume that we have nets $\{(\xi_{i },h_{i })\}_{i\in I}$ in $(\cX/\cG)\bowtie
 \cH$ and $\{y_{i }\}_{i\in I}$ in~$\cX$  which converge to $(\xi, h)$ and $y$, respectively, and which satisfy 
 \[
    s(\xi_{i },h_{i })
    =
    \mathfrak{r}(y_{i })
    ,\quad \text{ i.e., }\quad
     s_{\cX/\cG}(\xi_{i })
    =
    r_{\cX}(h_{i } \HleftX y_{i })\XrightG \cG.
 \]
  If we let $x_{i }\in \xi_{i }$ and $x\in \xi$ be the unique elements such that \[ u_{i }\coloneqq s_{\cX}(x_{i })
    =
    r_{\cX}(h_{i } \HleftX y_{i })\quad\text{ and }\quad u\coloneqq s_{\cX}(x)
    =
    r_{\cX}(h \HleftX y),\] then by Lemma~\ref{lem:exercise in top:cts maps}, it suffices to find a subnet of $\{x_{i } (h_{i } \HleftX y_{i })\}_{i\in I}$ that converges to $x(h\HleftX y)$. As $(h_{i },y_{i })\to (h,y)$, we only need to show that a subnet of  $\{x_{i }\}_{i\in I}$ converges to $x$; furthermore, it gives us that $u_{i }\to u$. 
    Since $\xi_{i }\to\xi$
        and since $q\colon \cX\to\cX/\cG$ is  open,
    Proposition~\ref{Fell's criterion}
 then implies that there exists a subnet $\{\xi_{j}\}_{j\in J}$ of $\{\xi_{i }\}_{i\in I}$ and lifts $z_{i}\in \xi_{j}$ such that $z_{i}\to x$. As $x_{j}\in \xi_{j}$ also, there exist $t_i\in \cG$ such that $z_{i}=x_{j}\XrightG t_{i}$.  In particular, by continuity of $s_{\cX}$ and by \eqref{item:R10}, we have $u_{j} \XrightG t_{i}=s_{\cX}(z_{i})\to s_{\cX}(x)=u$. Since \(u_{j}\to u\), we therefore have that
    \[
       ( u_{j}\XrightG t_{i},  u_{j})
       \to
       ( u,u).
    \]
    As the right action of~$\cG$ on~$\cX$ is \assumption{free and proper}, it now follows from \cite[Corollary 2.26]{Wil2019} that $t_{i}\to \sigma_{\cX}(u)= \sigma_{\cX}(x)$ 
        by definition of $\sigma_{\cX}$.
    Thus,
    \(
        x_{j}=z_{i}\XrightG t_{i}\inv
    \) converges to $x\XrightG \sigma_{\cX}(x)\inv= x\XrightG \sigma_{\cX}(x)=x$ by \eqref{item:R2}.

\smallskip

    To show that the left action is proper, suppose $y_{i }\to y$ and $(\xi_{i },h_{i })\cdot y_{i } \to z$ in $\cX$; according to \cite[Proposition 2.17]{Wil2019}, it suffices to show that $\{(\xi_{i },h_{i })\}_{i\in I}$ has a convergent subnet. As before, let $x_{i }\in \xi_{i } $ be the unique element such that \(u_{i }\coloneqq s_{\cX}(x_{i })
    =
    r_{\cX}(h_{i } \HleftX y_{i })\), so that $(\xi_{i },h_{i })\cdot y_{i } = x_{i } (h_{i }\HleftX y_{i })\to z$. 
    
    We have $s_{\cX}(y_{i })\to s_{\cX}(y)$ and
    \[
        (h_{i }\HrightX y_{i })\HleftX s_{\cX}(y_{i })
        =
        s_{\cX}(h_{i }\HleftX y_{i })
        =
        s_{\cX}((\xi_{i },h_{i })\cdot y_{i }) 
        \to
        s_{\cX}(z).
    \]
    \assumption{Since $\HleftX$ is proper,}  this implies that (a subnet of) $\{h_{i }\HrightX y_{i }\}_{i\in I}$ converges in $\cH$; let $g$ be its limit. Note that
    \begin{align*}
        h_{i } \HleftX r_{\cX}(y_{i })
        &=
        h_{i } \HleftX (y_{i } y_{i }\inv)
        =
        (h_{i }\HleftX y_{i })\bigl[ (h_{i } \HrightX y_{i }) \HleftX y_{i }\bigr]
        &&\text{by \eqref{item:L4}}.
    \end{align*}
    If we multiply by $(h_{i }\HleftX y_{i })\inv$ on the left, we therefore get
     \begin{align}\label{eq:convergence to gy}
        (h_{i }\HleftX y_{i })\inv \, \bigl[h_{i } \HleftX r_{\cX}(y_{i })\bigr]
        &=
        (h_{i } \HrightX y_{i }) \HleftX y_{i }
        \to 
        g \HleftX y.
    \end{align}
    Since~$\cH$ leaves $\cX\z$ invariant (Lemma~\ref{lem:restricted action on cXz}), we have 
    \begin{align*}
        ( h_{i } \HleftX y_{i } )\inv 
        \bigl[h_{i } \HleftX r_{\cX} (y_{i })\bigr]
        =
        ( h_{i } \HleftX y_{i })\inv ,
    \end{align*}
    and so it follows from \eqref{eq:convergence to gy} that \(
         h_{i } \HleftX y_{i } 
         \to
         (g\HleftX y)\inv .\)
    Again, since  $y_{i }\to y$, \assumption{properness of $\HleftX$} now implies that (a subnet of) $h_{i }$ converges in $\cH$; let $h$ be its limit.
    Thus
    \[
        x_{i } =
        \bigl[x_{i } (h_{i } \HleftX y_{i })\bigr]
        (h_{i } \HleftX y_{i })\inv
        \to 
        z (h\HleftX y)\inv.
    \]
    We have shown that (a subnet of) $\{(x_{i }, h_{i })\}_{i\in I}$ converges, namely to $\bigl(z (h\HleftX y)\inv, h\bigr)$. We conclude that (a subnet of) $\{(\xi_{i },h_{i })\}_{i\in I}$ converges as well. This concludes our proof of properness.

\smallskip

We now want to verify that these two actions commute. Pick $(\xi,h)\in (\cX/\cG)\bowtie \cH$, $y\in \cX$, and $(t,\eta)\in \cG\bowtie (\cH\backslash \cX)$ with matching range resp.\ source and momentum maps.
Let $x$ be the unique element in $\xi$   such that $s_{\cX}(x)=r_{\cX}(h\HleftX y)$. We want to argue that we can chose a particular representative of $\eta$. We compute
\begin{align*}
    \mathfrak{s}(x(h\HleftX y)) &=\cH\HleftX s_{\cX}(x(h\HleftX y)) && \text{(def'n of $\mathfrak{s}$)}\\
    &= \cH\HleftX s_{\cX}(h\HleftX y) \\
    &= \cH\HleftX [(h\HrightX y) \HleftX s_{\cX}(y)]   && \text{(by \eqref{item:L10})}\\
    &= \cH\HleftX s_{\cX}(y) = r(t,\eta) && \text{(def'n of $\cG\bowtie (\cH\backslash\cX)$)}.
\end{align*}
This shows that $(t,\eta) $ can act on the right of $x(h\HleftX y)$. Our previous explanation now implies  that there exists a unique representative $z\in \eta$ which has range  equal to $s_{\cX}([x(h\HleftX y)]\XrightG t)$.
This choice of $x$ and $z$ makes the following computation particularly easy:
\begin{align*}
    \bigl[(\xi,h) \cdot y\bigr]\cdot (t,\eta) &= [x(h\HleftX y)]\cdot (t,\eta)
    =\bigl[\bigl(x(h\HleftX y)\bigr)\XrightG t\bigr]z
    \\
     % \bigl[(\xi,h) \cdot y\bigr]\cdot (t,\eta)
     &=\bigl[x\XrightG \bigl((h\HleftX y)\XleftG t\bigr)\bigr] ((h\HleftX y)\XrightG t) z
    &&\text{(by \eqref{item:R4})}
    \\
     &
     =\bigl[x\XrightG \bigl(y\XleftG t\bigr)\bigr] (h\HleftX (y\XrightG t)) z
    &&
    \text{(by \ref{cond.compatible.c3} and \ref{cond.compatible.c1})}.
\end{align*}

On the other hand,
let $z'\in \eta$ satisfy $r_{\cX}(z')=s_{\cX}(y\XrightG t)$, and let $x'\in \xi$ be the unique element such that $s_{\cX}(x')=r_{\cX}(h\HleftX((y\XrightG t) z'))$. Then
\begin{align*}
    (\xi,h) \cdot \bigl[y\cdot (t,\eta)\bigr] &= (\xi,h) \cdot [(y\XrightG t)z'] &&\text{(choice of $z'$)}\\
    &=x'\bigl[h\HleftX((y\XrightG t)z')\bigr]&&\text{(choice of $x'$)} \\
    &=x'[h\HleftX y\XrightG t]
    \bigl[(h \HrightX (y\XrightG t) 
    )\HleftX z'\bigr]&& \text{(by \eqref{item:L4})} \\
    &
    =x'[h\HleftX y\XrightG t]
    \bigl[(h
    \HrightX y
    )\HleftX z'\bigr]
    && 
    \text{(by \ref{cond.compatible.c3})}.
\end{align*}
Thus,
to prove that $\bigl[(\xi,h) \cdot y\bigr]\cdot (t,\eta)=(\xi,h) \cdot \bigl[y\cdot (t,\eta)\bigr]$,
it suffices to show that
\[
x\XrightG \bigl(y\XleftG t\bigr)
= x'
\quad\text{and}\quad
z=
    (h\HrightX y)
    \HleftX z'.
\]
For the right equation, we compute the range of the right-hand side as
\begin{align*}
    r_{\cX}\bigl(
    (h\HrightX y)
    \HleftX z'\bigr) 
    &= (h\HrightX y)\HleftX r_{\cX}(z') &&\text{(by \eqref{item:L10})} \\
    &= (h\HrightX y)\HleftX s_{\cX}(y\XrightG t)&&\text{(choice of $z'$)} \\
    &= (h\HrightX y)\HleftX [s_{\cX}(y) \XrightG t] &&\text{(by \eqref{item:R10})}\\
    &= s_{\cX}(h\HleftX y)\XrightG t &&\text{(by \eqref{item:L10}).} 
\intertext{On the other hand, }
    r_{\cX}(z) &= s_{\cX}\bigl([x(h\HleftX y)]\XrightG t\bigr)&&\text{(choice of $z$)} \\
    &= s_{\cX}(x(h\HleftX y)) \XrightG t &&\text{(by \eqref{item:L10})} \\
    &= s_{\cX}(h\HleftX y)\XrightG t.
\end{align*}
Both combined yield
\begin{equation*}
 r_{\cX}\bigl((h\HleftX y\XrightG t) \HleftX z'\bigr) =r_{\cX}(z) .
\end{equation*}
Since $\cH\HleftX z=\cH\HleftX z'=
H\HleftX \bigl(
    (h\HrightX y)
    \HleftX z'\bigr)$ \assumption{and since the actions are free,} it follows from  Lemma~\ref{lm.uniqueness} 
that $z=
    (h\HrightX y)
    \HleftX z'$. A similar argument shows that $x'=x\XrightG 
    (y\XleftG t)
    $. 
    We proved that the left-$(\cX/\cG)\bowtie \cH$ and the right-$\cG\bowtie (\cH\backslash \cX)$ actions on~$\cX$ commute. 
\end{proof}

We now prove the first main result (Theorem~\ref{thm.groupoid.eq}) which states that~$\cX$ is a $(\cX/\cG)\bowtie \cH \sme \cG\bowtie (\cH\backslash \cX)$-equivalence.

\begin{proof}[Proof of Theorem~\ref{thm.groupoid.eq}] According to Proposition~\ref{prop.actions.X}, we have commuting free and proper left $(\cX/\cG)\bowtie \cH$- and right $\cG\bowtie (\cH\backslash \cX)$-actions on $\cX$. It remains to show that $\mathfrak{r}$ induces a homeomorphism $\tilde{\mathfrak{r}}$
between $\cX/(\cG\bowtie (\cH\backslash \cX))$ and $((\cX/\cG)\bowtie \cH)\z $; a similar proof will then show that $\mathfrak{s}$ induces an analogous homeomorphism.

Fix $y\in \cX$ and consider any $(t,\eta)\in \cG\bowtie (\cH\backslash \cX)$ 
with $\mathfrak{s}(y)=r(t,\eta)$. Let $z\in \eta$ be the unique element such that
$r_{\cX}(z)=s_{\cX}(y\XrightG t
)$, so that by definition of the right-$\cG\bowtie (\cH\backslash \cX)$-action, 
$y\cdot (s,\eta)=(y\XrightG t)z$.
Consider its range in $\cX$:
\begin{align*}
    r_{\cX}(y\cdot (t,\eta)) &= r_{\cX}((y\XrightG t)z) =r_{\cX}(y\XrightG t) \\
    &=r_{\cX}(y)\XrightG(y\XleftG t) \in r_{\cX}(y)\XrightG G.
\end{align*}
Therefore, if we write $\overline{y}$ for the equivalence class of $y$ in $\cX/(\cG\bowtie (\cH\backslash \cX))$, then
$\tilde{\mathfrak{r}} (\overline{y})
=
r_{\cX}(y)\XrightG \cG$ is well defined. Surjectivitiy, continuity, and opennes of $\tilde{\mathfrak{r}}$ is trivial, since $\mathfrak{r}$ is surjective, continuous, and open.
To see that $\tilde{\mathfrak{r}}$ 
is injective, take any $y,y'$ with $r_{\cX}(y)\XrightG \cG=r_{\cX}(y')\XrightG \cG$; we need to find $t
\in \cG$ and  $\eta\in \cH\backslash\cX$ 
such that $y\cdot (s,\eta
)=y'$.  By assumption, 
there exists $s
\in \cG$ such that $r_{\cX}(y')=r_{\cX}(y)\XrightG s
$. Set $t
=y\inv \XleftG s
$ 
Then $s
=y\XleftG t
$ and thus by \eqref{item:R10},
\[r_{\cX}(y')=r_{\cX}(y)\XrightG (y\XleftG  t
) = r_{\cX}(y\XrightG t
).\]
Since $y'$ and $y\XrightG  t
$ have the same range in $\cX$, we may let $x=(y\XrightG  t
)\inv  y'\in \cX$, so that $y'=(y\XrightG  t
)x$, i.e., $y'=y\cdot ( t
,\cH\HleftX x)$.
\end{proof}

\begin{remark}
    Let us briefly recap which topological assumption in Theorem~\ref{thm.groupoid.eq} was needed for which part of the proof. We required the source map of $\cH$ to be open in order for the quotient map $q\colon \cX\to\cH\backslash \cH$ to be open which, in turn,  we used to show that the momentum map $\tilde{\sigma}$ 
    of the $\cG$-action on $\cH\backslash\cX$ is continuous (see  proof of Proposition~\ref{prop.ssp actions on quotients}). Freeness of the $\cH$-action on $\cX$ allowed us to turn $\cH\backslash\cX$ into a groupoid (Lemma~\ref{lm.orbit multiplication}), and  its properness plus openness of $q$ was needed to make $\cH\backslash\cX$ a \LCH\ groupoid (Proposition~\ref{prop.gpd structure on orbit space}). Lastly, the source map of $\cX$ was required to be open in order to prove that the momentum map $\mathfrak{s}$ of the right $\cG\bowtie (\cH\backslash \cX)$-action on $\cX$  (Proposition~\ref{prop.actions.X}) is open and can therefore induce a homeomorphism of the quotient by the right $(\cX/\cG)\bowtie \cH$-action onto the unit space of $\cG\bowtie (\cH\backslash \cX)$.
\end{remark}

\begin{corollary}[cf.\ {\cite[Proposition 2.47]{Wil2019}}]\label{cor.gprd.one.sided} Suppose $\cH$ and $\cX$ are groupoids and that $\cH$ has a \ssla\ on $\cX$ that is free and proper. If $s_{\cX}$ and $s_{\cH}$ are open maps, then the groupoids $\cX\bowtie \cH$ and $\cH\backslash \cX$ are equivalent.
\end{corollary}

\begin{proof}
    We have seen in Example~\ref{ex:ssla of trivial gp and gpd:sometimes a something} that $\cX$ is a $(\cH,\{e\})$-\sspa\ (modulo switching the roles of $\cH$ and $\cG$).
    Theorem~\ref{thm.groupoid.eq} thus implies that $(\cX/\{e\})\bowtie \cH$ and $\{e\}\bowtie (\cH\backslash \cX)$ are equivalent groupoids.
    By Examples~\ref{ex:ssla of trivial group and groupoid:bowtie} and~\ref{ex:ssla of trivial group and groupoid:quotient}, we have $\{e\}\bowtie (\cH\backslash \cX) \cong \cH\backslash \cX$ and $(\cX/\{e\})\bowtie \cH \cong \cX\bowtie \cH$, respectively. The claim now follows.
\end{proof}

\subsection{%
Applications of Theorem~\ref{thm.groupoid.eq}
}%
\label{ssec:Applications of thm.groupoid.eq}

\begin{example}[%
    continuation of Examples~\ref{ex.example1 - part 1} and~\ref{ex.example1 - part 2}%
    ]\label{ex.example1 - part 3}     
    Suppose again that a \LCH\ group $K=G\bowtie H$ acts on the left on a \LCH\ space $X$, denoted by $\ast$. We let $\cX=G\ltimes X$ be the transformation groupoid, and we define the \ssla\ $\HleftX$ and $\HrightX$ of $H$ on $\cX$ as in~\eqref{eq:actions in example1}.
    We assume that $\ast$ is free and proper, so that $\HleftX$ and $\HrightX$ are free and proper by our computations in Example~\ref{ex.example1 - part 2}. Thus, by
    Corollary~\ref{cor.gprd.one.sided}, we 
    get that $\cX\bowtie H$ is equivalent to $H\backslash \cX$. (Here, the assumption that the source maps are open is trivially satisfied: the source map of $H$ is constant and the source map of $X$ is the identity map.)

    Note that the map 
    \[
        \phi\colon (G\ltimes X)\bowtie H \to (G\bowtie H) \ltimes X,
        \quad
        ((t,x),h) \mapsto ((t,h), h\inv \ast x)
        ,%.
    \]
    is a groupoid isomorphism $\cX\bowtie H \cong K\ltimes X$. Indeed, using
    the definition of $\bowtie$ in $\cX\bowtie H$, we compute
    the product of two elements of the domain to be
    \[
        \bigl((t,x),h\bigr)\, \bigl((s,y),k\bigr)
        =
        \bigl(
            (t,x)[h\HleftX (s,y)],
            [h\HrightX (s,y)]k
        \bigr)
        =
        \bigl(
            (t,x)(h\cdot s,h|_{s}\ast y),
            h|_{s}k
        \bigr).
    \]
    Of the tuple on the far right-hand side, the first component is a product in $\cX$; it
    is defined if and only if the source of $(t,x)$ equals the range of $(h\cdot s,h|_{s}\ast y)$.
    In other words, we must have
    $x= (h\cdot s)\ast[h|_{s}\ast y]=[(h\cdot s )(h|_{s})]\ast y= [hs]\ast y$, in which case their product is $(t[h\cdot s],h|_{s}\ast y)$. Therefore,
    the composition in $\cX\bowtie H$ can be described succinctly as follows:
    \[
        \bigl((t,[hs]\ast y),h\bigr)\, \bigl((s,y),k\bigr)
        =
        \bigl(
            (t[h\cdot s],h|_{s}\ast y),
            h|_{s}k
        \bigr)
        .
    \]
    Applying $\phi$, we end up with
    \[
        \phi
        \Bigl(
            \bigl(
                (t,[hs]\ast y),h
            \bigr)\,
            \bigl((s,y),k\bigr)
        \Bigr)
        =
        \bigl(
            (t[h\cdot s],h|_{s}k),
            [h|_{s}k]\inv \ast [h|_{s}\ast y]
        \bigr)
        =
        \bigl(
            (t[h\cdot s],h|_{s}k),
            k\inv \ast y
        \bigr)
        .
    \]
    On the other
    hand, 
    the product of $\phi((t, h), x)$ with $\phi((s, k), y)$ in the codomain $K\ltimes X$  is defined if and only if the source $h\inv \ast x$ of $((t, h),h\inv x)$ equals the range $(s,k)\ast (k\inv \ast y)$ of $((s, k),k\inv \ast y)$.
    In other words, we get the same necessary condition for composability as above, namely that
    $x=[hs]\ast y$, in which case
    \[
        \phi \bigl((t,[hs]\ast y),h\bigr)\, \phi\bigl((s,y),k\bigr)
        =
        \bigl( (t,h)(s,k), k\inv \ast y\bigr)
        .
    \]
    In $K$, we have $(t,h)(s,k)=(t[h\cdot s], h|_{s}k)$, which shows that indeed 
    \[
        \phi \bigl((t,x),h\bigr)\, \bigl((s,y),k\bigr)
        =
        \phi \bigl((t,x),h\bigr)\, \phi\bigl((s,y),k\bigr).
    \]
    \end{example}
The setup in Example~\ref{ex.example1 - part 3} arises abundantly in group dynamics. 

\begin{example}[First special case of Example~\ref{ex.example1 - part 3}]\label{ex.example1 a)}
    In the above example, suppose that $G=\{e\}$, so $K=H$ and $\cX=X$ is a trivial groupoid (i.e., a space). The action $\HrightX$ is now trivial and the action $\HleftX$ is exactly the action $\ast$ of $K$ on $X$ that we started with. If $\ast$ is free and proper, Example~\ref{ex.example1 - part 3} shows that the transformation groupoid $K\ltimes X$ is equivalent to $K\backslash X$. Note that the trivial groupoid $K\backslash X$ always admits a Haar system (see, for example, \cite[Example 1.22]{Wil2019}). Assuming that the two groupoids are second countable, $K\ltimes X$ therefore also admits a Haar system by \cite[Theorem 2.1]{Wil:Haar}. We may now apply \cite[Theorem 2.8]{MRW:Grpd}, which states that the C*-algebras of  equivalent groupoids with Haar systems are Morita equivalent. In other words, we exhibit the known result that the crossed product $C_0(X)\rtimes K$ is Morita equivalent to $C_0(K\backslash X)$.
\end{example}

The following 
is a concrete example using a finite group $K$. 

\begin{example}[Second special case of Example~\ref{ex.example1 - part 3}]\label{ex.example.concrete}
Consider the symmetric group $S_4$, which is a group of order $24$, 
and the elements
\[
a=\left(\begin{matrix}1 & 2 & 3\end{matrix}\right)
\quad\text{and}\quad
r=\left(\begin{matrix}1 & 2 & 3 & 4\end{matrix}\right),
f=\left(\begin{matrix}1 & 3\end{matrix}\right).
\]
Let
$G=\langle a
\rangle$ and $H=\langle r
, f
\rangle$;  
one can verify that $G$ and $H$ are 
of order $3$ and $8$ respectively, 
that neither subgroup is normal,
and that $G\cong C_3$ and $H\cong D_4$. 

Since $|S_4|=|G|\cdot |H|$ and $|G\cap H|=1$, we must have $S_4=G\cdot H$, i.e., each element in $S_4$ is a unique product of the form $t h$ for $t\in G$ and $h\in H$. In other words, $S_{4}=K$ is the internal \ZS\ product of $G$ and $H$, and in particular, we get \ZS\ actions $G \arrowlssa H$ in such a way that any product $h t$ of $h\in H$ and $t\in G$ in $S_{4}$ can be uniquely decomposed as 
\[h t=(h\cdot t)(h|_{t})\]
where $h\cdot t\in G$ and $h|_{t}\in H$. 
Tables~\ref{table.ZS.S4_action} and \ref{table.ZS.S4_restr} contains an overview of these actions.

\begin{table}[ht]
    \begin{minipage}{.49\linewidth}
          \caption{\label{table.ZS.S4_action} Action map $h\cdot t$ on $S_4$}\centering
            \begin{tabular}{c|ccc} 
                \diagbox{$h$}{$t$} &    $e$   &   $a$    &   $a^2$    \\ 
                                  \hline
             {$e$} & $e$   &  $a$   &  $a^2$      \\
             {$r$} &  $e$  &  $a^2$ &  $a$   \\
             {$r^2$} & $e$ &  $a$   &  $a^2$   \\
             {$r^3$} & $e$ &  $a^2$ &  $a$     \\
             {$f$} & $e$   &  $a^2$ &  $a$      \\
             {$rf$} & $e$  &  $a$   &  $a^2$      \\
             {$r^2f$} & $e$&  $a^2$ &  $a$       \\
             {$r^3f$} & $e$&  $a$   &  $a^2$    
            \end{tabular}
    \end{minipage}
    \begin{minipage}{.49\linewidth}
          \caption{\label{table.ZS.S4_restr} Restriction map $h|_{t}$ on $S_4$}\centering
            \begin{tabular}{c|ccc} 
            \diagbox{$h$}{$t$}      &   $e$    &    $a$   &  $a^2$     \\\hline {$e$}   & $e$    & $e$    & $e$      \\{$r$}   & $r$    & $r^2f$ & $r^3$     \\{$r^2$} & $r^2$  & $rf$   & $r^3f$     \\{$r^3$} & $r^3$  & $r$    & $r^2f$     \\{$f$}   & $f$    & $f$    & $f$     \\{$rf$}  & $rf$   & $r^3f$ & $r^2$     \\{$r^2f$}& $r^2f$ & $r^3$  & $r$     \\{$r^3f$}& $r^3f$ & $r^2$  & $rf$     
            \end{tabular}
    \end{minipage}
\end{table}

Now let $X=S_4$ and we let $S_4$ act on $X$ by left translation,
so that $K\ltimes X = S_{4}\tensor[_{\operatorname{lt}}]{\ltimes}{} S_{4}$. One
can explicitly write out all the orbits in $H\backslash (G\ltimes X)$,
and verify that the nine elements in $G\ltimes G\subseteq G\ltimes X$ are in different $H$-orbits. Since $|H\backslash (G\ltimes X)|=9$, we have $H\backslash (G\ltimes X)\cong G\ltimes G$. 
By
    Example~\ref{ex.example1 - part 3},
we conclude that the groupoids $S_{4}\tensor[_{\operatorname{lt}}]{\ltimes}{} S_{4}$ and $G\ltimes G$ are equivalent. By 
the Stone--von Neumann Theorem,
their groupoid C*-algebras 
are given by
$\mathcal{K}(\ell^2(S_4))\cong M_{24}(\mathbb{C})$ and $\mathcal{K}(\ell^2(G))\cong M_{3}(\mathbb{C})$. Consequently, these C*-algebras are Morita equivalent.
\end{example}

\begin{example}[continuation of Example~\ref{ex:skew product}]\label{ex:skew product, part 2}
    Suppose again that $\mathbf{c}\colon \cG\to H$ is a continuous homomorphism from a groupoid to a group. In Example~\ref{ex:skew product}, we described a \ssla\ of $H$ on the skew-product groupoid $\cG(\mathbf{c})$. This action is free and proper. 
    Note that $s_{\cG}$ is  open  if and only if $s_{\cG(\mathbf{c})}$ is  open, in which case it follows from Corollary~\ref{cor.gprd.one.sided} that $\cG(\mathbf{c})\bowtie H$ is equivalent to $H\backslash\cG(\mathbf{c})\cong\cG$.
\end{example}

\subsection{Haar Systems on quotients}

To construct a right Haar systems on $\cX/\cG$ out of a right
Haar system on~$\cX$, 
we again require $\XrightG$-invariance. 

\begin{lemma}[cf.\ {\cite[Prop.\ A.10]{KMQW2013}}]\label{lem.orbit.Haar.right}
    Suppose $\cG$ and $\cX$ are \LCH\ groupoids, that $\cG$ has a \assumption{free and proper} \ssra\ on $\cX$, and that $\cX$ has a $\XrightG$-invariant right Haar system $\{\lambda_u\}_{u\in\cX\z}$ (Definition~\ref{df.lambda G-invariance}). 
    Then there exists a right Haar system $\{\kappa_{u\XrightG \cG}\}_{u\XrightG \cG %\in \cX\XrightG \cG
    }$ on $\cX/\cG$ given for any $\widehat{f}\in C_c(\cX/\cG)$ by
\[\int \widehat{f}(x\XrightG G) \dif\kappa_{u\XrightG G}(x\XrightG G)
=\int_{\cX} \widehat{f}(x\XrightG G)\dif\lambda_u(x).
\]
\end{lemma}

\begin{proof}
    The argument is verbatim as in the proof of \cite[Prop.\ A.10]{KMQW2013}, only that the range map of $\cX$ has to be replaced by its source map. To be precise, we will invoke \cite[Lemma 1.3]{Renault1987Fr} for $(X,Y,G,\pi)=(\cX,\cX\z, \cG,s_{\cX})$. Since we assumed
    \assumption{%
    $\XrightG$ to be free and proper,
    }%
    $\cX$ is a principal $\cG$-space.
    \assumption{% 
    Since $\cX$ is assumed to have a Haar system,
    }%
    its continuous source map $s_{\cX}
    $ is open \cite[Prop.\ 1.23]{Wil2019}. It is furthermore equivariant by \eqref{item:R10}, so that we may apply \cite[Lemma 1.3]{Renault1987Fr}. The given formula for $\kappa$ is hence a system for the map $\cX/\cG\to \cX\z/\cG$, $x\XrightG \cG\mapsto s_{\cX}(x)\XrightG \cG$, which is the source map of the groupoid $\cX/\cG$. In other words, $\kappa$ is a right Haar system for $\cX/\cG$.
\end{proof}

\begin{lemma}\label{lm.quotient Haar system is invariant}
Suppose $\cX,\cH,\cG$ are \LCH\ groupoids and that $\cX$ has a left $\cH$-action $\HleftX$ and a \assumption{free and proper} right $\cG$-action $\XrightG$.
Assume $\{\lambda_u\}_{u\in\cX\z}$ is a $\XrightG$-invariant right Haar system on $\cX$  (Definition~\ref{df.lambda G-invariance}), and let $\{\kappa_{u\XrightG \cG}\}_{u\XrightG \cG}$ be the induced right Haar system on $\cX/\cG$ (Lemma~\ref{lem.orbit.Haar.right}).
If the left Haar system $\{\lambda^{u}\}_{u\in\cX\z}$ on $\cX$ defined by $\lambda^{u}(E)=\lambda_u(E\inv )$ is $\HleftX$-invariant   (Definition~\ref{df.lambda H-invariance}), then the left Haar system $\{\kappa^{u\XrightG \cG}\}_{u\XrightG \cG}$ on $\cX/\cG$ associated to $\{\kappa_{u\XrightG \cG}\}_{u\XrightG \cG}$ is $\qHleftX$-invariant.
\end{lemma}

\begin{proof}
    The computation is straight forward: on the one hand,
    \begin{align*}
        \kappa^{u\XrightG \cG} (h\inv\qHleftX [E\XrightG \cG])
        &=
        \kappa^{u\XrightG \cG} ([h\inv\HleftX E]\XrightG \cG)
        &&\text{(def'n of $\qHleftX$)}
        \\
        &=
        \kappa_{u\XrightG \cG} \bigl([h\inv\HleftX E]\inv\XrightG \cG\bigr)
        &&\text{(def'n of $\kappa^{u\XrightG \cG}$ and of ${}\inv$ on $\cX/\cG$)}
        \\
        &=
        \lambda_{u} ([h\inv\HleftX E]\inv)
        &&\text{(def'n of $\kappa_{u\XrightG \cG}$)}
        \\
        &=
        \lambda^{u} (h\inv\HleftX E)
        &&\text{(def'n of $\lambda^{u}$)}
        \\
        &=
        \lambda^{h\HleftX u} ( E)
        &&\text{($\HleftX$-invariance of $\lambda^{u}$).}
\intertext{On the other hand,}
        \kappa^{h\qHleftX[u\XrightG \cG]} (E\XrightG \cG)
        &=
        \kappa_{h\qHleftX[u\XrightG \cG]} ([E\XrightG \cG]\inv)
        &&\text{(def'n of $\kappa^{h\qHleftX[u\XrightG \cG]}$)}
        \\
        &=\kappa_{[ h\qHleftX u]\XrightG \cG} (E\inv\XrightG \cG)
        &&\text{(def'n  of $\qHleftX$ and of ${}\inv$ on $\cX/\cG$)}
        \\
        &=\lambda_{h\HleftX u} ( E\inv)
        &&\text{(def'n of $\kappa_{[ h\qHleftX u]\XrightG \cG}$)}
        \\
        &=\lambda^{h\HleftX u} ( E)
        &&\text{(def'n of $\lambda^{h\HleftX u}$).}
    \end{align*}
    This shows that $\kappa^{u\XrightG \cG} (h\inv\qHleftX [E\XrightG \cG])=\kappa^{h\qHleftX[u\XrightG \cG]} (E\XrightG \cG)$.
\end{proof}

\begin{corollary}\label{cor.gpd equiv plus Haar makes SME}
    Suppose $\cG,\cH,\cX$ are as in  Theorem~\ref{thm.groupoid.eq}. 
    \repeatable{Assumption on Haar in cor.gpd equiv plus Haar makes SME}{%
         Assume that $\cX$ has a $\HleftX$-invariant left Haar system  %$\{\lambda^{u}\}_{u\in\cX\z}$
         (Definition~\ref{df.lambda H-invariance}) whose associated right Haar system 
         %$\{\lambda_u\}_{u\in\cX\z}$ 
         is $\XrightG$-invariant%
         %  (Definition~\ref{df.lambda G-invariance})%
    }%
    . If $\cH$ and $\cG$ have Haar systems, then so do $(\cX/\cG)\bowtie \cH$ and $\cG\bowtie (\cH\backslash \cX)$, and so their C*-algebras  $\textrm{C}^*((\cX/\cG)\bowtie \cH)$ and $\textrm{C}^*(\cG\bowtie (\cH\backslash \cX))$ are Morita equivalent.
\end{corollary}

In Corollary~\ref{cor.Fell bdl equiv plus Haar makes SME}, we will generalize the above result to Fell bundle C*-algebras.

\begin{proof}
    Since $\cG$ acts 
    \assumption{%
    properly and freely
    }%
    on $\cX$, it follows from Lemma~\ref{lem.orbit.Haar.right} that the right Haar system on $\cX$ induces a right Haar system on $\cX/\cG$. By Lemma~\ref{lm.quotient Haar system is invariant}, the associated left Haar system is $\qHleftX$-invariant. 
    It follows from Proposition~\ref{prop.ZSProduct.Haar.left} that $(\cX/\cG)\bowtie\cH$ has a Haar system. Since our assumptions are symmetric, we likewise get a Haar system on $\cG\bowtie(\cH\backslash\cX)$. As the two groupoids are equivalent by Theorem~\ref{thm.groupoid.eq} and have Haar systems, it follows from \cite[Theorem 2.8]{MRW:Grpd} that their C*-algebras are Morita equivalent.
\end{proof}

\section{\Ss\ actions on Fell bundles}\label{sec.ssa.Fell} 

Fell bundles were originally introduced by Fell as ``C*-algebraic bundles'' \cite{FellBundleBook}; they are 
a powerful tool to study C*-algebras graded by groups or groupoids, and many C*-algebras can be realized as Fell bundle C*-algebras. One may refer to \cite{ExelFellBundle, Kumjian1998, Yamagami1990} for a more detailed discussion on the subject;
 in Subsection~\ref{ssec:Fell bdl properties}, the reader can find the definition that we are going to be using.

\subsection{\Ss\ left actions on Fell bundles}

We will now extend the notion of \ss\ actions to Fell bundles. 
Similar to the construction 
 of a \ZS\ product Fell bundle
 in~\cite{DL2021}, this
    will allow
 us to construct a 
 \ssp\
 Fell bundle.

\begin{definition}\label{df.ss.FellBundle.action.left}
Let 
	$\cB=(q_{\cB}\colon B\to\cX)$
	be a Fell bundle.
Suppose~$\cH$ has a left \ss\ action on~$\cX$ with
 momentum map $\rho_{\cX}\colon \cX\to \cH\z$.
Define $\rho_{\cB}=\rho_{\cX}\circ q_{\cB}$ and let 
\[\cH 
\bfp{s}{\rho} 
\cB=\{(h,b)
\in \cH \times \cB
: s_{\cH}(h)=\rho_{\cB}(b)\}\]
be equipped with the subspace topology of $\cH\times \cB$. 
A 
\em
left \ss~$\cH$-action on~$\cB$
is a continuous map \[\mvisiblespace\HleftB\mvisiblespace \colon \cH \bfp{s%_{\cH}
}{\rho%_{\cB}
} \cB\to \cB\]
satisfying the following conditions:
\repeatable{lssa on Fell bdl}{
\begin{enumerate}[label=\textup{(B\arabic*)}]
    \item\label{item:leftFell.1} For any $(h,x)\in \cH 
        \bfp{s}{\rho}
    \cX$, the map $h\HleftB\mvisiblespace$ maps $\cB_x$ into $\cB_{h\HleftX x}$ and is linear.
    \item\label{item:leftFell.2} For any 
   	 $(k,h)\in\cH^{(2)}$, 
        we have $k\HleftB (h \HleftB \mvisiblespace)=(kh)\HleftB \mvisiblespace$.
    \item\label{item:leftFell.3} For any $u\in \cH\z $, the map $u \HleftB\mvisiblespace$ is the identity.
    \item\label{item:leftFell.4} For any $(b,c)\in \cB^{(2)}$ such that $(h,bc)\in \cH \bfp{s%_{\cH}
    }{\rho%_{\cB}
    } \cB$, we have
    \[h\HleftB (bc)=(h\HleftB b)\left[ (h\HrightX q_{\cB}(b) )\HleftB c\right].\]
    \item\label{item:leftFell.5} For any 
    	$(h,b)\in \cH \bfp{s%_{\cH}
    	}{\rho%_{\cB}
    	} \cB$,  we have
    \[(h\HleftB b)^* = [h\HrightX q_{\cB}(b) ]\HleftB b^*.\]
\end{enumerate}
}
\end{definition}

Writing $h\HrightB b\coloneqq h\HrightX q_{\cB}(b)\in\cH$ for $(h,b)\in \cH \bfp{s%_{\cH}
}{\rho%_{\cB}
} \cB$  highlights the similarities between the above definition and Definition~\ref{df.left.selfsimilar}; compare
\eqref{item:L4} and \eqref{item:L9} on the left to \ref{item:leftFell.4} resp.\ \ref{item:leftFell.5} on the right: 
    \begin{align*}
        h\HleftX (xy)&=(h\HleftX x)[(h\HrightX x)\HleftX y]
        &\text{ versus }&&
         h\HleftB (bc)&=(h\HleftB b)\left[ (h\HrightB b )\HleftB c\right],
        \\
        (h\HleftX x)\inv  &= (h\HrightX x) \HleftX x\inv
        &\text{ versus }&&
        (h\HleftB b)^* &= [h\HrightB b ]\HleftB b^*
        .
    \end{align*}

\begin{remark} When~$\cX$ and~$\cH$ form a matched pair of groupoids, Definition~\ref{df.ss.FellBundle.action.left} coincides with the notion of 
a
$(\cX,\cH)$-compatible~$\cH$-action \cite[Definition 3.1]{DL2021}. 
\end{remark}

In general, we saw in Proposition~\ref{prop.ss.to.matched.left} that the
    \ssp\ groupoid
$\cX\bowtie \cH$ is isomorphic to the 
 \ZS\ product groupoid $\cX\bowtie\widetilde{\cH}$.
The next proposition proves that 
a similar result holds in the realm of Fell bundles. 

\begin{proposition}\label{prop:Htilde}
    \repeatable{ssla on Fell bundle}{%
    Suppose $\cH$ has a \ssla\ $\HleftB$ on a Fell bundle $\cB=(q_{\cB}\colon B\to \cX)$
    }%
and write
$r_{\cB}=r_{\cX}\circ q_{\cB}$.
    Let $\widetilde{\cH}=\{(u,h)\in \cX\z \times \cH: \rho_{\cX}(u)=r_{\cH}(h)\}$ be the transformation groupoid of the $\cH$-action on $\cX\z$
with 
source map
given by $s_{\widetilde{\cH}}(u,h)=h\inv \HleftX u$.
Let 
\[\wtbeta \colon \widetilde{\cH} 
        \bfp{s}{r}
    \cB \to \cB
\quad\text{
be defined
by }\quad
\wtbeta ((u,h),b)=h\HleftB b.\]
Then $\wtbeta $ is a $(\cX,\widetilde{\cH})$-compatible $\widetilde{\cH}$-action on~$\cB$ in the sense of \cite[Definition 3.1]{DL2021}.
\end{proposition}

\begin{proof} To see that $\wtbeta $ is well defined, take $(u,h)\in\widetilde{\cH}$ and $b\in\cB_x$ with $s_{\widetilde{\cH}}(u,h)=r_{\cX}(x)$. Since $s_{\widetilde{\cH}}(u,h)=h\inv \HleftX u$, we have 
\[\rho_{\cX}(x)=\rho_{\cX}\z(r_{\cX}(x)) = \rho_{\cX}\z(h\inv \HleftX u) \overset{\eqref{item:L1}}{=} r_{\cH}(h\inv )=s_{\cH}(h).\]
Therefore, $(h,b)\in \cH \bfp{s%_{\cH}
}{\rho%_{\cB}
} \cB$ and $\wtbeta $ is well defined. It is routine to check that $\wtbeta $ is indeed an $(\cX,\widetilde{\cH})$-compatible $\widetilde{\cH}$-action on~$\cB$. 
\end{proof}

One immediate consequence is that, 
fiberwise, $\HleftB$
  shares all the nice properties
of
an $(\cX,\cH)$-compatible action. For example, 
the following is a consequence of  \cite[Corollary 3.3]{DL2021}:

\begin{corollary}\label{cor.isometric.action} For each $h\in \cH$ and $x\in \cX$ with $s_{\cH}(h)=\rho_{\cX}(x)$, the map $h\HleftB\mvisiblespace \colon \cB_x\to \cB_{h\HleftX x}$ is isometric. 
\end{corollary}

\subsection{The \ssp\ Fell bundle}

The \ZS\ product Fell bundle was first defined in~\cite[Theorem 3.8]{DL2021} under 
the assumption
 that
\begin{enumerate}
    \item 
    the underlying groupoids~$\cX$ and~$\cH$ form a matched pair and
    \item the underlying groupoids are \etale. 
\end{enumerate}
Inspired by the construction of semi-crossed product Fell bundles in~\cite{KMQW2010}, we now define a similar construction with these two assumptions removed. 
To be precise,
we aim to define the product Fell bundle from a \ss~$\cH$-action on a Fell bundle~$\cB$, where the underlying groupoids are \LCH. 

\begin{definition}\label{df.ss.product.bundle.left}
\txtrepeat{ssla on Fell bundle}
% Suppose we are given a left \ss\ $\cH$-action on $\cB=(q_{\cB}\colon B\to\cX)$ as in 
(Definition~\ref{df.ss.FellBundle.action.left}).
Define the {\em (left) \ssp\ Fell bundle} $\cB\BbowtieH  \cH$ to 
have the total space
\[
B\BbowtieH  \cH
%\cC
=
B
\bfp{\rho_{\cX} \circ s_{\cX}\circ q_{\cB}}{r_{\cH}} \cH = \{(b,h)\in B\times\cH: (q_{\cB}(b),h) \in \cX\bowtie \cH\}\]
with bundle projection
$q_{\cB\BbowtieH\cH}
(b,h)=(q_{\cB}(b),h)$, mapping 
    $B\BbowtieH  \cH$
to $\cX\bowtie \cH$. The fiber 
\[
\left(\cB\BbowtieH  \cH\right)_{(x,h)}=\{(b,h)\in B\BbowtieH  \cH: q_{\cB}(b)=x\}
\]
is equipped with the norm $\|(b,h)\|=\|b\|$.

    As always, let 
    \[(\cB\BbowtieH  \cH)^{(2)}
        \coloneqq(B\BbowtieH  \cH)
            \bfp{s_{\cB\BbowtieH  \cH}}{r_{\cB\BbowtieH  \cH}}
            (B\BbowtieH  \cH),
    \]
and define multiplication and involution by
\[(a,h)(b,k)=\bigl(a[h\HleftB b], [h\HrightB b ]k\bigr)
\quad\text{and}\quad
(b,h)^* = ( h\inv \HleftB b^*, h\inv  \HrightB 
b^*).\]
\end{definition}
We note that the proof that $\cB\BbowtieH  \cH$ is a Fell bundle over $\cX\bowtie \cH$ follows {\em mutatis mutandis} as in 
the proof in \cite[Section 3]{DL2021}.

\bigskip

For the first example, we will need a bit of notation.
\begin{notation}
    Let $\cA=(q_{\cA}\colon A \to \cK\z)$ be an \usc\
C*-bundle over the unit space $\cK\z$ of a groupoid $\cK $, and let $(\cA, \cK , \alpha)$ be a groupoid dynamical system (see \cite[Definition 4.1]{MW:Renaults} or \cite[Chapter 3]{Goehle:thesis} for more details).
    We let  $\cB(\cA,\cK ,\alpha)$ denote the Fell bundle associated to this dynamical system: as a set, it is given by $\cA\bfp{q}{r}\cK $ with bundle projection $q_{\cB}(a,k)=k$. The involution is given by
        \(
            (a,k)^* \coloneqq \left( \alpha_{k\inv} (a)^*,k\inv\right),
        \)
    and the product of two elements $(a_i,k_i)\in \cB(\cA,\cK ,\alpha)$ with $(k_1,k_2)\in \cK^{(2)}$ is given by 
        \[
            (a_1,k_1)\cdot (a_2,k_2)\coloneqq \left(a_{1}\alpha_{k_{1}}(a_{2}), k_{1}k_{2}\right).
        \]
    The C*-algebra of this Fell bundle is exactly the groupoid crossed product $\cA\rtimes_\alpha \cK $ \cite[Example 2.8]{MW2008}.
\end{notation}

\begin{example}[generalization of {\cite[Example  3.10]{DL2021}}] \label{ex:ZS of CP}
    Suppose $\cH$ has a \ssla\ on a groupoid $\cX$ and  $(\cA,\cX\bowtie\cH,\alpha)$ is a groupoid dynamical system. Let $\alpha|_{\cX}$ be the restriction of $\alpha$ to the subgroupoid $\cX$, i.e., $(\alpha|_{\cX})_{x}\coloneqq \alpha_{(x,\rho_{\cX}(x))}$. Then $\cH$ has a \ssla\ $\HleftB$ on $\cB=\cB(\cA,\cX,\alpha|_{\cX})$ defined for $h\in \cH$ and $(a,x)\in \cA\bfp{q}{r}\cX$ with $s_{\cH}(h)=\rho_{\cB}(a,x)=\rho_{\cX}(x)$  by 
    \[
        h\HleftB (a,x) \coloneqq 
        \left(\alpha_{(r(h),h)} (a), h\HleftX x\right).
    \]
    One can check that
    \[
        \cB(\cA, \cX\bowtie \cH,\alpha)
        \cong
        \cB(\cA, \cX,\alpha|_{\cX}) \BbowtieH \cH
         .
    \]
\end{example}

\begin{remark}\label{rm.saturated is inherited by bowtie}
    If $\cH$ has a \ssla\ $\HleftB$ on a Fell bundle $\cB$, then
    \[
        (\cB\BbowtieH \cH)_{(x,h)}
        \cdot
        (\cB\BbowtieH \cH)_{(y,k)}
        =
        (
        \cB_{x} 
        \times\{h\})
        \cdot
        (
        \cB_{y}
        \times\{k\} )
        \subseteq
        \cB_{x(h\HleftX y)}\times \{(h\HrightX y)k\}
        .
    \]
    Moreover, our assumptions on $\HleftB$ imply that $h\HleftB \cB_{y}=\cB_{h\HleftB y}$, rather than merely a containment of the left-hand side in the right-hand side. Thus, if $\cB$ is saturated (meaning that the
        closed linear span of the
    $\cB$-product of any $\cB_{x_{1}}$ with any compatible $\cB_{x_{2}}$ equals the entire $\cB_{x_{1}x_{2}}$), then by the above argument, we automatically have that $\cB\BbowtieH \cH$ is saturated also.
\end{remark}

Similar to the case of a \ssp\ groupoid (see Lemma~\ref{lem.ss.is.matched}), we can lift the action $\HleftB$ to a $\widetilde{\cH}$-action $\wtbeta$, where $(\cX,\widetilde{\cH})$ is a matched pair of groupoids.  When the groupoids $\cX$ and $\cH$ are \etale, this construction is closely related to the construction in \cite{DL2021} 
in the following sense.

\begin{proposition}
    If the groupoids~$\cX$ and~$\cH$ are \etale, then so is the groupoid $\widetilde{\cH}$ from Proposition~\ref{prop:Htilde}
    and the \ssp\ Fell bundle $\cB\BbowtieH  \cH$ is isomorphic to the \ZS\ product Fell bundle $\cB\bowtie_{\wtbeta } \widetilde{\cH}$ constructed in~\cite{DL2021}.
\end{proposition}

\begin{remark}\label{rm.df.ss.product.bundle.right}
    As always, a similar construction can be done on the other side: if~$\cB$ carries a {\em right} \ss\ $\cG$-action $\BrightG$, we can let $\cG\GbowtieB  \cB$ be given as the bundle with the total space
    \[
    \cG\GbowtieB B
    =
    \cG
        \bfp{s }{\sigma}
    B = \{(s,b)\in \cG\times B: (s,q_{\cB}(b)) \in \cG\bowtie \cX\}\]
    and the analogous Fell bundle structure.
\end{remark}

 \section{The orbit Fell bundle from \ss\ actions}\label{sec.orbit.Fell}

The following is analogous to the construction in \cite[Corollary A.12]{KMQW2013}.

\begin{definition}
    If $\cH$ is a groupoid and a topological space $B$ is a left $\cH$-space, where the action is denoted by $\HleftB$, we may let
$ \cH\backslash B=\{\cH \HleftB  b: b\in B\}$ be the quotient space which we equip  with the quotient topology, i.e., the largest topology making $\pi\colon B\to \cH\backslash B$ continuous.
\end{definition}
\begin{remark}\label{rmk:pi open}
    We will frequently assume that an acting groupoid $\cH$ has open source
    %(and range) 
    map, because then \cite[Prop.\ 2.12]{Wil2019} implies that the quotient map $\pi\colon B\to \cH\backslash B$ is open. 
\end{remark}

    When $\cH$ has a \ssla\ on a Fell bundle $\cB=(q_{\cB}\colon B\to \cX)$
(Definition~\ref{df.ss.FellBundle.action.left}), then $B$ is a left $\cH$-space. In this case, since
$h\HleftB\mvisiblespace$ maps $\cB_{x}$ to $\cB_{h\HleftX x}$ by \ref{item:leftFell.1},
the map 
\begin{equation}\label{eq:def p tilde cB}
q_{\wtcB}\colon \cH\backslash B\to \cH\backslash\cX\quad\text{ given by }\quad \cH \HleftB  b\mapsto \cH\HleftX q_{\cB}(b)
\end{equation}
is well defined, 
and we let $\wtcB \coloneqq (q_{\wtcB}\colon \cH\backslash B \to \cH\backslash\cX)$.
The fiber over $\xi\in \cH\backslash \cX$ of the bundle  is therefore given by 
\[(\wtcB)_{\xi} = \{\cH \HleftB  b: b\in \cB \text{ such that } q_{\cB}(b)\in \xi\}
.\]

\begin{lemma}\label{lem:fibers of quotient bundle}
    Suppose the 
    \ssla\ $\HleftX$ of $\cH$ on the groupoid $\cX$
is free. Let $\xi\in \cH\backslash\cX$. For $\Xi,\Theta$ in the fiber $(\wtcB)_{\xi}$ and for $z\in \mathbb{C}$, we may let
    \begin{gather*}
    \|\Xi\|\coloneqq\|b\| \quad\text{ and }\quad
        z\,\Xi = \cH \HleftB  (zb)
        \quad\text{ where } b\in\Xi, \text{ and}
        \\
        \Xi+\Theta 
        =
        \cH \HleftB  \bigl([h \HleftB b]+c\bigr)
        \quad
        \text{ where } b\in\Xi,c\in\Theta, h\in\cH
        \text{ such that }
        q_{\cB} (c) = q_{\cB} (h \HleftB b).
    \end{gather*}
    With this structure, $(\wtcB)_{\xi}$ is  a complex Banach space. 
\end{lemma}

\begin{proof}
    First note that $\norm{\cdot}$ is well defined: Since $h\HleftB\mvisiblespace$ is isometric on each fiber, $\cH\HleftB  a=\cH\HleftB  b$ implies $\|a\|=\|b\|$. Likewise, scalar multiplication is well defined since each $h\HleftB\mvisiblespace$ is $\mathbb{C}$-linear by assumption.
    
    To see that addition is well defined, we first check that $h$ exists. If we pick any $b\in \Xi, c\in \Theta$, then by definition of the fiber $(\wtcB)_{\xi}$, we have $q_{\cB} (b),q_{\cB} (c)\in \xi$. In particular, there exists $h\in\cH$ such that $q_{\cB} (c)=h\HleftX q_{\cB} (b) = q_{\cB} (h \HleftB b)$. This shows that $c$ and $h \HleftB b$ are in the same fiber of $\cB$,  so that $[h \HleftB b]+c$ makes sense. It remains to check that $\Xi+\Theta$ does not depend on the choices, so assume that we are given $b',c',h'$ with $q_{\cB} (c') = q_{\cB} (h' \HleftB b')$. As $b,b'\in\Xi$ and $c,c'\in \Theta$, there exist $k,l\in\cH$ such that $b'=k\HleftB b$ and $c'=l\HleftB c$. In particular,
    \begin{align*}
        h'\HleftX q_{\cB} (b')
        &=q_{\cB} (c')= q_{\cB} (l\HleftB c)
        =
        l\HleftX q_{\cB}(c)
        =
        l\HleftX \left[h\HleftX q_{\cB} (b)\right]
        \\
        &=
        (lh)\HleftX q_{\cB} (k\inv\HleftB b')
        =
        (lhk\inv)\HleftX q_{\cB} (b').
    \end{align*}
    \assumption{Since the $\cH$-action on~$\cX$ is free}, we conclude that $h'=lhk\inv$, and thus
    \begin{align*}
        [h'\HleftB b']+c'
        &=
        (lhk\inv) \HleftB (k\HleftB b)+l\HleftB c
        =
        (lh)\HleftB b + l\HleftB c
        =
        l\HleftB \bigl([h \HleftB b]+c\bigr),
    \end{align*}
    which shows that $[h'\HleftB b']+c'$ and $[h \HleftB b]+c$ represent the same class in $(\wtcB)_{\xi}$.
    
    It is now easy to check that we have a normed vector space. To see that $(\wtcB)_{\xi}$ is complete, let $(\Xi_n)_n$ be a Cauchy sequence. If we pick arbitrary $b_n\in\Xi_n$ for each $n$, then we can find $h_n\in\cH$ such that $c_{n}\coloneqq h_n\HleftB b_n$ is in the same fiber as the representative $b_{1}$ of $\xi_{1}$; say, in $\cB_{x}$. We now have a sequence $(c_n)_n$ in $ \cB_{x}$. Note that, by definition of the linear structure on $(\wtcB)_{\xi}$, we have $\Xi_n - \Xi_m=\cH \HleftB  (c_{n}-c_{m})$, so that
    \[
        \norm{\Xi_n - \Xi_m} = \norm{c_{n}-c_{m}}_{\cB_{x}}.
    \]
    Thus, $(c_{n})_{n}$ is a Cauchy sequence in the Banach space $\cB_{x}$ and hence converges to some element $c$. As 
    \[
        \norm{\Xi_n - \cH \HleftB c} = \norm{c_{n}-c}_{\cB_{x}},
    \]
    we conclude that $\Xi_n \to \cH \HleftB c$ in norm in $(\wtcB)_{\xi}$.
\end{proof}

\begin{corollary}\label{cor:widetildecB is USCBb}
    Suppose the 
        \ssla\ $\HleftX$ of $\cH$ on the groupoid $\cX$
is free and $\cH$ as open source 
%(and range) 
map.  Then $\wtcB = (q_{\wtcB}\colon \cH\backslash B \to \cH\backslash\cX)$ is a \USCBb.
\end{corollary}

\begin{proof}
    We will check that we can apply
    \cite[\CitationFromZStwo]{DL:ZS2}
    to the commutative diagram
    \[
    \begin{tikzcd}[ampersand replacement=\&]%% not sure why I need that here
        B
        \ar[r, "\pi"]\ar[d, "q_{\cB}"']
        \&
        \cH\backslash B
        \ar[d, "q_{\wtcB}"]
        \\
        \cX
        \ar[r,]
        \&
        \cH\backslash\cX
    \end{tikzcd}
    \]
    We have already noted in Lemma~\ref{lem:fibers of quotient bundle} that the fibers of $\wtcB$ are complex Banach spaces.     By definition of the topologies of the spaces on the right-hand side, the vertical maps are quotient maps. Moreover, $\pi$ is open by Remark~\ref{rmk:pi open} and $\cX\to\cH\backslash\cX$ is open by \cite[Proposition 2.12]{Wil2019} 
    \assumption{since $s_{\cH}$ is open}. Therefore, 
    Assumption~(i) of \cite[\CitationFromZStwo]{DL:ZS2} holds.
       By definition of the Banach space
    structure on the fibers of $\wtcB$ (see Lemma~\ref{lem:fibers of quotient bundle}), 
    Assumption~(ii) holds.

    Lastly,     let $\Xi\in \cH\backslash B$ and $x \in
    q_{\wtcB}
    (\Xi)$ be given, and take any $b\in  \Xi\subseteq B$. Since $q_{\cB}(b)\in q_{\wtcB}(\Xi)$, there exists $h\in \cH$ such that $x= h\HleftX q_{\cB}(b)=q_{\cB} (h \HleftB b)$. This means that $h \HleftB b \in \cB_{x}$ satisfies $\pi(h \HleftB b)=\Xi$, since $\pi\circ (h\HleftB\mvisiblespace)=\pi$ where both are defined. This proves
    the final
    Assumption~(iii) of \cite[\CitationFromZStwo]{DL:ZS2}.
\end{proof}

As before, we will write $s_{\wtcB}\coloneqq s_{\cH\backslash\cX}\circ q_{\wtcB}$ and $r_{\wtcB}\coloneqq r_{\cH\backslash\cX}\circ q_{\wtcB}$. 

\begin{proposition}\label{prop.Fell bdl structure on quotient bundle}
Suppose the
    \ssla\ $\HleftX$ of $\cH$ on the groupoid $\cX$
is free
    and proper
and $\cH$ as open source  
map. For two elements $\Xi,\Theta$ of $\wtcB%=\cH\backslash \cB
$ with $s_{\wtcB}(\Xi)=r_{\wtcB}(\Theta)$,
    define 
    \[
    \Xi\Theta
    =
    \cH \HleftB  (bc)
    \quad
    \text{where }b\in \Xi, c\in \Theta \text{ are such that }(b,c)\in\cB^{(2)}.
    \]
    Further, let $(\cH \HleftB  b)^*=\cH \HleftB  b^*$. With this structure, $\wtcB$ is a Fell bundle, which we call the {\em left quotient bundle of~$\cB$ by $\cH$}.
\end{proposition}

\begin{proof}
Since $\cH$ acts freely and properly on $\cX$, the quotient $\cH\backslash\cX$ is a groupoid by Proposition~\ref{prop.gpd structure on orbit space}.
We first verify that
$b\in\Xi$ and $c\in\Theta$ exist, so start with two {\em arbitrary} elements $b\in \Xi$ and $c'\in \Theta$. By construction of $s_{\wtcB}$ and $r_{\wtcB}$, we have $s_{\cB}(b)\in s_{\wtcB}(\Xi)$ and $r_{\cB}(c')\in r_{\wtcB}(\Theta)$. By assumption, the equivalence classes are the same element in $\cH\backslash \cX$, so there exists $h\in \cH$ such that $s_{\cB}(b)=h\HleftX r_{\cB}(c')$. We have% by \eqref{item:L10}
\[  
    h\HleftX r_{\cB}(c')
    =
    h\HleftX r_{\cX}(q_{\cB}(c'))
    \overset{\eqref{item:L10}}{=}
    r_{\cX}(h\HleftX q_{\cB}(c'))
    \overset{\ref{item:leftFell.1}}{=}
    r_{\cX}(q_{\cB}(h\HleftB c')).
\]
Thus, for the element $c\coloneqq h\HleftB c'$ of $\Theta$, we have shown that $(b,c)\in \cB^{(2)}$. Next, we must show that  the multiplication  does not depend on the choice of $(b,c)\in \cB^{(2)}$, so assume that $(b_{1},c_{1})$ is another composable pair of~$\cB$ for which $b_{1}\in \Xi$ and $c_{1}\in \Theta$. Then there exist $k,l\in\cH$ such that $b_{1}=k\HleftB b$ and $c_{1}=l\HleftB c$. A computation similar to that in the proof of Lemma~\ref{lm.orbit multiplication} shows that
\[
    [k\HrightB b ]\HleftX s_{\cB} (b) 
    =
    s_{\cB}(b_{1})
    =
    r_{\cB} (c_{1})
    =
    l\HleftX r_{\cB}(c)
    =
    l\HleftX s_{\cB}(b).
\]
\assumption{Since the $\cH$-action on~$\cX$ is free},  we conclude $l=k\HrightB b $, so that \ref{item:leftFell.4} implies
\[
    b_{1}c_{1}=[k\HleftB b][l\HleftB c]
    =
    [k\HleftB b] \bigl[(k\HrightB b )\HleftB c\bigr]
    =
 k\HleftB bc.
\]
In other words, $b_{1}c_{1}\in \cH \HleftB  bc$, as claimed.

\smallskip

Now, if $b_{1} = k\HleftB b$,  then $b_{1}^* = (k\HleftB b)^*= [k\HrightB b ]\HleftB b^*$, which shows that $\cH  \HleftB  b_{1}^*=\cH  \HleftB  b^*$, i.e., involution is well defined on $\wtcB$.

\smallskip

As noted in Corollary~\ref{cor:widetildecB is USCBb}, $\wtcB$ is a \USCBb. Moreover, the algebraic and norm-related properties for Fell bundles (that is, (F1)--(F10) as listed in Subsection~\ref{ssec:Fell bdl properties}) are all swiftly verified and follow from the respective properties of $\cB$. For example, to show (F10), take an arbitrary $\Xi\in \wtcB$ and any $b\in \Xi$; let $u\coloneqq s_{\cB}(b)$. Since~$\cB$ is a Fell bundle, we have  $b^*\cdot b=c^*c$  for some $c\in \cB_{u}$. The definition of the multiplication and involution on $\wtcB$ thus implies that 
\[\Xi^*\Xi = \cH \HleftB  (b^*\cdot b)=\cH \HleftB  (c^*c)=(\cH \HleftB  c)^* (\cH \HleftB  c).\]
Since 
$\cH \HleftB  c\in (\wtcB)_{\cH \HleftX  u}$
and $\cH \HleftX  u =\cH \HleftX  s_{\cB}(b)= s_{\wtcB}(\Xi)$,
this proves that $\Xi^*\Xi$ is a positive element of the C*-algebra~$(\wtcB)_{\cH \HleftX  u}$, as needed for (F10).
\end{proof}

\begin{remark}\label{rm.saturated is inherited by quotient}
    If $\cB$ is saturated, then so is $\cH\backslash\cB$. Indeed, take $(\xi_{1},\xi_{2}) \in (\cH\backslash\cX)^{(2)}$ and let $\Theta\in (\cH\backslash\cB)_{\xi_{1}\xi_{2}}$ be arbitrary. By definition of the fiber, there exists $b\in\cB$ with $q_{\cB}(b)\in \xi_{1}\xi_{2}$ and $\cH\HleftB b=\Theta$. Since $\xi_{1}$ and $\xi_{2}$ are composable, we can find  $x_{i}\in\xi_{i}$ such that $(x_{1},x_{2})\in\cX^{(2)}$. Thus, there exists $h\in\cH$ such that \eqref{item:L4} implies $q_{\cB}(b)=h\HleftX (x_{1}x_{2})= y_{1}y_{2}$, where $y_{1}\coloneqq h\HleftX x_{1}$ and $y_{2}= (h\HrightX x_{1})\HleftX x_{2}$. Since $\cB$ is saturated, we can approximate $b$ by linear combinations of products of elements in $\cB_{y_{1}}$ and in $\cB_{y_{2}}$. Since $y_{i}\in\cH\HleftX x_{i}=\xi_{i}$, the image of these elements under $\pi$ are in $(\cH\backslash\cB)_{\xi_{i}}$ and, by definition of the linear and topological structure on $\cH\backslash\cB$, they approximate $\cH\HleftB b=\Theta$, as claimed.
\end{remark}

\begin{example}\label{ex:CP with quotient}
    \repeatable{assumptions of ex:CP with quotient}{%
        Suppose that the
    \ssla\ $\HleftX$ of $\cH$ on the groupoid $\cX$ is free and proper, that $\cH$ has open source map, and that $(\cA, \cH\backslash\cX, \alpha)$ is a groupoid dynamical system.
    }%
    If we define $\tilde{\alpha}_{x}\coloneqq \alpha_{\cH\HleftX x}$ for $x\in \cX$, then $(\cA, \cX, \tilde{\alpha})$ is a groupoid dynamical system. Moreover, $\cH$ has a \ssla\ on 
    $\cB(\cA,\cX,\tilde{\alpha})$ given by $h\HleftB (a,x)\coloneqq (a,h\HleftX x)$ and the quotient bundle $\cH\backslash \cB(\cA,\cX,\tilde{\alpha})$ is exactly $ \cB(\cA,\cH\backslash\cX,\alpha )$.
\end{example}

\begin{remark}
    Analogously to Proposition~\ref{prop.Fell bdl structure on quotient bundle}, we can define the {\em right} quotient bundle $%\overline{\cB}=
    \cB/\cG$ from the right \ss\ action $\BrightG$ of~$\cG$ on $\cB$. 
    We denote an element of $\cB/\cG$ by $b \BrightG  \cG$ and let $q_{\cB/\cG}(b \BrightG  \cG)=q_{\cB}(b) \XrightG \cG$. 
\end{remark}

We next require
a Fell bundle analogue of
\intune\ actions.

\begin{assumptionthm}\label{all assumptions}
    We assume that
    \begin{enumerate}
        \item $\cG$ and $\cH$ are \LCH\ groupoids;
        \item $\cX$ is a $(\cH,\cG)$-\sspa\ with \ss\ actions $\HleftX$ of $\cH$ resp.\ $\XrightG$ of $\cG$
        (Definition~\ref{df.something});
        in particular, the actions are \intune, free, and proper, and the source maps of all three groupoids are open;
        \item $\cB=(q_{\cB}\colon B\to \cX)$ is a saturated Fell bundle,
        \item $\cH$ and~$\cG$ act on the left resp.\ right of~$\cB$ by \ss\ actions $\HleftB$ resp.\ $\BrightG$; and
        \item for any $h\in\cH$, $b\in \cB_x$, $t\in\cG$ for which  $(h\HleftX x)\XrightG t$ is well defined, we have:
        \begin{equation}\tag{BC1}\label{cond.compatible.Fell.BC1}
    (h\HleftB b) \BrightG t = h\HleftB (b \BrightG t).\end{equation}
    \end{enumerate}
    Note that, with the notation introduced after Definition~\ref{df.ss.FellBundle.action.left}, we automatically also have
    \begin{align}
        \tag{BC2}\label{cond.compatible.Fell.BC2}
        (h\HleftB b) \BleftG t &= b \BleftG t
        \\
        \tag{BC3}\label{cond.compatible.Fell.BC3}
        h\HrightB (b \BrightG t) &= h \HrightB b
    \end{align}
    as a consequence of Condition~\ref{item:leftFell.1} combined with Condition~\ref{cond.compatible.c2} resp.\ \ref{cond.compatible.c3}.
\end{assumptionthm}

We first show that the actions   $\HleftB$ and $\BrightG$ on~$\cB$ pass to the quotients.
    We remind the reader of some definitions we made earlier:
    \begin{align*}
    q_{\wtcB}&\colon \wtcB \to \cH\backslash\cX
    &&\quad\text{ is defined by }\quad&
    q_{\wtcB}(\cH \HleftB  b)&=
     \cH \HleftX q_{\cB}(b),
     \text{ and}
     \\ 
    q_{\cB/\cG}&\colon \cB/\cG \to \cX/\cG
    &&\quad\text{ is defined by }\quad&
q_{\cB/\cG} ( b\BrightG  \cG)&=q_{\cB}(b) \XrightG \cG.
     \end{align*}
     Moreover, in Proposition~\ref    {prop.ssp actions on quotients}, we defined a left \ss\ $\cH$-action on $\cX/\cG$ with momentum map  $\tilde\rho(x\XrightG G)=\rho_{\cX}(x)$, and a right \ss\ 
$\cG$-action on $\cH\backslash \cX$ with momentum map 
$\tilde\sigma(\cH\HleftX x)=\sigma_{\cX}(x)$.
\begin{proposition}\label{prop.right.action.quotient}
    We assume all conditions in Assumption~\ref{all assumptions}.
    With 
    \[\mvisiblespace\qBrightG\mvisiblespace \colon \wtcB \bfp{\tilde{\sigma}
    \circ q%_{\wtcB}
    }{r%_{\cG}
    } \cG \to \wtcB,
    \qquad
    \Xi \qBrightG s
    = \cH \HleftB  [b \BrightG s]\quad \text{ where } b\in \Xi,\]
    $\qBrightG $ is a right \ss\ $\cG$-action on $\wtcB$, and with 
\[\mvisiblespace \qHleftB\mvisiblespace \colon \cH \bfp{s_{\cH}}{\tilde{\rho}
    \circ q} \cB/\cG \to \cB/\cG,
    \qquad
    h \qHleftB \Xi \coloneqq [h\HleftB b] \BrightG  \cG\quad \text{ where } b\in \Xi,\]
 $\qHleftB $ is a left \ss\ $\cH$-action on $\cB/\cG$.
\end{proposition}

\begin{proof}  As always, we will focus only on one of the two statements, namely $\qBrightG$.
    
    To see that $\qBrightG$ is well defined, assume  $c\in \Xi$, so there exists $h\in \cH$ such that $c=h \HleftB b$. Therefore, by 
   Equation~\eqref{cond.compatible.Fell.BC1}
    and the definition of $\HleftB $,
\[
    \cH \HleftB  [c \BrightG s]
    =
    \cH \HleftB  \bigl[(h \HleftB b) \BrightG s \bigr]
    =
    \cH \HleftB  \bigl[ h\HleftB (b \BrightG s)\bigr]
    =
    \cH \HleftB  [b \BrightG s]
.\]
It remains to show that $\qBrightG$ satisfies all the conditions listed in Definition~\ref{df.ss.FellBundle.action.left}.
    We start with the algebraic properties.
    For~\ref{item:leftFell.1},
    take an arbitrary element $\xi=\cH\HleftX x\in \cH\backslash\cX$ and $\Xi \in (\wtcB)_{\xi}$. If $b\in \Xi\cap \cB_x$, then
    \[
        \Xi \qBrightG s
        =
        \cH  \HleftB  [b \BrightG s]
        \quad\in\quad
        (\wtcB)_{\cH \HleftX (x \XrightG s)}
        =
        (\wtcB)_{\xi \qXrightG s}.
    \]
    This is linear as a map
    $(\wtcB)_{\xi}\to (\wtcB)_{\xi \qXrightG s}$
    because $\mvisiblespace \BrightG s$ is linear as a map $\cB_x\to \cB_{x\XrightG s}$ 
    and because of how we defined the linear structure on the fibers of the quotient (see Lemma~\ref{lem:fibers of quotient bundle}).
    
    Both~\ref{item:leftFell.2} and~\ref{item:leftFell.3} are trivial.
    For~\ref{item:leftFell.4}, let $(\Xi,\Theta)\in (\wtcB)^{(2)}$. If $b\in \Xi$ and $c\in \Theta$ with $(b,c)\in\cB^{(2)}$, then
    $
        \Xi\Theta=
    \cH \HleftB  (bc)$ by our definition in Proposition~\ref{prop.Fell bdl structure on quotient bundle}. If $x=q_{\cB}(b)$, then
    \begin{align*}
        [\cH \HleftB  bc]\qBrightG s
        &=
        \cH \HleftB 
        (
            bc\BrightG s
        )
        &&\text{(def'n of $\qBrightG$)}
        \\
        &=
        \cH \HleftB 
        \bigl(
            [b\BrightG (x\XleftG s)][c \BrightG s]
        \bigr)
        &&\text{(Property~\ref{item:leftFell.4} for $\BrightG$)}
        \\
        &=
        \bigl[\cH \HleftB 
        \bigl(
            b\BrightG (x\XleftG s)
        \bigr)
        \bigr]
        \,
        [\cH \HleftB 
        (
            c \BrightG s
        )]
        &&\text{(def'n of $\wtcB$; see Prop.\ \ref{prop.Fell bdl structure on quotient bundle})}
        \\
        &=
        \bigl[(\cH \HleftB  b)\qBrightG (x\XleftG s)\bigr]
        \,
        [(\cH \HleftB  c)\qBrightG s]
        &&\text{(def' of $\qBrightG$).}
    \end{align*}
    Since
    $x\XleftG s=(\cH\HleftX x)\qXleftG s=q_{\wtcB}(\Xi)\qXleftG s$ (see the definition of $\qXleftG$ in Proposition~\ref{prop.ssp actions on quotients} and that of $q_{\wtcB}$ in Equation~\eqref{eq:def p tilde cB}) and since $\Xi=\cH\HleftB b$ and $\Theta=\cH\HleftB c$, this proves that
    \[
            \Xi
            \Theta
        \qBrightG s
        =
        \bigl(\Xi\qBrightG [q_{\wtcB}(\Xi)\qXleftG s]\bigr)
        \, 
        (\Theta\qBrightG s),
    \]
    as required.
    
    For~\ref{item:leftFell.5}, we compute
    \begin{align*}
        \bigl(\Xi\qBrightG s\bigr)^*
        &=
        \bigl(
            \cH  \HleftB  [b \BrightG s]
        \bigr)^*
        &&\text{for } b\in\Xi
        \\
        &=
            \cH  \HleftB  [b \BrightG s]^*
        && \text{(involution on $\wtcB$; see Prop.~\ref{prop.Fell bdl structure on quotient bundle})}
        \\
        &=
            \cH  \HleftB  \bigl[ b^*\BrightG(b\BleftG s)\bigr]
        && \text{(Property~\ref{item:leftFell.5} for $\BrightG$)}
        \\
        &=
            \bigl(\cH  \HleftB  b^*\bigr)
            \qBrightG \bigl(b\BleftG s\bigr)
        && \text{(def'n of $\qBrightG $)}
        \\
        &=
          \bigl(\cH  \HleftB  b^*\bigr)\qBrightG \bigl(q_{\wtcB}(\Xi)\qXleftG s\bigr)
        && \text{(def'n of $\qXleftG$; see Prop.~\ref{prop.ssp actions on quotients})}
        \\
        &=
          \Xi^*\qBrightG \bigl(q_{\wtcB}(\Xi)\qXleftG s\bigr)
        && \text{(involution on $\wtcB$; see Prop.~\ref{prop.Fell bdl structure on quotient bundle})},
    \end{align*}
    as required.

    Lastly, we have to check that $\qBrightG$ is continuous, so let $\{(\Xi_{i},s_{i})\}_{i\in I}$ be a net in $\wtcB \bfp{\tilde{\sigma}
    \circ q}{r_{\cG}} \cG$ that converges to $(\Xi,s)$. Since the quotient map $ \cB\to \wtcB$ is open by Remark~\ref{rmk:pi open}, there exists a subnet $\{\Xi_{ j }\}_{j\in J}$ of $\{\Xi_{i}\}_{i\in I}$ and lifts $b_\mu\in \Xi_{ j }$, $b\in \Xi$ such that $b_j\to b$ in $\cB$. Since $\BrightG$ is continuous, it follows that $b_j \BrightG s_{ j }\to b\BrightG s$, so that 
    \[  
        \Xi_{ j } \qBrightG s_{ j }
        =
        \cH\HleftB  [b_\mu \BrightG s_{ j }]
        \to 
        \cH\HleftB  [b\BrightG s]
        =
        \Xi \qBrightG s
        .
    \]
    By Lemma~\ref{lem:exercise in top:cts maps}, this suffices to conclude that $\qBrightG $ is continuous.
\end{proof}

 \section{The symmetric im\-prim\-i\-tiv\-ity theorem for \ss\ actions}\label{sec.main}

Let us rehash what the conditions in Assumption~\ref{all assumptions} imply.
By Proposition~\ref{prop.Fell bdl structure on quotient bundle}, we get two quotient Fell bundles: the right quotient~$\cB/\cG$ over~$\cX/\cG$ and the left quotient~$\cH\backslash\cB$ over the groupoid $\cH\backslash\cX$. These are saturated by Remark~\ref{rm.saturated is inherited by quotient}, since $\cB$ is assumed to be saturated.
We have seen in Proposition~\ref{prop.right.action.quotient} that $\cH\backslash\cB$ carries a right \ss\ $\cG$-action $\qBrightG$, and likewise, $\cB/\cG$ carries a left \ss\ $\cH$-action $\qHleftB$. We can therefore take two \ss\ products, as explained in Definition~\ref{df.ss.product.bundle.left} resp.\ Remark~\ref{rm.df.ss.product.bundle.right}:
\begin{itemize}
    \item the product $(\cB/\cG)\BbowtieH\cH$ of $\cB/\cG$ with~$\cH$ is a bundle over $(\cX/\cG)\bowtie\cH$ and will be denoted $q_{\cA}\colon \cA\to (\cX/\cG)\bowtie\cH$, while
    \item the product $\cG\GbowtieB(\cH\backslash\cB)$ of $\cH\backslash\cB$ with~$\cG$ is a bundle over $\cG\bowtie(\cH\backslash \cX)$ and will be denoted $q_{\cC}\colon \cC\to  \cG\bowtie(\cH\backslash \cX)$.
\end{itemize}
These \ssp\ Fell bundles are saturated by Remark~\ref{rm.saturated is inherited by bowtie}.

We now prove that $\cA$ and $\cC$ are equivalent via the bundle $\cB$ in the sense of {\cite[Definition 6.1]{MW2008}}. Recall that $\cX$ is a groupoid equivalence between $(\cX/\cG) \bowtie \cH$ and $\cG\bowtie (\cH\backslash \cX)$ by Theorem~\ref{thm.groupoid.eq} when equipped with the structure defined in Proposition~\ref{prop.actions.X}. We remind the reader that $\mathfrak{r}\colon \cX\to [(\cX/\cG)\bowtie \cH]\z$ denotes the momentum map of $\cX$ for that left action and $\mathfrak{s}\colon \cX\to [\cG\bowtie(\cH\backslash\cX)]\z$ the momentum map for that right-action. Consequently, we will write $\mathfrak{r}_{\cB}\coloneqq \mathfrak{r}\circ q_{\cB}$, not to be confused with $r_{\cB}=r_{\cX}\circ q_{\cB}\colon \cB\to \cX\z$; we likewise let $\mathfrak{s}_{\cB}\coloneqq \mathfrak{s}\circ q_{\cB}$.

\begin{theorem}[cf.\ {\cite[Theorem 3.1]{KMQW2013}}]\label{thm.imprimitivity.main}
We assume all conditions in Assumption~\ref{all assumptions}.
Then~$\cB$ is a Fell bundle equivalence between $\cA=(\cB/\cG) \BbowtieH \cH$ and $\cC=\cG\GbowtieB (\cH\backslash \cB)$ in the following way:
\begin{enumerate}[label=\textup{(\arabic*)}]
    \item\label{thm:it:left action} $\cA$
    acts on the left of~$\cB$: 
    whenever $(\Theta,h)\in \cA$ and $b\in\cB$ are such that $s_{\cA}(\Theta,h) = \mathfrak{r}_{\cB}(b)$ in $[(\cX/\cG)\bowtie \cH]\z$, we let
    \[(\Theta,h) \cdot b = a [h\HleftB b], \text{ where } a\in\Theta \text{ is such that }s_{\cB}(a)=r_{\cB}(h\HleftB b). \]  
    \item\label{thm:it:right action} $\cC$ 
    acts on the right of~$\cB$: 
    whenever $b\in\cB$ and $(t,\Xi)\in \cC$  are such that $\mathfrak{s}_{\cB}(b)=r_{\cC}(t,\Xi) $ in $[\cG\bowtie(\cH\backslash\cX)]\z$, we let
    \[b\cdot (t, \Xi) = [b\BrightG t] c, \text{ where } c\in\Xi \text{ is such that } s_{\cB}(b\BrightG t)=r_{\cB}(c). \]    
    \item\label{thm:it:left ip} The left $\cA$-valued inner product 
    defined on $\cB\bfp{\mathfrak{s}}{\mathfrak{s}}\cB$
    is given by
    \[\linner{\cA}{a}{b}=
    \bigl(
        [a(h\HleftB b^*)]  \BrightG  \cG, h \HrightB b^*
    \bigr),\]
    where $h$ is the unique element of $\cH$ such that $s_{\cB}(a)=h\HleftX s_{\cB}(b)$.
    \item\label{thm:it:right ip} The right $\cC$-valued inner product 
    defined on $\cB\bfp{\mathfrak{r}}{\mathfrak{r}}\cB$
    is given by
    \[
    \rinner{\cC}{a}{b}=
    \bigl(
        a^*\BleftG
        t , \cH  \HleftB  [(a^*\BrightG t)b] \bigr),
    \]
    where
    $t$ is the unique element of $\cG$ such that $r_{\cB}(a)\XrightG t=r_{\cB}(b)$.
\end{enumerate}
\end{theorem}

    \begin{example}[see also Example~\ref{ex:recovering KMQW2023, Lemma 3.2}]\label{ex:recovering KMQW2023, Thm3.1}
    Theorem~\ref{thm.imprimitivity.main} recovers \cite[Theorem 3.1]{KMQW2013}: 
    suppose $G$ and $H$ are \LCH\ groups with commuting actions on a Fell bundle~$\cB=(q_{\cB}\colon B\to \cX)$ by Fell bundle automorphisms, where $\cX$ is a \LCH\ groupoid.
    The induced actions of $G$ and $H$ on $\cX$ are then by groupoid automorphisms, and so (with
    $\cX$ acting trivially on $H$ and $G$) they are \ss\ actions on $\cX$. If the actions are free and proper, then
    $\cB$ as described in Theorem~\ref{thm.imprimitivity.main} is a Fell bundle equivalence between
    the semi-direct product bundles
    $(\cB/G)\rtimes H$ and $G\ltimes (H\backslash \cB)$ as considered in \cite{KMQW2013}.
    \end{example}

We will do the proof in pieces.

\begin{lemma}\label{lem:thm:actions}
    The formulas in \ref{thm:it:left action} and \ref{thm:it:right action} of Theorem~\ref{thm.imprimitivity.main} define actions on the \USCBb\ $\cB$ in the sense of Definition~\ref{df.USCBb-action}.
\end{lemma}
\begin{proof}
    We will  follow similar ideas as in the proof of Proposition~\ref{prop.actions.X}  and we will only do the proof for the left action; the other one follows {\em mutatis mutandis}. We will denote the source and range map of $(\cX/\cG)\bowtie \cH$ merely by $r$ resp.~$s$. 
    
    First, let us check that the condition  $s_{\cA}(\Theta,h) = \mathfrak{r}_{\cB}(b)$ implies that there indeed exists $a\in \Theta$ with $s_{\cB}(a)=r_{\cB}(h\HleftB b)$, so that $a[h\HleftB b]$ makes sense. If $a_0$ is {\em any} element of $\Theta$, then
    \begin{align*}
         s_{\cA}(\Theta,h)
         &=
         s \bigl(q_{\cA}(\Theta,h)\bigr)
         =
         s \bigl(q_{\cB/\cG}(\Theta),h\bigr)
         &&\text{(def'n of $q_{\cA}$ in Definition~\ref{df.ss.product.bundle.left})}
         \\
         &=
         s \bigl(q_{\cB}(a_0)\XrightG \cG,h\bigr)
         &&\text{(def'n of $q_{\cB/\cG}$; cf.\ \eqref{eq:def p tilde cB} on p.~\pageref{eq:def p tilde cB})}
         \\
         &=
         h\inv \qHleftX s_{\cX/\cG}(q_{\cB}(a_0)\XrightG \cG)
         &&\text{(Rmk~\ref{rm.unit space of ssp} and def'n of $s$; cf.\ Definition~\ref{def.ZSProduct.groupoid.left})}
         \\
         &=
         h\inv \qHleftX [s_{\cB}(a_0)\XrightG \cG]
         &&\text{(def'n of $s_{\cX/\cG}$; cf.\ Lemma~\ref{lem:s and r of H under X})}
         \\
         &=
         [h\inv \HleftX s_{\cB}(a_0)]\XrightG \cG
         &&\text{(def'n of $\qHleftX$; see Proposition~\ref{prop.ssp actions on quotients})}
    \end{align*}
    On the other hand, $\mathfrak{r}_{\cB}(b)=r_{\cB}(b)\XrightG \cG$, and so our assumption $s_{\cA}(\Theta,h) = \mathfrak{r}_{\cB}(b)$ implies that there exists $t\in\cG$ such that 
    \[
        r_{\cB}(b)= [h\inv \HleftX s_{\cB}(a_0)]\XrightG t \overset{\ref{cond.compatible.c1}}{=} h\inv \HleftX[ s_{\cB}(a_0)\XrightG t],
    \]
    i.e., 
        $h \HleftX r_{\cX}(q_{\cB}(b))
        =
        s_{\cX}(q_{\cB}(a_0))\XrightG t$. 
    Since
    \begin{align*}
        h \HleftX r_{\cX}(q_{\cB}(b))
        \overset{\eqref{item:L10}}{=}
        r_{\cX}(h \HleftX  q_{\cB}(b))
        \overset{\ref{item:leftFell.1}}{=}
        r_{\cX}(q_{\cB}(h \HleftB  b))
    \end{align*}
    and likewise, $s_{\cB}(a_0)\XrightG t=     s_{\cB}(a_0\BrightG t)$,
    we may thus let  $a\coloneqq a_0\BrightG t$, which is the required element of $a_0\BrightG \cG =\Theta$.

    Note that this chosen representative $a\in\Theta$ is unique, since \assumption{the $\cG$-action on $\cX$ is free:} 
    if $a\BrightG s$ also satisfies $s_{\cB}(a\BrightG s)=r_{\cB}(h\HleftB b)$, then 
    \[s_{\cB}(a)\XrightG s \overset{\eqref{item:R10}}{=}s_{\cB}(a\BrightG s)=s_{\cB}(a), \text{ so } s\in \cG\z.\]

    \medskip
    
   To see that the left action is continuous, assume that we have a net $\{(\Theta_{ i },h_{ i }, b_{ i })\}_{ i \in I}$ in $\cA\bfp{s}{\mathfrak{r}}\cB$ that converges to $(\Theta,h,b)$. For each $ i $, let $a_{ i }\in\Theta_{ i }$ be the unique element such that $u_ i \coloneqq s_{\cB}(a_{ i })=r_{\cB}(h_{ i }\HleftB b_{ i })$. By Lemma~\ref{lem:exercise in top:cts maps}, it suffices to check that a subnet of $a_{ i } [h_{ i }\HleftB b_{ i }]$ converges to $a [h\HleftB b]$. Since $\HleftB$ is continuous, we already know that $\{h_{ i }\HleftB b_{ i }\}_{i\in I}$ converges to $h\HleftB b$; and so in particular $u_ i \to u\coloneqq s_{\cB}(a)$ in $\cX\z$, and since multiplication on $\cB$ is continuous, it suffices to show that a subnet of $\{a_{ i }\}_{ i }$ converges to $a$. 
   
   Since the quotient map $\cB\to \cB/\cG$ is open (cf.\ Remark~\ref{rmk:pi open})
 and since $\Theta_ i \to\Theta$, 
    Proposition~\ref{Fell's criterion}
 implies that there exists a subnet $\{\Theta_{ j }\}_{j\in J}$ and lifts $c_{ j }\in \Theta_{ j }$ such that $c_{ j }\to a$ in $\cB$. Since $a_{ j }\in\Theta_{ j }$ also, there exist $t_{ j }\in \cG$ such that $a_{ j }\BrightG t_{ j }=c_{ j }$. In particular, by continuity of $s_{\cB}$, we have 
 \[
    u_{ j }\XrightG t_{ j }
    =
    s_{\cB}(a_{ j })\XrightG t_{ j }
    \overset{\eqref{item:R10}}{=}
    s_{\cX}(q_{\cB}(a_{ j })\XrightG t_{ j })
    \overset{\ref{item:leftFell.1}}{=}
    s_{\cX}(q_{\cB}(a_{ j }\BrightG t_{ j }))
    =
    s_{\cB}(c_{ j })\to s_{\cB}(a)=u,
\]
   so that
 \[
    ( u_{ j }\XrightG t_{ j },  u_{ j })
      \to
    ( u,u)
    \quad \text{ in }
    \cX\z\times \cX\z.
 \]
 As the right action of~$\cG$ on~$\cX\z$ is \assumption{proper}, it now follows from \cite[Corollary 2.26]{Wil2019} that $t_{ j }$ converges; since the action is \assumption{free}, \eqref{item:R2} implies that it must converge to $\sigma_{\cX}(u)\in \cG\z$. This, in turn, implies that
 \[
    a_{ j }
    =
    c_{ j } \BrightG t_{ j }\inv
    \to
    a \BrightG \sigma_{\cX}(u)\inv
    \overset{\ref{item:leftFell.3}}{=}
    a,
 \]
 as needed.
    
    \medskip
    
    To see that \ref{item:FA:fiber} holds, we must check that $q_{\cB}((a \BrightG  \cG, h)\cdot b)=q_{\cA}(a \BrightG  \cG, h) \cdot q_{\cB}(b)$, 
    where $\cdot$ is the left-$(\cX/ \cG)\bowtie\cH$ action on $\cX$ as defined in Proposition~\ref{prop.actions.X}.
    Let %$a\in \cB_x$ and $b\in \cB_y$ such
    $q_{\cB}(a)=x$ and $q_{\cB}(b)=y$, so
    that $s_{\cB}(a) = s_{\cX}(x)$ equals $r_{\cB}(h\HleftB b)=r_{\cX}(h\HleftX y)$ and 
    \[q_{\cB}((a \BrightG  \cG, h)\cdot b) = q_{\cB}(a[h\HleftB b]) = x[h\HleftX y].\]
    On the other hand, $q_{\cA}    (a \BrightG  \cG, h)=(x\XrightG \cG, h)$. 
    By Proposition \ref{prop.actions.X}, since  $s_{\cX}(x)=r_{\cX}(h\HleftX y)$, we know that $(x\XrightG \cG, h)$ can act on the left of $y$ and we get
    \[q_{\cA}(a \BrightG  \cG, h) \cdot q_{\cB}(b)=(x\XrightG \cG, h)\cdot y = x[h\HleftX y].\]
    This proves \ref{item:FA:fiber}.
    
    Next, we must show that \ref{item:FA:assoc} holds, i.e., associativity, so for $i=1,2$
    pick $a_{i}\in\cB_{x_{i}},b\in\cB$ and $h_{i}\in\cH$ with appropriate range and sources such that
    \[
        (a_{1} \BrightG  \cG,h_{1})(a_{2} \BrightG  \cG,h_{2})\quad\text{and}\quad (a_{2} \BrightG  \cG,h_{2}) \cdot b
    \]
    make sense; we have to show
    \begin{equation}\label{eq:assoc}
        \bigl[(a_{1} \BrightG  \cG,h_{1})(a_{2} \BrightG  \cG,h_{2})\bigr]\cdot b 
        =
        (a_{1} \BrightG  \cG,h_{1})\cdot\bigl[(a_{2} \BrightG  \cG,h_{2}) \cdot b\bigr].
    \end{equation}
    In $\cA=(\cB/\cG)\BbowtieH\cH$, we have
    %We have
    \begin{align*}
        &(a_{1} \BrightG  \cG,h_{1})(a_{2} \BrightG  \cG,h_{2})
        \\
        &
        =
        \bigl(
            [a_{1} \BrightG  \cG](h_{1}\qHleftB [a_{2} \BrightG  \cG]), [h_{1}\qHrightX q_{\cA}(a_{2} \BrightG  \cG)]h_{2}
        \bigr)
        &&\text{(Definition \ref{df.ss.product.bundle.left})}
        \\
        &=
        \bigl(
            [a_{1} \BrightG  \cG] [(h_{1} \HleftB a_{2}) \BrightG  \cG], [h_{1}\qHrightX ( x_{2}  \XrightG \cG)]h_{2}
        \bigr)
        &&\text{(def'n of $\qHleftB$ and $q_{\cA}$ in Prop.~\ref{prop.right.action.quotient})}
        \\
        &=\bigl(
            [a_{1}  (h_{1} \HleftB a_{2})] \BrightG  \cG,  [h_{1}\HrightX  x_{2}  ]h_{2}
        \bigr)
        &&\text{(Prop.~\ref{prop.Fell bdl structure on quotient bundle} for $\cB/\cG$; $\qHrightX$ in Prop.~\ref{prop.ssp actions on quotients}).}
    \end{align*}
    Therefore, we get
    \begin{align*}
        \bigl[(a_{1} \BrightG  \cG,h_{1})(a_{2} \BrightG  \cG,h_{2})
        \bigr]\cdot b 
        = a_{1}(h_{1}\HleftB a_{2})\bigl( [(h_{1}\HrightX x_{2})h_{2}]\HleftB b\bigr).
    \end{align*}
On the other hand,
    \begin{align*}
        (a_{1} \BrightG  \cG,h_{1}) \cdot \bigl[(a_{2} \BrightG  \cG,h_{2}) \cdot b\bigr]
        &=
        (a_{1} \BrightG  \cG,h_{1}) \cdot (a_{2}[h_{2}\HleftB b])
        \\&=
        a_{1}
        \Bigl[
            h_{1}\HleftB \bigl(a_{2}[h_{2}\HleftB b]\bigr)
        \Bigr]
        \\
        &= a_{1}(h_{1}\HleftB a_{2}) \bigl([(h_{1}\HrightX x_{2})h_{2}]\HleftB b\bigr)
        &&\text{(by \ref{item:leftFell.4} for $\cB$),}
    \end{align*}
    so we have shown Equation~\eqref{eq:assoc}.
    
    For \ref{item:FA:norm}, recall from Lemma~\ref{cor.isometric.action} that $h\HleftB \mvisiblespace$ is isometric and $\|(b \BrightG  \cG,h)\|=\|b\|$. Therefore,
\[\|(a \BrightG  \cG,h) \cdot b\| = \| a (h\HleftB b) \| \leq \|a\| \|b\| = \|(a \BrightG  \cG,h)\| \|b\|,\]
    as needed.
\end{proof}

\begin{lemma}[Regarding \ref{item:FE:actions}]\label{lem:thm:actions commute}
    The left and right actions commute.
\end{lemma}

\begin{proof}
    Let $(\Theta,h)\in \cA$, $b\in\cB$, and $(t,\Xi)\in \cC$ be such that $s_{\cA}(\Theta,h) = \mathfrak{r}_{\cB}(b)$ %in $[(\cX/\cG)\bowtie \cH]\z$ 
    and $\mathfrak{s}_{\cB}(b)=r_{\cC}(t,\Xi) $; % in $[\cG\bowtie(\cH\backslash\cX)]\z$;
    we have to confirm that $[(\Theta,h) \cdot b]\cdot (t, \Xi) =(\Theta,h) \cdot [b\cdot (t, \Xi) ]$.  
    For the left-hand side, we let $a$ be the (unique) element of $\Theta$ with $s_{\cB}(a)=r_{\cB}(h\HleftB b)$, so that $(\Theta,h) \cdot b = a [h\HleftB b]$; then let $c$ be the (unique) element in $\Xi$ with $s_{\cB}((a [h\HleftB b])\BrightG t)=r_{\cB}(c)$, so that 
    \begin{align*}
        [(\Theta,h) \cdot b]\cdot (t, \Xi)
        &=\bigl[(a [h\HleftB b])\BrightG t\bigr]c
        \\
        &=
        \left[a \BrightG ([h\HleftB b]\BleftG t)\right] 
        \, ([h\HleftB b]\BrightG t)
        \, c
        &&\text{(by \ref{item:leftFell.4} for $\BrightG$)}\\
        &=
        \left[a \BrightG (b\BleftG t)\right] 
        \, ([h\HleftB b]\BrightG t)
        \, c
        &&\text{(by \eqref{cond.compatible.Fell.BC2}).}
    \end{align*}
    On the other hand, let $c'$ be the (unique) element in $\Xi$ with $s_{\cB}(b\BrightG t)=r_{\cB}(c')$, so that $b\cdot (t, \Xi) = [b\BrightG t] c'$; then let $a'$ be the (unique) element in $\Theta$ with $s_{\cB}(a')=r_{\cB}(h\HleftB ([b\BrightG t] c'))$, so that
    \begin{align*}
        (\Theta,h) \cdot [b\cdot (t, \Xi)]
        &=
        a' \bigl[h\HleftB ([b\BrightG t] c')\bigr]
        \\
        &=
        a'
        \,
        (h\HleftB [b\BrightG t])
        \,
        \left[ (h\HrightB [b\BrightG t] )\HleftB c'\right]
        &&\text{(by \ref{item:leftFell.4})}
        \\
        &=
        a'
        \,
        (h\HleftB [b\BrightG t])
        \,
        \left[ (h\HrightB b )\HleftB c'\right]
        &&\text{(by \eqref{cond.compatible.Fell.BC3})}.
    \end{align*}
    Since $[h\HleftB b]\BrightG t= h\HleftB [b\BrightG t]$ by \eqref{cond.compatible.Fell.BC1}, we see that it suffices to check that
    \[
        a'=a \BrightG (b\BleftG t)
        \quad\text{and}\quad
        (h\HrightB b )\HleftB c' = c.
    \]
    Note that the second equation is the $\HleftB$-version of the first equation, so by symmetry, it suffices to check the first equation. We have $a \BrightG (b\BleftG t)\in a \BrightG \cG = \Theta$, so by uniqueness of $a'$, we only need to check that $s_{\cB}(a \BrightG (b\BleftG t))=r_{\cB}(h\HleftB ([b\BrightG t] c')).$

    Since 
    \[
        q_{\cB}(a \BrightG (b\BleftG t))
        =
        q_{\cB}(a) \XrightG [b\BleftG t],
    \]
    we have
    \begin{align*}
        s_{\cB}\bigl(a \BrightG (b\BleftG t)\bigr)
        &=
        s_{\cX}\bigl(q_{\cB}(a) \XrightG [b\BleftG t]\bigr)
        \\
        &=
        s_{\cB}(a) \XrightG [b\BleftG t]
        &&
        \text{(by \eqref{item:R10})}
        \\
        &=
        r_{\cB}(h\HleftB b)\XrightG [b\BleftG t]
        &&\text{(by choice of $a$)}
        \\
        &=
        \bigl(h\HleftX r_{\cB}(b)\bigr)\XrightG [b\BleftG t]
        &&\text{(by \eqref{item:L10})}
        \\
        &=
        h\HleftX \bigl( r_{\cB}(b)\XrightG [b\BleftG t]\bigr)
        &&\text{(by \ref{cond.compatible.c1})}
        \\
        &=
        h\HleftX  r_{\cX}\bigl(q_{\cB}(b)\XrightG t\bigr)
        &&\text{(by \eqref{item:R10}).}
    \end{align*}
    On the other hand,
    \begin{align*}
        q_{\cB}(h\HleftB ([b\BrightG t] c'))
        &\overset{\ref{item:leftFell.1}}{=}
        h\HleftX  q_{\cB}([b\BrightG t] c')
        \overset{\ref{cond.F1}}{=}
        h\HleftX  [q_{\cB}(b\BrightG t)q_{\cB}( c')]
        \\
        &\overset{\eqref{item:L4}}{=}
        \bigl[h\HleftX q_{\cB}(b\BrightG t)\bigr]
        \,
        \Bigl(\bigl[h\HrightX q_{\cB}(b\BrightG t)\bigr] \HleftX q_{\cB}( c')\Bigr),
    \end{align*}
    so that it follows from \eqref{item:L10} and \ref{item:leftFell.1} for $\BrightG$ that
    \[
        r_{\cB}(h\HleftB ([b\BrightG t] c'))
        =
        r_{\cX}\bigl(h\HleftX q_{\cB}(b\BrightG t)\bigr)
        =
        h\HleftX  r_{\cX}\bigl(q_{\cB}(b)\XrightG t\bigr).
    \]
    Our earlier computation therefore shows that $s_{\cB}\bigl(a \BrightG (b\BleftG t)\bigr)=r_{\cB}(h\HleftB ([b\BrightG t] c'))$, as needed.
This shows that the left and right actions commute.
\end{proof}

\begin{lemma}[Regarding \ref{item:FE:ip}]\label{lem:thm:ips}
    The formulas in \ref{thm:it:left ip} and \ref{thm:it:right ip} pf Theorem~\ref{thm.imprimitivity.main} define inner products on the \USCBb\ $\cB$ in the sense of Definition~\ref{df.FBequivalence}, \ref{item:FE:ip:fiber}--\ref{item:FE:ip:C*linear}.
\end{lemma}

\begin{proof}

    We will do the proof for the left inner product; the other one follows {\em mutatis mutandis}.
    
    First, we verify that the inner product is well defined. As $\mathfrak{s}_{\cB}(a)=\mathfrak{s}_{\cB}(b)$, the definition of $\mathfrak{s}_{\cB}=\mathfrak{s}\circ q_{\cB}$ implies the existence of $h$ satisfying $h\HleftX s_{\cB}(b)=s_{\cB}(a)$. As this implies $s_{\cH}(h)=\rho_{\cX}(s_{\cB}(b))$, we therefore have
    \[\rho_{\cB}(b^*) = \rho_{\cX}(r_{\cB}(b^*))=\rho_{\cX}(s_{\cB}(b))=s_{\cH}(h),\]
    and so $h\HleftB b^*$ and $h\HrightB b^*=h\HrightX q_{\cB}(b^*)$ make sense.
    Now $q_{\cB}(h\HleftB b^*)=h\HleftX q_{\cB}(b^*)$, and so by \eqref{item:L10}, we thus have
    \[
        r_{\cB}(h\HleftB b^*)
        =
        h\HleftB r_{\cB}(b^*) = h\HleftX s_{\cB}(b) = s_{\cB}(a),
    \]
 so that $a[h\HleftB b^*]$ makes sense. 
 
 To see that
 \[\linner{\cA}{a}{b}=
    \bigl(
        [a(h\HleftB b^*)]  \BrightG  \cG, h\HrightB b^*
    \bigr)\]
    is an element of $\cA=(\cB/\cG) \BbowtieH \cH$, we have to verify that
    \[
        (\rho_{\cX/\cG}\circ s_{\cX/\cG}\circ q_{\cB/\cG}) \bigl( [a(h\HleftB b^*)]  \BrightG  \cG\bigr)
        =
        r_{\cH}( h \HrightB b^*
        ).
    \]
    Recall from a $\BrightG$-version of Equation~\eqref{eq:def p tilde cB} that
    \begin{align*}
        q_{\cB/\cG} \bigl( [a(h\HleftB b^*)]  \BrightG  \cG\bigr)
        =
        q_{\cB} \bigl( a(h\HleftB b^*)  \bigr)\XrightG  \cG.
    \end{align*}
    Moreover, $s_{\cX/\cG}(x\XrightG \cG)=s_{\cX}(x)\XrightG \cG$ (cf.\ the definition before Lemma~\ref{lem:s and r of H under X}) and $\rho_{\cX/\cG} (x\XrightG \cG)=\rho_{\cX} (x)$ by the definition in Proposition~\ref{prop.ssp actions on quotients}.
    Thus
    \[
        (\rho_{\cX/\cG}\circ s_{\cX/\cG}\circ q_{\cB/\cG}) \bigl( [a(h\HleftB b^*)]  \BrightG  \cG\bigr)
        =
        \rho_{\cX}
        \bigl( s_{\cB}(a[h\HleftB b^*]) \bigr)
        =
        \rho_{\cX}
        \bigl( s_{\cB}(h\HleftB b^*) \bigr).
    \]
    On the other hand,  we have
    \[
        r_{\cH}( h\HrightB b^* 
        )
        =
        r_{\cH}( h\HrightX q_{\cB}(b^*)
        )
        \overset{\eqref{item:L1}}{=}
       \rho_{\cX}\bigl(s_{\cX}(h\HleftX q_{\cB}(b^*) )\bigr)
       \overset{\ref{item:leftFell.1}}{=}
        \rho_{\cX}
        \bigl( s_{\cB}(h\HleftB b^*) \bigr),
    \]
    as required. The inner product is thus well defined and lands in the right space.

        \medskip

    Since multiplication on $\cB$, $\HleftB$, and $\BrightG$ are linear, we see that $\linner{\cA}{\cdot}{\cdot}$ is linear in the first and conjugate linear in the second coordinate. To check that it satisfies the other required properties, let $x\coloneqq q_{\cA}(a)$ and $y\coloneqq q_{\cA}(b)$ and $h\in\cH$ be as above.

    \medskip

    For \ref{item:FE:ip:fiber}, we must check that, when $q_{\cA}\bigl(\lip\cA<a,b>\bigr)
      \in (\cX/\cG) \bowtie \cH$ acts on the left of $y$, it yields $x$.
      By the definition of $q_{\cA}$ (see Definition~\ref{df.ss.product.bundle.left}) and our computations above, we have
      \begin{align*}
        q_{\cA}\bigl(\lip\cA<a,b>\bigr)
        &=
        \bigl(
            q_{\cB/\cG}\bigl([a(h\HleftB b^*)]  \BrightG  \cG\bigr), h\HrightB b^*
        \bigr)
        =
        \bigl(
            q_{\cB} \bigl( a(h\HleftB b^*)  \bigr)\XrightG  \cG, h\HrightB b^*
        \bigr)
        \\
        &=
        \bigl(
            [x (h\HleftX y\inv)]\XrightG  \cG, h\HrightX y\inv
        \bigr).
      \end{align*}
        By the definition of the left $(\cX/\cG)\bowtie \cH$-action on $\cX$ (Proposition~\ref{prop.actions.X}),
        \begin{align*}
        \bigl(
            [x (h\HleftX y\inv)]\XrightG  \cG, h\HrightX y\inv
        \bigr)\cdot y 
        &= x(h\HleftX y\inv )[(h\HrightX y\inv ) \HleftX y] 
        = x(h\HleftX (y\inv y)) &&\text{(by \eqref{item:L4})} \\
        &= x(h\HleftX s_{\cX}(y)) = x s_{\cX}(x)=x&&\text{(by choice of $h$).}
        \end{align*}

    \medskip
    
    To show \ref{item:FE:ip:adjoint}, we must prove that   $\lip\cA<a,b>^{*}=\lip\cA<b,a>$
        Since $s_{\cX}(y)=h\inv \HleftX s_{\cX}(x)$, we have
        \[\linner{\cA}{b}{a} =
        \bigl(
            \bigl[b(h\inv \HleftB a^*)\bigr]  \BrightG  \cG, h\inv \HrightX x\inv 
        \bigr).\]
        Using the definition of the involution on $\cA$ (see Definition~\ref{df.ss.product.bundle.left}), 
        we can compute the adjoint of 
        \[
            \linner{\cA}{a}{b}=
            \bigl(
                [a(h\HleftB b^*)]  \BrightG  \cG, h\HrightB b^* 
            \bigr).
        \]
        Its
        $\cB/\cG$-component
        \begin{equation}\label{eq:BG comp of lip a b}
            \Bigl[(h\HrightX y\inv )\inv \HleftB \bigl(a[h\HleftB b^*]\bigr)^* \Bigr]
            \BrightG  \cG
            \quad\text{ has to equal }\quad
            \bigl[b(h\inv \HleftB a^*)\bigr]  \BrightG  \cG
        \end{equation}
        and that its $\cH$-component
        \begin{equation}\label{eq:H comp of lip a b}
            (h\HrightX y\inv )\inv  \HrightX \bigl(x[h\HleftX y\inv ]\bigr)\inv \quad\text{ has to equal }\quad
            h\inv \HrightX x\inv .
        \end{equation}
        If $z\coloneqq h\HleftX y\inv$ and $k\coloneqq h\HrightX y\inv$, then by \eqref{item:L9}, we have $k\inv = h\inv \HrightX z$. Thus, the asserted equality in \eqref{eq:H comp of lip a b} is easily seen: % using \eqref{item:L3}:
        \begin{align*}
            (h\HrightX y\inv )\inv  \HrightX (x [h\HleftX y\inv] )\inv
            &=
            k\inv  \HrightX (x z )\inv
            \\
            &=
            (h\inv \HrightX  z )
            \HrightX
            (x z )\inv
            \overset{\eqref{item:L3}}{=}
            h\inv \HrightX x\inv .
        \end{align*}
        For the asserted equality in \eqref{eq:BG comp of lip a b}, we compute 
        \[
            \bigl(a[h\HleftB b^*]\bigr)^*
            =
            [h\HleftB b^*]^* a^*
            \overset{\ref{item:leftFell.5}}{=}
            \bigl[(h\HrightB b^*)\HleftB b\bigr] a^*.
        \]
        If $c\coloneqq (h\HrightB b^*)\HleftB b = k\HleftB b$, then we have for the left-hand side of  \eqref{eq:BG comp of lip a b}
        \begin{align}\label{eq:LHS of B/G component}
            (h\HrightX y\inv)\inv   \HleftB \bigl(a[h\HleftB b^*]\bigr)^*
            &=
            k\inv   \HleftB (c a^*)
            \overset{\ref{item:leftFell.4}}{=}
            (k\inv   \HleftB c)
            \bigl[
                (k\inv \HrightB c)
                \HleftB a^*
            \bigr]
            .
        \end{align}
        Since $k= h\HrightX y\inv = h\HrightB b^*$, we have 
        \begin{align*}
            k\inv   \HleftB c
            &=
            k\inv   \HleftB [(h\HrightB b^*)\HleftB b]
            \overset{\ref{item:leftFell.2}}{=}
            \bigl[k\inv   (h\HrightB b^*)\bigr]\HleftB b
            \overset{\ref{item:leftFell.3}}{=}
            b.
        \end{align*}
         On the other hand, $q_{\cB}(c)=k\HleftX y$ by \eqref{cond.compatible.Fell.BC1}, so that
        \begin{align*}
            k\inv \HrightB c
            &=
            k\inv \HrightX (k\HleftX y)
            \overset{\eqref{item:L9}}{=}
            (k\HrightX y)\inv
            =
            \bigl([h\HrightX y\inv]\HrightX y\bigr)\inv
            \overset{\eqref{item:L3}}{=}
            h\inv.
        \end{align*}
        Plugging the results of our last computations back into Equation~\eqref{eq:LHS of B/G component}, we get
        \begin{align*}
            (h\HrightX y\inv)\inv   \HleftB \bigl(a[h\HleftB b^*]\bigr)^*
            &=
            b
            \bigl(
                h\inv
                \HleftB a^*
            \bigr),
        \end{align*}
        which is, on the nose, what we needed for \eqref{eq:BG comp of lip a b}.

    \medskip

    Lastly, for \ref{item:FE:ip:C*linear}, we need that the inner product is $\cA$-linear in the first component, so let $(\Theta,k)$ be an arbitrary element of $\cA$ with $s_{\cA}(\Theta,k)=\mathfrak{r}_{\cB}(a)$. If $c\in\Theta$ is such that $s_{\cB}(c)=r_{\cB}(k\HleftB a)$, then our definition of the left $\cA$-action on $\cB$ (see \ref{thm.imprimitivity.main}\ref{thm:it:left action}) yields
    \(
        (\Theta,k)\cdot a = c[k\HleftB a].
    \)
    Note that $m\coloneqq (k\HrightX x)h$ is the unique element of $\cH$ such that $s_{\cB}(c[k\HleftB a])=m\HleftX s_{\cB}(b)$, since
    \[
        s_{\cB}(c[k\HleftB a])
        =
        s_{\cB}(k\HleftB a)
        \overset{\eqref{item:L10}}{=}
        (k\HrightX x)\HleftX s_{\cB} (a)
        =
        (k\HrightX x)\HleftX [h\HleftX s_{\cB} (b)].
    \]
    We have
    \begin{equation}\label{eq:lip:A-lin,1}
        \linner{\cA}{(\Theta,k)\cdot a}{b}=
    \Bigl(
        \bigl[
            (c[k\HleftB a])(m\HleftB b^*)
        \bigr]\BrightG  \cG, m\HrightB b^* 
    \Bigr).
    \end{equation}
    On the other hand, according to Definition~\ref{df.ss.product.bundle.left}, the product of
    \begin{align*}
        (\Theta,k)\linner{\cA}{a}{b}
        &=
        (\Theta,k)\,
        \bigl(
            [a(h\HleftB b^*)]  \BrightG  \cG, h\HrightB b^*
        \bigr)
    \end{align*}
    in $\cA$ has $\cB/\cG$-component
    \begin{align}\label{eq:B/G component of Theta times lip}
        \Theta \biggl[k\qHleftB \Bigl(\bigl[a(h\HleftB b^*)\bigr]  \BrightG  \cG\Bigr)\biggr]
        =
        \Theta \biggl[\Bigl(k\HleftB \bigl[a(h\HleftB b^*)\bigr]\Bigr)  \BrightG  \cG\biggr]
        .
    \end{align}
    We compute
    \begin{align*}
        k\HleftB \bigl[a(h\HleftB b^*)\bigr]
        &=
        (k\HleftB a)
        \bigl[
            (k\HrightB a)\HleftB (h\HleftB b^*)
        \bigr]
        &&\text{(by \ref{item:leftFell.4})}
        \\
        &=
        (k\HleftB a)
        (m\HleftB  b^*)
        &&\text{(by \ref{item:leftFell.2})}.
    \end{align*}
    Note that $c\in\Theta$ was chosen such that $s_{\cB}(c)=r_{\cB}(k\HleftB a)$, so that the above computation together with the definition of the multiplication in $\cB/\cG$ (cf.\ Proposition~\ref{prop.Fell bdl structure on quotient bundle}) shows that the $\cB/\cG$-component of $(\Theta,k)\linner{\cA}{a}{b}$ is
    \begin{align*}
        \Theta \Bigl[
            \bigl((k\HleftB a)(m\HleftB  b^*)\bigr)  \BrightG  \cG\Bigr]
        % \\
        &=
       \Bigl[
            c[(k\HleftB a)(m\HleftB  b^*)]
        \Bigr]\BrightG  \cG,
    \end{align*}
    which, by associativity of the multiplication on $\cB$, is exactly the $\cB/\cG$-component of $\linner{\cA}{(\Theta,k)\cdot a}{b}$; see \eqref{eq:lip:A-lin,1}.
    
    Similarly, the  $\cH$-component of $(\Theta,k)\linner{\cA}{ a}{b}$ is given by
    \begin{align*}
        &\biggl[k\qHrightX q_{\cB/\cG}\Bigl(\bigl[a(h\HleftB b^*)\bigr]  \BrightG  \cG\Bigr) \biggr](h \HrightB b^*)
        \\
        &\qquad=
        \biggl[k\qHrightX \Bigl(q_{\cB}\bigl(a(h\HleftB b^*)\bigr)  \BrightG  \cG\Bigr) \biggr](h \HrightB b^*)
        &&\text{(def'n of $q_{\cB/\cG}$)}
        \\
        &\qquad=
        \Bigl[k\HrightX q_{\cB}\bigl(a(h\HleftB b^*)\bigr) \Bigr](h \HrightB b^*)
        &&\text{(def'n of $\qHrightX$)}
        \\
        &\qquad=
        \Bigl[(k\HrightX x )\HrightX q_{\cB}(h\HleftB b^*)\Bigr](h \HrightB b^*)
        &&\text{(by \eqref{item:L3} and \ref{cond.F1})}
        \\
        &\qquad=
        \Bigl[(k\HrightX x )\HrightX (h\HleftX y\inv)\Bigr](h \HrightX y\inv)&&\text{(by \eqref{cond.compatible.Fell.BC1} and def'n of $y$)}
        \\
        &\qquad=
        [(k\HrightX x)h]\HrightX y\inv
        &&\text{(by \eqref{item:L6})}
    \end{align*}
    which is  exactly $
        m\HrightB b^*$, as needed.
\end{proof}

\begin{lemma}[Regarding \ref{item:FE:ip:compatibility}]\label{lem:thm:ips-imprim condition}
   The inner products on the \USCBb\ $\cB$ satisfy \ref{item:FE:ip:compatibility}, i.e., $\lip\cA<a,b>\cdot c=a\cdot\rip\cC<b,c>$ whenever both inner products make sense.
\end{lemma}

\begin{proof}
        Let $a\in \cB_x$, $b\in \cB_y$, and $c\in \cB_z$. For the inner products to be defined, we require $\mathfrak{s}_{\cB}(a)=\mathfrak{s}_{\cB}(b)$ and  $\mathfrak{r}_{\cB}(b)=\mathfrak{r}_{\cB}(c)$,  so
        there 
        exist 
        $h\in\cH$ and $t\in\cG$ such that such that $s_{\cX}(x)=h\HleftX s_{\cX}(y)$ resp.\ $r_{\cX}(y)\XrightG t = r_{\cX}(z)$, so that
        \begin{align*}
            \linner{\cA}{a}{b}=
            \bigl([a(h\HleftB b^*)]  \BrightG  \cG, h\HrightB b^* \bigr)
            \quad\text{ and }\quad
            \rinner{\cC}{b}{c}=\bigl(b^* \BleftG t, \cH \HleftB [ (b^* \BrightG t)c]\bigr)
            .
        \end{align*}
    If we let $\Theta\coloneqq [a(h\HleftB b^*)]  \BrightG  \cG$, then
    \begin{align*}
        s_{\cA}\bigl(\linner{\cA}{a}{b}\bigr)
        &=
        (h\HrightB b^*)\inv \qHleftX s_{\cB/\cG} (\Theta)
        &&\text{(cf.\ Def.~\ref{def.ZSProduct.groupoid.left} and Rmk.~\ref{rm.unit space of ssp})}
        \\
        &=
        (h\HrightX y\inv)\inv \qHleftX s_{\cX/\cG}\bigl([x(h\HleftX y\inv )]\XrightG \cG \bigr)
        \\
        &=
        \bigl[
            (h\HrightX y\inv)\inv \HleftX s_{\cX}(x[h\HleftX y\inv ])
        \bigr]\XrightG \cG
        &&\text{(def'n of $s_{\cX/\cG}$ and $\qHleftX$)}.
    \end{align*}
    Since
    \begin{align*}
        s_{\cX}\bigl(x[h\HleftX y\inv]\bigr) 
        &=
        s_{\cX}(h\HleftX y\inv)
        \overset{\eqref{item:L10}}{=}
        [h\HrightX y\inv ] \HleftX r_{\cX}(y),
    \end{align*}
    it follows that
     \begin{align*}
        s_{\cA}\bigl(\linner{\cA}{a}{b}\bigr)
        &=
        r_{\cX}(y)\XrightG \cG
        =
        r_{\cB}(c)\XrightG \cG
        =
        \mathfrak{r}_{\cB}(c),
    \end{align*}
    so that $\linner{\cA}{a}{b}\cdot c$ is indeed defined. Moreover, we see that $t$ can act on the right of $a(h\HleftB b^*)$ and that 
    \begin{align*}
        s_{\cB}
        \bigl(
            [a(h\HleftB b^*)]\BrightG t
        \bigr)
        &=
         s_{\cX}\bigl(x[h\HleftX y\inv]\bigr)\XrightG t
        \overset{\ref{cond.compatible.c1}}{=}
         [h\HrightX y\inv ] \HleftX [r_{\cX}(y)\XrightG t] 
    \\& =
         [h\HrightX y\inv ] \HleftX r_{\cX}(z) 
         \overset{\eqref{item:L10}}{=}
         r_{\cX}\bigl( [h\HrightX y\inv ] \HleftX  z\bigr)
        .
    \end{align*}
    Thus, $[a(h\HleftB b^*)]\BrightG t$ is the (unique) element of $\Theta$ whose image under $s_{\cB}$ equals $r_{\cB} ([h\HrightB b^*] \HleftB c)$, so that
    \begin{align}\label{eq:lip:cdot c}
        \linner{\cA}{a}{b}\cdot c
        &=
        \Bigl(\bigl[a(h\HleftB b^*)\bigr]\BrightG t\Bigr) \bigl([h\HrightB b^*]\HleftB c\bigr)
        .
    \end{align}
    A similar argument shows that $a\cdot \rinner{\cC}{b}{c}$ is well defined and that
    \begin{align}\label{eq:lip:a cdot}
    a \cdot \rinner{\cC}{b}{c}
    = 
    \bigl(a\BrightG[b^* \BleftG t]\bigr) \Bigl(h\HleftB\bigl[(b^* \BrightG t)c\bigr]\Bigr).
    \end{align}
    We compute the first element of the product in \eqref{eq:lip:cdot c} to be
    \begin{align*}
        [a(h\HleftB b^*)]\BrightG t
        &=
        \Bigl(a\BrightG \bigl[(h\HleftB b^*)\BleftG t\bigr]\Bigr)
        \bigl[(h\HleftB b^*)\BrightG t\bigr]
        &&\text{(by \eqref{item:R4})}
        \\
        &=
        \bigl(a\BrightG [ b^*\BleftG t]\bigr)
        \bigl[h\HleftB (b^*\BrightG t)\bigr]
        &&\text{(by \ref{cond.compatible.c2} and \eqref{cond.compatible.Fell.BC1})}
    \intertext{and its second element to be}
        [h\HrightB b^*]\HleftB c
        &=
        \bigl[h\HrightB (b^*\BrightG t)\bigr]\HleftB c
        &&\text{(by \ref{cond.compatible.c3})},
    \intertext{so that it follows from \eqref{eq:lip:a cdot} that}
        \linner{\cA}{a}{b}\cdot c
        &=
        \bigl(a\BrightG [ b^*\BleftG t]\bigr)
        \bigl[h\HleftB (b^*\BrightG t)\bigr] \Bigl(\bigl[h\HrightB (b^*\BrightG t)\bigr]\HleftB c\Bigr)
        \\
        &=
        \bigl(a\BrightG [ b^*\BleftG t]\bigr)
        \Bigl(h\HleftB \bigl[(b^*\BrightG t) c\bigr]\Bigr)
        =
        a \cdot \rinner{\cC}{b}{c}
        &&\text{(by \ref{item:leftFell.4})}.
        \qedhere
    \end{align*}
\end{proof}

\begin{lemma}[Regarding \ref{item:FE:SMEs}]\label{lem:thm:SMEs}
    With the induced actions, 
    each $B(x)$ is a $A\bigl(\mathfrak{r}
    (x)\bigr)\sme
    C\bigl(\mathfrak{s}
    (x)\bigr)$-\ib.
\end{lemma}

\begin{proof}
    For $x\in\cX$, we have (see the definitions of $\mathfrak{s}$ and $\mathfrak{r}$ in Proposition~\ref{prop.actions.X}):
    \[
    \mathfrak{r}(x)=r_{\cX}(x)\XrightG \cG \quad\text{and}\quad
    \mathfrak{s}(x)=\cH\HleftX s_{\cX}(x).\]
    Recall that here, we have identified the unit spaces of $(\cX/\cG)\bowtie \cH$ and  $\cG\bowtie (\cH\backslash \cX)$ with those of $\cX/\cG$ resp.\ $\cH\backslash\cX$; cf.\ Remark~\ref{rm.unit space of ssp}. Thus, if we want to think of $\mathfrak{r}(x)$ and $\mathfrak{s}(x)$ in $(\cX/\cG)\bowtie \cH$ resp.\ $\cG\bowtie (\cH\backslash \cX)$, we must  write
    \[
    \mathfrak{r}(x)
    =
    \bigl(r_{\cX}(x)\XrightG \cG, \rho_{\cX}(x)\bigr)  \quad\text{and}\quad
    \mathfrak{s}(x)
    =
     \bigl(\sigma_{\cX}(x) ,\cH \HleftX s_{\cX}(x)\bigr) ,\]
     where we have used that
     $\rho_{\cX}\z\circ r_{\cX} = \rho_{\cX}$ and $\sigma_{\cX}\z\circ s_{\cX} = \sigma_{\cX}$ by definition of the right-hand sides.
     
     Now, recall that $\cB$ is a Fell bundle equivalence between $\cB$ and itself; in particular, we know that each $B(x)$ is a $B\bigl(r_{\cX}
    (x)\bigr)\sme
    B\bigl(s_{\cX}
    (x)\bigr)$-\ib. We claim that the fiber $A(\mathfrak{r}(x))$ is isomorphic to $B(r_{\cX}(x))$ and likewise that $C(\mathfrak{s}(x))$ is isomorphic to $B(s_{\cX}(x))$, and that these isomorphisms turn the canonical
      $B(r(x))\sme B(s(x))$-\ib\ $B(x)$
    into our
      bi-Hilbertian $A(\mathfrak{r}(x))\sme C(\mathfrak{s}(x))$-module $B(x)$,
    proving that the latter is an \ib\ also. We will do so for the fibre $A(\mathfrak{r}(x))$ of $\cA=(\cB/\cG) \BbowtieH \cH$.
    
    \smallskip 
    Define 
    \[\psi\colon B(r_{\cX}(x)) \to A(\mathfrak{r}(x))
    \quad\text{by}\quad \psi(a)=(a \BrightG  \cG, \rho_{\cX}(x)).
    \]
    This map is clearly linear, $*$-preserving, surjective, and injective, since the norm on $\cA_{\mathfrak{r}(x)}$ is inherited from $\cB_{r(x)}$. Therefore, $\psi$ defines an isomorphism of C*-algebras. 
    
    Notice that this isomorphism indeed turns the left $A(\mathfrak{r}(x))$-action on $B(x)$ into the left $B(r_{\cX}(x))$-multiplication on $B(x)$: if $b\in B(x)$ and $a\in \cB_{r_{\cX}(x)}$, then 
    \[
        s_{\cB}(a) = s_{\cX} (r_{\cX}(x))
        =
        r_{\cX}(x)
        =
        r_{\cB}(b)
        \overset{\eqref{cond.compatible.Fell.BC3}}{=}
        r_{\cB}(\rho_{\cX}(x)\HleftB b),
    \]
    proving that $a$ is the unique element in $a\BrightG \cG$ such that $s_{\cB}(a)=r_{\cB}(\rho_{\cX}(x)\HleftB b)$, so that the definition of the left $\cA$-action on $\cB$ implies 
    \[
    \psi(a)\cdot b
    =
    (a \BrightG  \cG, \rho_{\cX}(x))\cdot b=ab,\]
    as claimed.
\end{proof}

\begin{proof}[Proof of Theorem~\ref{thm.imprimitivity.main}]
    The groupoids $(\cX/\cG)\bowtie\cH$ and $\cG\bowtie(\cH\backslash\cX)$ are \LCH: In Proposition~\ref{prop.gpd structure on orbit space}, we have seen that the quotient of \LCH\ groupoids is again \LCH, and clearly so is the \ssp\ of such groupoids. 
    
    We have seen that $\cX$ is a groupoid equivalence between $(\cX/\cG) \bowtie \cH$ and $\cG\bowtie (\cH\backslash \cX)$ (Theorem~\ref{thm.groupoid.eq} and Proposition~\ref{prop.actions.X}), and that $\cA$ and $\cC$ are Fell bundles by Definition~\ref{df.ss.product.bundle.left} and Remark~\ref{rm.df.ss.product.bundle.right}. 
    Moreover, as $\cB$ is assumed to be saturated, it follows from Remark~\ref{rm.saturated is inherited by quotient} that $\cH\backslash\cB$ and $\cB/\cG$ are saturated also. Consequently, it follows from  Remark~\ref{rm.saturated is inherited by bowtie} that $\cA$ and $\cC$ are saturated, and so we are dealing with the right ingredients.
    
    We have then checked that all conditions in  Definition~\ref{df.FBequivalence} are satisfied. Indeed,
    \begin{itemize}[leftmargin = 2cm]
        \item[Re \ref{item:FE:actions}:] Lemma~\ref{lem:thm:actions} shows that the formulas in~\ref{thm.imprimitivity.main}\ref{thm:it:left action} and~\ref{thm.imprimitivity.main}\ref{thm:it:right action} define actions in the sense of Definition~\ref{df.USCBb-action}, and Lemma~\ref{lem:thm:actions commute} shows that they commute.
        \item[Re \ref{item:FE:ip}:]  Lemma~\ref{lem:thm:ips} shows that the formulas in~\ref{thm.imprimitivity.main}\ref{thm:it:left ip} and~\ref{thm.imprimitivity.main}\ref{thm:it:right ip} define inner products, while Lemma~\ref{lem:thm:ips-imprim condition} shows that they satisfy the im\-prim\-i\-tiv\-ity condition \ref{item:FE:ip:compatibility}, and finally
        \item[Re \ref{item:FE:SMEs}:]  Lemma~\ref{lem:thm:SMEs} shows that each $B(x)$ is an \ib.\qedhere
    \end{itemize}
\end{proof}

\begin{corollary}\label{cor.Fell bdl equiv plus Haar makes SME}
We assume all conditions in Assumption~\ref{all assumptions}. 
\txtrepeat{Assumption on Haar in cor.gpd equiv plus Haar makes SME}, 
and that $\cH$ and $\cG$ also have Haar systems.  Then
the Fell bundle C*-algebras $\textrm{C}^*((\cB/\cG) \BbowtieH \cH)$ and $\textrm{C}^*(\cG\GbowtieB (\cH\backslash \cB))$ are Morita equivalent.
\end{corollary}
Recall from Corollary~\ref{cor:ssp gpd:r-discrete, etale} that the assumptions regarding Haar systems are satisfied if $\cX,\cH,\cG$ are \etale.

\begin{proof}
    All Fell bundles in sight are saturated, since $\cB$ is saturated.
    By Theorem~\ref{thm.imprimitivity.main}, the Fell bundles $(\cB/\cG) \BbowtieH \cH$ and $\cG\GbowtieB (\cH\backslash \cB)$ are equivalent. Recall from Corollary~\ref{cor.gpd equiv plus Haar makes SME} that both groupoids $(\cX/\cG) \bowtie \cH$ and $\cG\GbowtieB (\cH\backslash \cX)$ allow Haar systems, so that the claim now follows from an application of 
    \cite[Theorem 6.4]{MW2008}.
\end{proof}

One immediate application is the one-sided im\-prim\-i\-tiv\-ity theorem by setting $\cG=\{e\}$. 
\begin{corollary}\label{cor.imprimitivity.one.side}Let~$\cX$ be a groupoid and~$\cB$ be a Fell bundle over $\cX$. 
Suppose~$\cH$ has a \ssla\ on the Fell bundle~$\cB$, and that the action of $\cH$ on $\cX$ is free and proper.
Then~$\cB$ is a Fell bundle equivalence between $\cB \BbowtieH \cH$ and $\cH\backslash \cB$. In particular, 
    if $\cX$ has a $\HleftX$-invariant Haar system and $\cH$ admits any Haar system, then
$\textrm{C}^*(\cB \BbowtieH \cH)$ and $\textrm{C}^*(\cH\backslash \cB)$ are Morita equivalent. 
\end{corollary}

\begin{example}[combination of previous examples]\label{ex:combo of -ex:ZS of CP- and -ex:CP with quotient}
    \txtrepeat{assumptions of ex:CP with quotient}
    We have stated in Example~\ref{ex:CP with quotient} that   
    \begin{equation}\label{eq:quot FB}
        \cH\backslash\cB (\cA, \cX, \tilde{\alpha})
        \cong
        \cB (\cA, \cH\backslash\cX, \alpha )
         ,
    \end{equation}
    where
    $\tilde{\alpha}=\alpha\circ q$ for $q$ the quotient map.
    On the other hand, if we let $p\colon \cX\bowtie\cH\to \cX$ be the projection onto the first component, then $(\cA, \cX\bowtie \cH, \tilde{\alpha}\circ p)$ is a groupoid dynamical system  on $\cA$ whose restriction to $\cX$ is $\tilde{\alpha}$. By Example~\ref{ex:ZS of CP}, $\cH$ thus has a \ssla\ on $\cB(\cA, \cX, \tilde{\alpha})$ given by $h\HleftB (a,x) \coloneqq (a,h\HleftX x)$, and we have
     \begin{equation}\label{eq:ZS FB}
        \cB(\cA, \cX, \tilde{\alpha}) \BbowtieH \cH \cong \cB(\cA, \cX\bowtie \cH, \tilde{\alpha}\circ p).
    \end{equation}
    By Corollary \ref{cor.imprimitivity.one.side}, the Fell bundles on the left-hand sides of~\eqref{eq:quot FB} and~\eqref{eq:ZS FB} are equivalent, so that
    \(
        \cB (\cA, \cH\backslash\cX, \alpha ) 
     \) and \(
        \cB(\cA, \cX\bowtie \cH, \tilde{\alpha}\circ p)
    \) are also equivalent. If the groupoids have appropriate Haar systems (for example, if they are \etale), then this implies that the groupoid crossed product
    \(
        \cA\rtimes_{\alpha }(\cH\backslash\cX)
    \) is Morita equivalent to
    \( \cA\rtimes_{\tilde{\alpha}\circ p}( \cX\bowtie \cH)\).
\end{example}

\section{Examples on \DR\ Groupoids}\label{sec.Deaconu}

One interesting class of \ss\ action arises from \DR\ groupoids \cite[Section 3]{SW2016}, and we devote the last section to examples arising from this class of groupoids. It is observed in \cite[Proposition 5.1]{BPRRW2017} that a \DR\ groupoid generated by a pair of $*$-commuting endomorphisms has a \ZS\ product structure. We will describe this as \ssp\ in more detail, and apply our main result on equivalent groupoids (Theorem~\ref{thm.groupoid.eq}) in this context. 

We first give a brief overview of \DR\ groupoids. For $Y$ a topological space, we say a map $\sigma\colon Y\to Y$ is an {\em endomorphism} if it is a surjective local homeomorphism, and we denote the collection of all endomorphisms on $Y$ by $\Endo(Y)$.
We note that an endomorphism may not be injective. Suppose $\theta\colon  \mathbb{N}^{k} \to \Endo(Y)$ is a semigroup action on $Y$ by endomorphisms. The  \DR\ groupoid, denoted $Y\rtimes_\theta \mathbb{N}^{k}$, is defined as
\[Y\rtimes_\theta \mathbb{N}^{k} = \left\{(x,p-q,y)\in Y\times \mathbb{Z}^{k} \times Y: \theta_p(x) = \theta_q(y)\right\}\]
with multiplication and inverse given by
\begin{align*}
    (x,p-q,y) (y, m-n,z) &= (x, (p+m)-(q+n), z)%, \quad\text{if }y=w 
    ,\\
    (x,p-q,y)\inv &= (y, q-p, x).
\end{align*}
Its range and source maps are therefore given by
\begin{align*}
    r(x,p-q,y)&=(x,0,x), \\
    s(x,p-q,y)&=(y,0,y),
\end{align*}
and its unit space is identified as $\{(x,0,x):x\in Y\}\approx Y$.
    We give $Y\rtimes_{\theta}\mathbb{N}^{k}$ the topology induced by the basic open sets $ Z_{\theta}(U, m, n, V )$, defined for open subsets $U,V\subseteq Y$ and vectors $m,n\in\mathbb{N}^{k}$ by
    \[
        Z_{\theta}(U, m, n, V ) \coloneqq \{(x, m - n, y) : x \in U , y \in V \text{ and } \theta_{m} x = \theta_{n} y\}.
    \]
This makes $Y\rtimes_{\theta}\mathbb{N}^{k}$ a \LCH\ \etale\ groupoid \cite[Lemma 3.1.]{SW2016}.

\smallskip

To two commuting elements  $S, T\in \Endo(Y)$, we can naturally associate an $\mathbb{N}^2$-action on $Y$ given by $\theta_{p,m}(x)=T^p S^{m} x$. We let $\cK=Y\rtimes_\theta \mathbb{N}^2$ be the corresponding \DR\ groupoid. Each of the endomorphisms $S$ and $T$ corresponds to an $\mathbb{N}$-action on $Y$, so we can define their respective \DR\ groupoid as
\begin{align*}
    \cH &= Y\rtimes_T \mathbb{N} =\left\{(x, p-q, y)\in Y\times \mathbb{Z}
    \times 
    Y: T^p x= T^q y\right\}, \\
    \cX %
    &= Y\rtimes_S \mathbb{N} =\left\{(x, m-n, y)\in Y\times \mathbb{Z}
    \times 
     Y: S^{m} x= S^{n} y\right\}. 
\end{align*}

From now on, we fix $S$ and $T$ and further assume that they {\em $*$-commute}: not only do we have $ST=TS$, but whenever $Sx=Ty$ for some $x,y\in Y$, then there exists a unique $z\in Y$ such that $Tz=x$ and $Sz=y$. Note that, for all integers $p,q\geq 1$,
$S^p, T^q$  are also $*$-commuting. 
It was observed in \cite[Proposition 5.1]{BPRRW2017} that, in this setting, $\cK$ can be realized as the \ZS\ product groupoid $\cX %
    \bowtie\cH$.
The proof uses a unique decomposition property, but did not describe the actions of $\cX %
    $ and $\cH$ on each other %without describing the actions 
explicitly, so we start by giving such a description. %first describe the \ss\ actions between $\cH$ and $\cG$ that build $\cK$.

\begin{lemma}\label{lemma.Deaconu.actions} Let $\cH$ and $\cX %
    $ be the \DR\ groupoids described above. Then the following maps define a \ssla\ of $\cH$ on $\cX %
    $, where $w\in Y$ is the unique element that satisfies $S^{n} w = S^{m} x$ and $T^p w = T^q z$:
\begin{align*}
    \cH\cart \cX %
    \colon &&(x, p-q, y)\HleftX (y, m-n, z) & = (x, m-n, w)\in \cX %
    \\
    \cH\calb \cX %
    \colon && (x, p-q, y) \HrightX (y, m-n, z) & = (w, p-q, z)\in \cH
\end{align*}
\end{lemma}
\begin{proof} 
First, the element $w\in Y$ exists because
 \[
T^p (S^{m} x)=S^{m} (T^p x) = S^{m}  (T^q y)  = T^q (S^{m} y) = T^q (S^{n} z) = S^{n}(T^q  z).
\]
 % \[T^p (S^{m} x) = S^{m} (T^p y) = T^p (S^{m} y) = T^q (S^{n} z)=S^{n} (T^q z),\]
We apply the $*$-commuting condition for $T^p$ and $S^{n}$ to obtain the desired $w$. 

From \cite[Proposition 5.1]{BPRRW2017}, $Y\rtimes_\theta \mathbb{N}^2$ is an internal \ZS\ product of the groupoids $\cH$ and $\cX %
    $. Here, we embed $\cH$ and $\cX %
    $ as subgroupoids of $Y\rtimes_\theta \mathbb{N}^2$ by identifying $(x,k,y)\in \cH$ and $(y,\ell,z)\in\cX$ as $(x,(k,0),y)$ resp.\ $(y,(0,\ell), z)$ in $ Y\rtimes_\theta \mathbb{N}^2$. 

It follows from \cite[Proposition 3.4]{BPRRW2017} that the corresponding \ss\ actions are uniquely determined by the equation
\[gh=(h\HleftX g)(h\HrightX g), \quad h\in\cH, g\in\cX .\]
Therefore, it suffices to verify that the \ssla\ of $\cH$ on $\cX %
    $ satisfies this equation. 

Pick any $x,y,z\in Y$ and $p,q,m,n\in \mathbb{Z}$ such that $$(x,(p-q,0),y)\in \cH\subseteq Y\rtimes_\theta \mathbb{N}^2\quad\text{and}\quad(y,(0, m-n),z)\in \cG\subseteq Y\rtimes_\theta \mathbb{N}^2.$$ If $w\in Y$ is the unique element that satisfies $S^{n} w= S^{m} x$ and $T^p w=T^q z$, then
\begin{align*}
    (x,(p-q,0),y)(y, (0,m-n), z) &= (x, (p-q, m-n), z) \\
    &= (x, (0, m-n), w)(w, (p-q,0), z). \qedhere
\end{align*}
\end{proof}

For a map $T\colon Y\to Y$, we say that $x\in Y$ is a {\em periodic point} for $T$ if $T^{k} x = x$ for some $k\in\mathbb{N}^{\times}$. If no such $x$  exists,  we call $T$ {\em non-periodic}.

\begin{lemma}\label{lemma.Deaconu.free} The \ssla\ $\HleftX$ defined in Lemma~\ref{lemma.Deaconu.actions} is free if and only if $T$ is non-periodic. 
\end{lemma}

\begin{proof} 
Suppose $\HleftX$ is not free, so there exists $x,y,z$ and $p\neq q$ such that 
% $T$ is non-periodic. Supposing 
$(x,p-q,y)\HleftX (y,m-n,z)=(y,m-n,z)$. By definition of $\HleftX$, this equality forces $x=y$. Since $(x,p-q,x)\in\cH$ by assumption, this implies $T^p x=T^q x$, so since $p\neq q$, $T$ has a periodic point.
% . But $T$ is non-periodic, implying $p=q$, and thus $(x,p-q,y)=(x,p-p,x)\in \cH\z $.  

Conversely, assume $T$ has  a periodic point $x$, so there exists $k>0$ with $T^{k} x= x$. In this case, $(x,k,x)\in \cH\setminus \cH\z$ and $(x,0,x)\in \cX $. One can easily verify that $(x,k,x)\HleftX (x,0,x)=(x,0,x)$.
\end{proof}

While the action $\HleftX$ in Lemma~\ref{lemma.Deaconu.actions} may not be a proper map in general, the examples on certain classes of $2$-graphs that we shall consider later satisfy this property. With properness, Corollary~\ref{cor.gprd.one.sided} implies that the \ssp\ groupoid $\cH\bowtie \cX %
    \cong Y\rtimes \mathbb{N}^2$ is equivalent to the quotient groupoid $\cH\backslash \cX %
    $, 
    which the authors conjecture is another \DR\ groupoid.

\begin{conjecture}\label{conj.DCGroupoid}
    Partition $Y$ into the equivalence classes given by $[z]_{T}=\cup_{p,q\in\mathbb{N}}\{w\in Y: T^p w=T^q z\}$. On the quotient space $Y_{T}$, define $\widehat{S}\colon Y_{T}\to Y_{T}$  by $\widehat{S}([z]_T)=[Sz]_T$.  If $T$ is non-periodic, then the map
    \begin{equation}\label{eq:Phi}
        \Phi\colon\cH\backslash\cX %
        \to Y_{T}\rtimes_{\widehat{S}}\mathbb{N},
        \quad
        \cH\HleftX(y,k,z)\mapsto
        ([y]_{T}, k, [z]_{T}),
    \end{equation}
    is an (algebraic) isomorphism of groupoids. If, furthermore, the \ssla\ $\HleftX$ defined in Lemma~\ref{lemma.Deaconu.actions} is proper  and $\widehat{S}$ is locally injective (so that both groupoids are \LCH), then $\Phi$ is a homeomorphism.
\end{conjecture}
While it is easy to show that $\Phi$ is a continuous bijection that preserves the groupoid structure, we found no reason for $\Phi$ to be open. We are furthermore unsure under which circumstances $\HleftX$ is proper  or $\widehat{S}$ locally injective. If the conjecture is true, then it would follow from Corollary~\ref{cor.gprd.one.sided} that the \DR\ groupoids $Y\rtimes_\theta \mathbb{N}^2$ and $Y_{T} \rtimes_{\widehat{S}} \mathbb{N}$ are equivalent.

\bigskip

We now find a concrete example from a class of $2$-graphs for which we can describe the quotient explicitly. C*-algebras of higher rank graphs were first introduced by Kumjian and Pask \cite{KumjianPask2000}. A {\em $k$-graph} is a small category $\Lambda$ with a functor $d\colon\Lambda\to\mathbb{N}^d$ that satisfies the following factorization property: whenever $d(\lambda)=m+n$ for $m,n\in\mathbb{N}^d$, there exist unique $\mu,\nu\in\Lambda$ with $\lambda =\mu\nu$, $d(\mu)=m$, and $d(\nu)=n$.

One can treat a $2$-graph as a directed graph on the vertex set $\Lambda^0=d\inv (0)$, all of whose edges are colored in one of two colors, each corresponding to one of the two copies of $\mathbb{N}$ in $\mathbb{N}^2$; 
we choose red (depicted as solid edges and labeled with $e$'s) and blue (depicted as dashed edges and labeled with $f$'s). The factorization property ensures that every blue-red path can be uniquely written as a red-blue path. We will use $r$ and $s$ to denote the range and source of an edge in $\Lambda$. A graph $\Lambda$ is {\em row finite} if for each vertex $v\in\Lambda^0$, there are finitely many edges with range $v$. 
A vertex $v$ is called a {\em source} if there is no red or blue edge that has range $v$. 

In order to define the infinite path space of a $2$-graph $\Lambda$, we need some notation. Let $\Omega_2$  be the category with unit space $\mathbb{N}^2$ and morphisms $$\Omega_2^*=\{\bigl((n_1,m_1), (n_2,m_2)\bigr): n_1\leq n_2, m_1\leq m_2\}.$$ Define $r\bigl((n_1,m_1), (n_2, m_2)\bigr)=(n_2,m_2)$ and $s\bigl((n_1,m_1), (n_2,m_2)\bigr)=(n_1,m_1)$. With the degree map $d\colon \Omega_2^* \to \mathbb{N}^2$ given by $d\bigl((n_1, m_1), (n_2, m_2)\bigr)=(n_2-n_1, m_2-m_1)$, $\Omega_2$ is a $2$-graph.

Now, for any $2$-graph $\Lambda$, define its infinite path space by
\[\Lambda^\infty = \{f\colon \Omega_2\to \Lambda : f\text{ is a $2$-graph morphism}\}.\]
Fix a $2$-graph $\Lambda$ and let $Y\coloneqq \Lambda^\infty$. Let $\sigma_b,\sigma_r\colon Y\to Y$ be the shift maps along blue and red edges respectively. When the $2$-graph is row-finite and has no sources, then since $\sigma_b$ and $\sigma_r$ commute, we can consider the \DR\ groupoid $Y\rtimes_\sigma \mathbb{N}^2$. In \cite{KumjianPask2000}, this groupoid is called the {\em path groupoid}, and Kumjian and Pask show that its C*-algebra is isomorphic to the  higher-rank graph C*-algebra $C^*(\Lambda)$ defined using \CK\ relations.

\medskip

We now seek to apply our result to certain $2$-graphs. 
Recall that a $2$-graph is {\em $1$-coaligned} if, 
given a red edge $e_{1}$ and a blue edge $f_{1}$ with the same source, there exists a unique red edge $e_{2}$ and a unique blue edge $f_{2}$ such that $f_{2} e_{1} = e_{2} f_{1}$.
% for any red and blue edges $e_r, e_b$ with the same source, there exists a unique pair of red and blue edges $f_r, f_b$ such that $f_b e_r = f_r e_b$.  
It was observed in \cite{MW:arxiv} that this condition implies that  the maps $\sigma_b, \sigma_r$ $*$-commute. Let 
$\cH=Y\rtimes_{\sigma_b} \mathbb{N}$ and $\cX=Y\rtimes_{\sigma_r} \mathbb{N}$. By \cite[Proposition 5.1]{BPRRW2017} resp.\ Lemma~\ref{lemma.Deaconu.actions}, the path groupoid $Y\rtimes_{\sigma} \mathbb{N}^2$ can be realized as the \ssp\ $\cX\bowtie\cH$. 
By Lemma~\ref{lemma.Deaconu.free}, the associated \ssla\ is free if and only if $\sigma_b$ is non-periodic, which happens if and only if there is no blue cycle. 

We now consider the following class of $2$-graphs. Let $V=\sqcup_{n\in\mathbb{Z}} V_n$, where $V_n$ are disjoint non-empty sets of finitely many vertices. Consider a $2$-graph $\Lambda$ that satisfies the following conditions, the first two of which are similar to those considered in \cite{PRRS2006}:
\begin{enumerate}[label=\textup{(\arabic*)}]
    \item\label{item:direction of blues} Each blue edge $f$ has $s(f)\in V_n$ and $r(f)\in V_{n+1}$ for some $n\in\mathbb{Z}$. 
    \item\label{item:red cycles} Each vertex in $V_n$ is on exactly one red cycle, whose vertices are all in $V_n$. 
    \item\label{item:eventually constant} There exists $N$ such that $|V_n|=1$ for all $n\geq N$, and there exists exactly one blue edge from $V_n$ to $V_{n+1}$ for all $n\geq N$. 
\end{enumerate}

Suppose $\Lambda$ is a $2$-graph that satisfies all three condition; an example is given in Figure~\ref{fig:example}.

\begin{figure}[h]
    \centering

    \begin{tikzpicture}[scale=1, every node/.style={scale=1}]
    \foreach \y in {-2, 0, 2}{
        \node at (0, \y) {$\bullet$};
        \node at (4, \y) {$\bullet$};
    }
    \foreach \y in {-3, -1, 1, 3}{
        \node at (2, \y) {$\bullet$};
    }
    \node at (6,0) {$\bullet$};
    \node at (8,0) {$\bullet$};
    
    \draw[-Stealth, red] plot [smooth]  coordinates {(0,-2) (0.1,-2.3) (0.1, -2.6) (0,-2.8) (-0.1, -2.6) (-0.1,-2.3) (0, -2)};
    
    \draw[-Stealth, red] plot [smooth]  coordinates {(0,0) (0.1,0.3) (0.2, 1) (0.1, 1.7) (0,2)};
    \draw[-Stealth, red] plot [smooth]  coordinates {(0,2) (-0.1,1.7) (-0.2, 1) (-0.1, 0.3) (0,0)};
    
    \draw[-Stealth, red] plot [smooth]  coordinates {(2,3) (2-0.1,2.7) (2-0.2, 2) (2-0.1, 1.3) (2,1)};
    \draw[-Stealth, red] plot [smooth]  coordinates {(2,1) (2-0.1,0.7) (2-0.2, 0) (2-0.1, -0.7) (2,-1)};
    \draw[-Stealth, red] plot [smooth]  coordinates {(2,-1) (2+0.2,-0.4) (2+0.4, 1) (2+0.2, 2.4) (2,3)};
    \draw[-Stealth, red] plot [smooth]  coordinates {(2,-3) (2+0.1,-1+-2.3) (2+0.1, -1+-2.6) (2,-1-2.8) (2-0.1, -1-2.6) (2-0.1,-1-2.3) (2, -3)};
    
    \draw[-Stealth, red] plot [smooth]  coordinates {(4,2) (4+0.1,2-0.3) (4+0.1, 2+-0.6) (4,2-0.8) (4-0.1, 2-0.6) (4-0.1,2-0.3) (4, 2)};
    \draw[-Stealth, red] plot [smooth]  coordinates {(4,0) (4-0.1,-0.3) (4-0.2, -1) (4-0.1, -1.7) (4,-2)};
    \draw[-Stealth, red] plot [smooth]  coordinates {(4,-2) (4.1,-1.7) (4.2, -1) (4.1, -0.3) (4,0)};
    
    \foreach \x in {6, 8} {
        \draw[-Stealth, red] plot [smooth]  coordinates {(\x,0) (\x+0.1,-0.3) (\x+0.1, -.6) (\x,-.8) (\x-0.1, -.6) (\x-0.1,-.3) (\x, 0)};
    }
    
    \foreach \x in {-2, 0, 2} {
        \foreach \y in {-3, -1, 1, 3} {
            \draw[-Stealth, dashed, blue] plot [smooth]  coordinates {(0,\x) (2,\y)};
        }
    }
    \foreach \x in {-0.2,  0.2} {
        \draw[-Stealth, dashed, blue] plot [smooth]  coordinates {(0,-2) (1,-2.5+\x) (2,-3)};
    }
    \foreach \y in {-1, 1, 3} {
        \draw[-Stealth, dashed, blue] plot [smooth]  coordinates {(2,\y) (4,2)};
    }
    \foreach \y in {0, -2} {
        \foreach \x in {-0.2, 0.2} {
            \draw[-Stealth, dashed, blue] plot [smooth]  coordinates {(2,-3) (3,\y*0.5-1.5+\x) (4,\y)};
        }
    }
    \foreach \y in {-2, 0, 2}{
        \draw[-Stealth, dashed, blue] plot  coordinates {(-0.5,\y) (0, \y)};
    }
    \draw[-Stealth, dashed, blue] plot [smooth]  coordinates {(2,-3) (4,2)};
    \foreach \y in {-2,0,2}{
        \draw[-Stealth, dashed, blue] plot [smooth]  coordinates {(4,\y) (6,0)};
    }
    \draw[-Stealth, dashed, blue] plot coordinates {(6,0) (8,0)};
    \draw[-Stealth, dashed, blue] plot coordinates {(8,0) (8.5,0)};
    \node at (9,0) {$\cdots$};
    \node at (-1,0) {$\cdots$};
    \end{tikzpicture}
\caption{An example of a $2$-graph satisfying Conditions~\ref{item:direction of blues}--\ref{item:eventually constant}}\label{fig:example}
\end{figure}
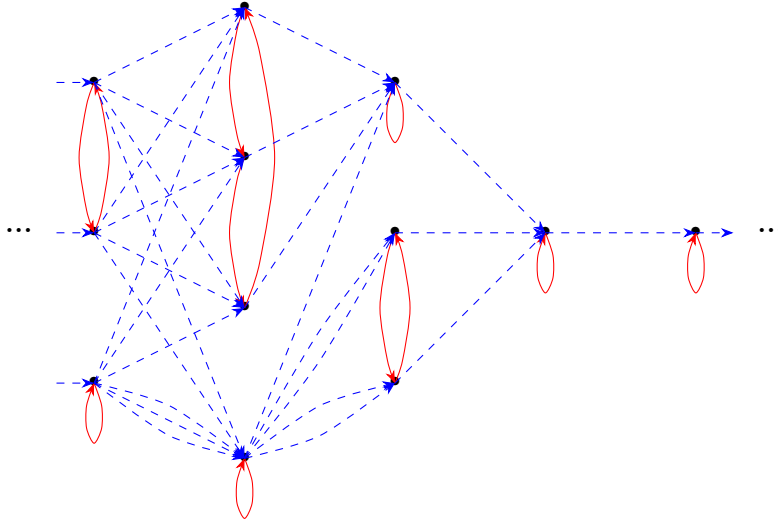

By  Condition~\ref{item:eventually constant}, we can let $V_n=\{v_n\}$  for $n\geq N$. Condition~\ref{item:red cycles} implies that $v_n$ must be on a loop of a red edge. Since there is no loop for blue edges by Condition~\ref{item:direction of blues}, $\sigma_b^{k}(x)\neq x$ for any infinite path $x$ and any $k>0$. Therefore, by Lemma~\ref{lemma.Deaconu.free}, the action $\HleftX$ is free. One can also verify that the associated \ssla\ $\HleftX$ is proper.

In this situation, we can prove that Conjecture~\ref{conj.DCGroupoid} is vacuously true. Observe that any two infinite paths $x,y\in Y$ are blue-shift equivalent: $x\sim_{\sigma_b} y$. This is because once we removed enough blue edges, the source of the infinite path will  eventually be one of the $v_n$ for $n\geq N$, and there is exactly one infinite path whose source is $v_n$. Therefore, the quotient space $Y_{\sigma_b}$ consists of a single point, the induced map $\widehat{\sigma_r}$ is automatically a local homeomorphism, and the \DR\ groupoid $Y_{\sigma_b}\rtimes_{\widehat{\sigma_r}} \mathbb{N}$   is precisely the group $\mathbb{Z}$. Discreteness of $\mathbb{Z}$ forces the bijection $\Phi$ in Conjecture~\ref{conj.DCGroupoid} to be open, so that $\cH\backslash\cX \cong \mathbb{Z}$.

Since the group C*-algebra of $\mathbb{Z}$ is $C(\mathbb{T})$, we reach the following conclusion.

\begin{corollary}\label{cor.rank2graph.main} Let $\Lambda$ be a $2$-graph that satisfies Conditions~\ref{item:direction of blues} through~\ref{item:eventually constant}. Then its C*-algebra $C^*(\Lambda)$ is Morita equivalent to $C(\mathbb{T})$. 
\end{corollary}

\begin{example} 
We will consider a specific, very simplistic $2$-graph. For each $n\in \mathbb{Z}$, we let $V_n=\{v_n\}$, and  we assume that there is exactly one blue edge $f_n$ coming out of $v_n$ with $r(f_n)=v_{n+1}$. In order for Condition~\ref{item:red cycles} to be satisfied, $v_n$ is on a red loop; we will call it $e_n$. This takes care of Conditions~\ref{item:direction of blues} and~\ref{item:eventually constant}. The factorization relation must be given by $f_n e_n = e_{n+1} f_n$.

The \CK\ relations establish that the $2$-graph C*-algebra $C^*(\Lambda)$ is the universal C*-algebra generated by projections $P_n$ and partial isometries $E_n$ and $F_n$ such that
\begin{enumerate}
    \item $E_n E_n^*=E_n^* E_n = P_n$,
    \item $F_n F_n^* = P_{n+1}$ and $F_n^* F_n=P_n$,
    \item $F_n E_n=E_{n+1} F_n$,
    \item $\sum P_n=I$. 
\end{enumerate}

Let  $\{\delta_n: n\in\mathbb{Z}\}$ be the canonical orthonormal basis of $\ell^2(\mathbb{Z})$. Let $U\in\mathcal{B}(\ell^2(\mathbb{Z}))$ be the unilateral shift, so that $U\delta_{n}=\delta_{n+1}$ and $C^*(U)=C(\mathbb{T})$. Let $I$ be the identity map on $\ell^2(\mathbb{Z})$ and let~$\varepsilon_{n}$ be the rank-one projection onto the subspace spanned by $\delta_{n}$.
Consider the following operators on $\ell^2(\mathbb{Z}) \otimes \ell^2(\mathbb{Z})$:
\begin{align*}
    P_n' &= I\otimes \varepsilon_n %(\delta_n \delta_n^*)
    , 
    &
    %\\
    E_n' &= U\otimes \varepsilon_n %(\delta_n \delta_n^*)
    , 
    &
    %\\
    F_n' &= I\otimes (U\circ \varepsilon_n)%(\delta_{n+1} \delta_n^*)
    .
\end{align*}
One can easily verify that the family $\{E_n', P_n', F_n'\}$ satisfies the \CK\ relations, so by the universal property of $C^*(\Lambda)$, there exists a unique $*$-homomorphism $C^*(\Lambda) \to C^*(E_n',F_n',P_n')=: B$ which maps $E_{n}\mapsto E_{n}', F_{n}\mapsto F_{n}',P_{n}\mapsto P_{n}' $.

There exists a natural gauge-action $\beta$ of $\mathbb{T}^2$ on $B$ where $\beta_{z_1,z_2}$ sends $E_n'\mapsto z_1 E_n'$ and $F_n'\mapsto z_2 F_n'$. 
Since $P_{n}'\neq 0$, it follows from the gauge-invariant uniqueness theorem \cite[Theorem 3.4]{KumjianPask2000} that the map $C^*(\Lambda) \to  B$ is an isomorphism. Note that $B$ is precisely $C(\mathbb{T})\otimes \cK\bigl(\ell^2(\mathbb{Z})\bigr)$ which is known to be Morita equivalent to $C(\mathbb{T})$, just as Corollary~\ref{cor.rank2graph.main} predicted.
\end{example}

\begin{example} The assumption that  there is exactly one blue edge from $V_n$ to $V_{n+1}$ for large enough~$n$ is needed for the action $\HleftX$ to be proper. To see this, consider the following simple non-example where $V_n=\{v_n\}$, and for each $n$, there is a red edge $e_n$ with $s(e_n)=r(e_n)=v_n$, and two blue edges $f_n^1, f_n^2$ where $s(f_n^i)=v_n$ and $r(f_n^i)=v_{n+1}$. Assume the $2$-graph relation is given by $f_n^i e_n = e_{n+1} f_n^i$ for all $n\in\mathbb{Z}$ and $i\in\{1,2\}$. 

We first prove that $\HleftX$ is not proper. First, we notice that each vertex $v_n$ has a unique red loop $e_n$. Therefore, its infinite path $\mu\in\Lambda^\infty$ is uniquely determined by an infinite blue path. For simplicity, we treat $\Lambda^\infty$ as the infinite path space of blue paths. Define $\mu_n=f_1^1f_2^1\cdots f_n^1 f_{n+1}^2 f_{n+2}^2 \cdots$, which is the infinite path that uses $f_j^1$ for $1\leq j\leq n$ and $f_j^2$ for $j>n$. For $i\in\{1,2\}$, let $\nu^i=f_1^i f_2^i \cdots$ be the infinite path that always uses $f_j^i$. It is clear that $\mu_n\to\nu^1$. Now, let $h_n=(\mu_n, n, \nu^2)\in \cH=Y\rtimes_{\sigma_b} \mathbb{N}$, and consider $x = (\nu^2, 0, \nu^2) \in \cX= Y\rtimes_{\sigma_r}\mathbb{N}$. We have that
\[h_n\HleftX x = (\mu_n, 0, \mu_n) \to (\nu^1, 0, \nu^1) \in \cX.\]
However, the sequence $h_n$ has no convergent subsequence in $\cH=Y\rtimes_{\sigma_b} \mathbb{N}$, since their ranges $\mu_n$ converge to $\nu^1$ while their sources are always $\nu^2$, which is not tail equivalent to $\nu^1$. The action $\HleftX$ is not proper. 

But  $\sigma_r$ is injective on $Y$ (to be precise, it is the identity map on $Y$), and thus 
%{\color{orange}%
the map
$\widehat{\sigma_r}$ 
described in Conjecture~\ref{conj.DCGroupoid}
%}
is locally injective, so that $Y_{\sigma_b} \rtimes_{\widehat{\sigma_r}} \mathbb{N}$ is a \LCH\ groupoid. We are uncertain whether the groupoid $Y_{\sigma_b} \rtimes_{\widehat{\sigma_r}} \mathbb{N}$ 
%with its topology induced as a \DR\ groupoid 
is equivalent to $Y\rtimes_\sigma\mathbb{N}^2$.
\end{example}
 
 We  point out that the Morita equivalence among graph and higher-rank graph C*-algebras is an active area of research. The recent work of Eilers et al.\ \cite{ERRS2021} completely classified all the moves on a finite graph that preserve Morita equivalence. Recent work on higher rank graphs \cite{EFGGGP2022} identified several moves that preserve Morita equivalence. Our class of examples do not fall into the moves in \cite{EFGGGP2022}, and we hope our example will further this line of research.

\appendix

\section{Exercises in Topology}

Above, the most frequently used topological fact is Fell's Criterion, which we repeat here for convenience.
\begin{proposition}[Fell's Criterion; {\cite[Prop.\ 1.1]{Wil2019}}]\label{Fell's criterion}
    Let $f\colon X\to Y$ be a surjective map between topological spaces. Then $f$ is open if and only if, whenever $\{y_{i}\}_{i\in I}$ is a net in $Y$ that converges to some $f(x)$, there exists a subnet $\{y_{j}\}_{j\in J}$ which allows a lift $\{x_{j}\}_{j\in J}$ in $X$ under $f$ that converges to $x$.
\end{proposition}

   The next lemma is an immediate consequence of $(2)\implies (1)$ in \cite[Theorem 18.1.]{Munkres2000}.
\begin{lemma}
\label{lem:exercise in top:cts maps}
    If $f\colon X\to Y$ is a function, then the following are equivalent.
    \begin{enumerate}[label=\textup{(\arabic*)}]
    
        \item\label{it:fcts} $f$ is continuous.
        \item\label{it:subnet} If $\{x_{i}\}_{i\in I}$ is a net in $X$ which converges to $x$, then there exists a subnet $\{f(x_{j})\}_{j\in J}$ of $\{f(x_{i})\}_{i\in I}$ which converges to $f(x)$.
    \end{enumerate} 
\end{lemma}

\section{Cheat sheet}\label{sec.appendix.cheatsheet}

For the convenience of the reader, we repeat the notation for our different actions and what properties they have.
Most actions and maps are summarized in Diagram~\eqref{eq:diag}.

\begin{diag}[ht]
    \centering

    \begin{tikzpicture}[scale=1.5, every node/.style={scale=1}]
    \node at (-2,0) {$\cH$};
    \node at (-2,-3) {$\cH\z $};
    
    \draw[>=stealth, ->] (-2, -0.2) -- (-2, -2.8);
    \draw[>=stealth, ->] (2, -0.2) -- (2, -2.8);
    \draw[>=stealth, ->] (-1.15, -2.15) -- (-1.85, -2.8);
    \draw[>=stealth, ->] (1.15, -2.15) -- (1.85, -2.8);
    \node at (-1.65, -2.45) {$\tilde{\rho}$};
    \node at (1.7, -2.45) {$\tilde{\sigma}$};
    
    \draw[>=stealth, ->] (-0.05, -1.15) -- (-0.05, -2.8) -- (-0.25, -3) -- (-1.65, -3);
    \draw[>=stealth, ->] (0.05, -1.15) -- (0.05, -2.8) -- (0.25, -3) -- (1.65, -3);
    \node at (-1.25, -2.85) {$\rho_\cX$};
    \node at (1.25, -2.85) {$\sigma_\cX$};
    
    \node at (2,0) {$\cG$};
    \node at (2,-3) {$\cG\z $};
    
    \node at (0, -1) {$\cX$};
    
    \node at (-1, -2) {$\cX/\cG$};
    \node at (1,-2) {$\cH\backslash \cX$};
    \draw[>={stealth[left]}, <->] (-1.85,-0.05) -- (-0.15,-0.95);
    \draw[>=stealth, <->] (1.85,-0.05) -- (0.15,-0.95);
    
    \node at (-1.45, -0.15) {$\HrightX$};
    \node at (-0.55, -0.85) {$\HleftX$};
    \node at (1.45, -0.15) {$\XleftG$};
    \node at (0.55, -0.85) {$\XrightG$};
    
    \node at (-1.45, 0.15) {$\HrightB$};
    \node at (-0.55, 0.85) {$\HleftB$};
    \node at (1.45, 0.15) {$\BleftG$};
    \node at (0.55, 0.85) {$\BrightG$};
    
    \node at (-1.65, -0.45) {$\qHrightX$};
    \node at (-1.4, -1.55) {$\qHleftX$};
    \node at (1.65, -0.45) {$\qXleftG$};
    \node at (1.4, -1.55) {$\qXrightG$};
    
    \node at (-1.4, 1.55) {$\qHleftB$};
    \node at (1.4, 1.55) {$\qBrightG$};
    
    \draw[dashed, >=stealth, ->] (-0.15,-1.05) -- (-0.9,-1.8);
    \draw[dashed, >=stealth, ->] (0.15,-1.05) -- (0.9,-1.8);
    \draw[>=stealth, <->] (-1.95,-.2) -- (-1.1,-1.8);
    \draw[>=stealth, <->] (1.95,-.2) -- (1.1,-1.8);

    \node at (0, 1) {$\cB$};
    
    \node at (-1, 2) {$\cB/\cG$};
    \node at (1,2) {$\cH\backslash \cB$};
    \draw[>=stealth, <->] (-1.85,0.05) -- (-0.15,0.95);
    
    \draw[>=stealth, <->] (1.85,0.05) -- (0.15,0.95);

    \draw[dashed, >=stealth, ->] (-0.15,1.05) -- (-0.9,1.8);
    \draw[dashed, >=stealth, ->] (0.15,1.05) -- (0.9,1.8);
    
    \draw[>=stealth, <->] (-1.95,.2) -- (-1.1,1.8);
    \draw[>=stealth, <->] (1.95,.2) -- (1.1,1.8);
    
    \draw[>=stealth, ->] (0,.8) -- (0,-.8);
    
    \draw[>=stealth, ->] (1,1.8) -- (1,-1.8);
    \draw[>=stealth, ->] (-1,1.8) -- (-1,-1.8);

    \node at (-3.1,-1) {$(\cX/\cG)\bowtie \cH$};
    \node at (3.1,-1) {$\cG\bowtie (\cH\backslash \cX)$};
    \node at (-3.1,1) {$(\cB/\cG)\BbowtieH \cH$};
    \node at (3.1,1) {$\cG\GbowtieB (\cH\backslash \cB)$};
    \draw[>=stealth, ->] (3.1,.8) -- (3.1,-.8);
    \draw[>=stealth, ->] (-3.1,.8) -- (-3.1,-.8);
    
    \end{tikzpicture}
    \caption{Summary of all actions in this paper}\label{eq:diag}
\end{diag}

\subsection{The two-way actions of \texorpdfstring{$\cH$}{H} and \texorpdfstring{$\cX$}{X}}

We use the notation
%\eqrepeatnn{eq:X-H-actions}
\begin{align*}
    \cH\cart \cX\colon &&\cH \bfp{s%_{\cH}
    }{\rho%_{\cX}
    } \cX\ni (h,x) &\mapsto h\HleftX x \in \cX \\
    \cH\calb \cX\colon &&\cH \bfp{s%_{\cH}
    }{\rho%_{\cX}
    } \cX\ni (h,x) &\mapsto h\HrightX x \in \cH
\end{align*}

Their joint properties are (see Definition~\ref{df.left.selfsimilar} and the following pages):
    \eqrepeat{item:L1}
    \eqrepeat{item:L2}
    \eqrepeat{item:L3}
    \eqrepeat{item:L4}
    \eqrepeat{item:L5}
    \eqrepeat{item:L6}
    \eqrepeat{item:L7}
    \eqrepeat{item:L8-new}
    \eqrepeat{item:L9}
    \eqrepeat{item:L10}

\subsection{The two-way actions of \texorpdfstring{$\cX$}{X} and \texorpdfstring{$\cG$}{G}}

We use the notation
%\eqrepeatnn{eq:G-X-actions}
\begin{align*}
    \cX\calt \cG\colon &&\cX \bfp{\sigma%_{\cX}
    }{r%_{\cG}
    } \cG\ni (x,s) &\mapsto x\XrightG s \in \cX\\
        \cX \carb \cG\colon &&\cX\bfp{\sigma%_{\cX}
        }{r%_{\cG}
        } \cG\ni (x,s) &\mapsto x\XleftG s \in \cG
\end{align*}

Their joint properties are (see Definition~\ref{df.right.selfsimilar} and the following pages):
    \eqrepeat{item:R1}
    \eqrepeat{item:R2}
    \eqrepeat{item:R3}
    \eqrepeat{item:R4}
    \eqrepeat{item:R5}
    \eqrepeat{item:R6}
    \eqrepeat{item:R7}
    \eqrepeat{item:R8-new}
    \eqrepeat{item:R9}
    \eqrepeat{item:R10}

\subsection{Actions that are \intune}
In Definition~\ref{df.compatible}, we defined two sets of actions $\cH\arrowlssa\cX\arrowrssa\cG$ as above to be {\em \intune} if they satisfy
%\txtrepeat{properties of two sets of compatible actions}
\begin{enumerate}[leftmargin=1.5cm, label=\textup{(C\arabic*)}, start=0]
    \item  $\sigma_{\cX}(h\HleftX x) = \sigma_{\cX}(x)$ in $\cG\z$ and $
        \rho_{\cX}(x)
        =
        \rho_{\cX}(x\XrightG s)$ in $\cH\z$,
        \item $h\HleftX (x\XrightG s)=(h\HleftX x)\XrightG s$
            in $\cX$,
        \item $(h\HleftX x)\XleftG s=x\XleftG s$
            in $\cG$, and
        \item $h\HrightX (x\XrightG s)=h\HrightX x$
            in $\cH$.
\end{enumerate}    

\subsection{The induced actions on quotients} The above induce the following (see also Proposition~\ref{prop.ssp actions on quotients}):
% 
% 
    %\eqrepeatnn{eq:H-X over G-actions}
    \begin{align*}
        \cH \cart (\cX/\cG)\colon&&\cH \bfp{s}{\tilde{\rho}} (\cX/\cG)
        \ni (h,x\XrightG \cG) &\mapsto h\qHleftX (x\XrightG \cG)
        \coloneqq  (h \HleftX x)\XrightG \cG \in \cX/\cG
        \\
        \cH \calb (\cX/\cG)\colon&&\cH \bfp{s}{\tilde\rho} (\cX/\cG)\ni (h,x\XrightG \cG)
        &\mapsto h\qHrightX (x\XrightG \cG) \coloneqq 
        h\HrightX x \in \cH
    \end{align*}
    
    %\eqrepeatnn{eq:H_under_X-G-actions}
    \begin{align*}
        (\cH\backslash \cX) \calt \cG\colon
        &&
        (\cH\backslash \cX) \bfp{\tilde\sigma}{r} \cG\ni (\cH\HleftX x, s)
        &\mapsto (\cH\HleftX x) \qXrightG s\coloneqq \cH\HleftX (x\XrightG s) \in \cH\backslash \cX
        \\
        (\cH\backslash \cX) \carb \cG\colon
        &&
        (\cH\backslash \cX) \bfp{\tilde\sigma}{r} \cG\ni (\cH\HleftX x, s)
        &
        \mapsto (\cH\HleftX x) \qXleftG s\coloneqq x\XleftG s \in \cG
    \end{align*}

\subsubsection{Fell bundles et cetera}\label{ssec:Fell bdl properties}

\begin{definition}[{\cite[Definition 2.8]{BE2012}}]
    An \usc\ Banach bundle $\cB=(q_{\cB}\colon B\to\cG)$ over a (\LCH\ \etale) groupoid~$\cG$ is called {\em Fell bundle} 
    if it comes with continuous maps
    \[
        \cdot\;\colon
         \cB^{(2)}:=
         \left\{
            (a,b)\in B\times B  : (q_{\cB}(a),q_{\cB}(b))\in \cG^{(2)}
         \right\}
        \to
        B
        \quad\text{and}\quad
        {}^{\ast}\colon B \to B
    \]
    such that:
    \begin{enumerate}[label=\textup{(F\arabic*)}]
        \item\label{cond.F1} For each $(x,y)\in \cG^{(2)}$, $\cB_{x}\cdot \cB_y\subseteq\cB_{xy}$, i.e.\ $q_{\cB}(b\cdot c)=q_{\cB}(b) q_{\cB}(c)$ for all $(b,c)\in  \cB^{(2)}$.
        \item\label{cond.F2} The multiplication is bilinear.
        \item\label{cond.F3} The multiplication is associative, whenever it is defined.
        \item\label{cond.F4} If $(b,c)\in  \cB^{(2)}$, then $\|b\cdot c\|\leq\|b\|\|c\|$, where the norm is the Banach space norm of the respective fiber. 
        \item\label{cond.F5} For any $x\in \cG$, $\cB_{x}^*\subseteq \cB_{x\inv }$.
        \item\label{cond.F6} The involution map $b\mapsto b^*$ is conjugate linear.
        \item\label{cond.F7} If $(b,c)\in  \cB^{(2)}$, then $(b\cdot c)^*=c^*\cdot  b^*$. 
        \item\label{cond.F8} For any $b\in B$, $b^{**}=b$.
        \item\label{cond.F9} For any $b\in B$, $\|b^*\cdot  b\|=\|b\|^2=\|b^*\|^2$. 
        \item\label{cond.F10} For any $b\in B$, $b^*\cdot  b\geq 0$ in the fiber of~$\cB$ over $s_{\cB}(b)$. %$\cB_{s_{\cB}(b)}$. 
    \end{enumerate}
    We call~$\cB$ {\em saturated} if we have an equality of sets in Condition~\ref{cond.F1}. We will often write~$bc$ for~$b\cdot c$.
\end{definition}

\begin{definition}[{\cite{MW2008}}]\label{df.USCBb-action}
    Suppose that $\cA=(q_{\cA}\colon A\to \cG)$ is a Fell bundle over a (\LCH\ \etale) groupoid $\cG$, $X$ is a left $\cG$-space with momentum map $\rho_{\cX}\colon \cX\to \cG\z$, and $\cM=(q_{\cM}\colon M \to X)$ is a \USCBb; we let $\rho_{\cM}\coloneqq \rho_{\cX}\circ q_{\cM}$.  Then we say that {\em $\cA$
    acts on \textup(the left\textup) of $\cM$} if there is a continuous map 
    \(
      \cA\bfp{s}{\rho} \cM \to \cM,\,(a,m)\mapsto
    a\cdot m,\) such that
    \begin{enumerate}[leftmargin=2cm,label=\textup{(FA\arabic*)}]
    \item\label{item:FA:fiber} $q_{\cM}(a\cdot m)=q_{\cA}(a)\cdot q_{\cM}(m)$,
    \item\label{item:FA:assoc} $a\cdot (a\cdot m)=(ab)\cdot m$ for all appropriate $a\in A$, and
    \item\label{item:FA:norm} $\norm{a\cdot m}\leq\norm{a}\,\norm{m}$. 
    \end{enumerate}
\end{definition}

\begin{definition}[{\cite[Definition 6.1]{MW2008}}]\label{df.FBequivalence}
  Suppose that $\cH$, $\cG$ are \LCH\ \etale\ groupoids, that $\cX$ is an $(\cH ,\cG )$-equivalence with momentum maps $\rho_{\cX}\colon \cX\to \cH\z$ resp.\ $\sigma_{\cX}\colon \cX\to \cG\z$, that $\cA=(q_{\cA}\colon A\to \cH)$ and $\cC=(q_{\cC}\colon C\to \cG)$ are saturated Fell
  bundles, and that $\cM=(q_{\cM}\colon M\to \cX)$ is a \USCBb;  we let $\rho_{\cM}\coloneqq \rho_{\cX}\circ q_{\cM}$ and $\sigma_{\cM}\coloneqq \sigma_{\cX}\circ q_{\cM}$.  We say that $\cM$ is an {\em $\cA\sme\cC$-equivalence}
  if the following conditions hold.
  \begin{enumerate}[leftmargin=1.2cm,label=\textup{(FE\arabic*)}]
  \item\label{item:FE:actions} There is a left $\cA$-action and a right $\cC$-action on $\cM$
    such that $a\cdot (m\cdot c)=(a\cdot m)\cdot c$ for all $a\in A $,
    $m\in M $, and $c\in C $, wherever it makes sense.
  \item\label{item:FE:ip} There are sesquilinear maps
  \[
  \begin{tikzcd}[row sep=tiny, column sep = small]
    \lip\cA<\cdot,\cdot>\colon&[-15pt] M \bfp{{\sigma}
    }{\sigma
    } M \ar[r]& A, && \rip\cC<\cdot,\cdot>\colon&[-15pt] M \bfp{{\rho}
    }{\rho
    } M \ar[r]& C
    \\
    &(m_{1},m_{2})\ar[r,mapsto]& \lip\cA<m_{1},m_{2}>, && &(m_{1},m_{2})\ar[r,mapsto]& \rip\cC<m_{1},m_{2}>
  \end{tikzcd}
  \]
  such that for all appropriately chosen $m_{i}\in M$, $a\in A$, and $c\in C$, we have
    \begin{enumerate}[leftmargin=.5cm,label=\textup{(FE2.\alph*)}]
    \item\label{item:FE:ip:fiber}
        $q_{\cM}(m_{1})=q_{\cA}\bigl(\lip\cA<m_{1},m_{2}>\bigr) \HleftX q_{\cM}(m_{2})$
        and  
        $q_{\cM}(m_{1})\XrightG q_{\cC}\bigl(\rip\cC<m_{1},m_{2}>\bigr)= q_{\cM}(m_{2})$,
    \item\label{item:FE:ip:adjoint} $\lip\cA<m_{1},m_{2}>^{*}=\lip\cA<m_{2},m_{1}>$\; and \;
      $\rip\cC<m_{1},m_{2}>^{*}=\rip\cC<m_{2},m_{1}>$,
    \item\label{item:FE:ip:C*linear} $\lip\cA<a\cdot m_{1},m_{2}>=a\lip\cA<m_{1},m_{2}>$\; and \;
      $\rip\cC<m_{1},m_{2}\cdot c>=\rip\cC<m_{1},m_{2}>c$, and
    \item\label{item:FE:ip:compatibility} $\lip\cA<m_{1},m_{2}>\cdot m_{3}=m_{1}\cdot\rip\cC<m_{2},m_{3}>$.
    \end{enumerate}
  \item\label{item:FE:SMEs} With the actions coming from \ref{item:FE:actions} and the inner products
    coming from \ref{item:FE:ip}, each $M(x)$ is a $A\bigl({\rho_{\cX}}%r
    (x)\bigr)\sme
    C\bigl({\sigma_{\cX}}%s
    (x)\bigr)$-\ib.
  \end{enumerate}
\end{definition}

\subsection{\Ss\ actions on Fell bundles}
% On 
% \[\cH \bfp{s}{\rho} \cB=\{(h,b)
% \in \cH \times \cB
% : s_{\cH}(h)=\rho_{\cB}(b)\}\]
 
% {%
% \em
% left \ss~$\cH$-action on~$\cB$
% }%
% is a continuous
A continuous
map \[\mvisiblespace\HleftB\mvisiblespace \colon \cH \bfp{s%_{\cH}
}{\rho%_{\cB}
} \cB\to \cB\]
is a \ssla\ of the groupoid $\cH$ on the Fell bundle $\cB=(q_{\cB}\colon B\to\cX)$ if it satisfies (see also Definition~\ref{df.ss.FellBundle.action.left}):
%\txtrepeat{lssa on Fell bdl}
\begin{enumerate}[label=\textup{(B\arabic*)}]
    \item For any $(h,x)\in \cH \bfp{s}{\rho} \cX$, the map $h\HleftB\mvisiblespace$ maps $\cB_x$ into $\cB_{h\HleftX x}$ and is linear.
    \item For any 
   	 $(k,h)\in\cH^{(2)}$, we have $k\HleftB (h \HleftB \mvisiblespace)=(kh)\HleftB \mvisiblespace$.
    \item For any $u\in \cH\z $, the map $u \HleftB\mvisiblespace$ is the identity.
    \item For any $(b,c)\in \cB^{(2)}$ such that $(h,bc)\in \cH \bfp{s%_{\cH}
    }{\rho%_{\cB}
    } \cB$, we have
    \[h\HleftB (bc)=(h\HleftB b)\left[ (h\HrightX q_{\cB}(b) )\HleftB c\right].\]
    \item For any 
    	$(h,b)\in \cH \bfp{s%_{\cH}
    	}{\rho%_{\cB}
    	} \cB$,  we have
    \[(h\HleftB b)^* = [h\HrightX q_{\cB}(b) ]\HleftB b^*.\]
\end{enumerate}


\begin{thebibliography}{10}

\bibitem{AA2005}
M.~Aguiar and N.~Andruskiewitsch.
\newblock Representations of matched pairs of groupoids and applications to
  weak {H}opf algebras.
\newblock In {\em Algebraic structures and their representations}, volume 376
  of {\em Contemp. Math.}, pages 127--173. Amer. Math. Soc., Providence, RI,
  2005.

\bibitem{BPRRW2017}
N.~Brownlowe, D.~Pask, J.~Ramagge, D.~Robertson, and M.~F. Whittaker.
\newblock Zappa-{S}z\'{e}p product groupoids and {$C^*$}-blends.
\newblock {\em Semigroup Forum}, 94(3):500--519, 2017.

\bibitem{BRRW}
N.~Brownlowe, J.~Ramagge, D.~Robertson, and M.~F. Whittaker.
\newblock Zappa-{S}z\'ep products of semigroups and their {$C^\ast$}-algebras.
\newblock {\em J. Funct. Anal.}, 266(6):3937--3967, 2014.

\bibitem{BE2012}
A.~Buss and R.~Exel.
\newblock Fell bundles over inverse semigroups and twisted \'{e}tale groupoids.
\newblock {\em J. Operator Theory}, 67(1):153--205, 2012.

\bibitem{Deaconu} Deaconu, V. On groupoids and {$C^*$}-algebras from self-similar actions. {\em New York J. Math.}. \textbf{27} pp. 923-942 (2021)

\bibitem{DL2021}
A.~Duwenig and B.~Li.
\newblock The {Z}appa-{S}z\'{e}p product of a {F}ell bundle and a groupoid.
\newblock {\em J. Funct. Anal.}, 282(1):Paper No. 109268, 42, 2022.

\bibitem{DL:ZS2}
A.~Duwenig and B.~Li.
\newblock Equivalence of {F}ell bundles is an equivalence relation.
\newblock {\em M\"{u}nster J. Math.}, 16(1):95--145, 2023.

\bibitem{EFGGGP2022}
C.~Eckhardt, K.~Fieldhouse, D.~Gent, E.~Gillaspy, I.~Gonzales, and D.~Pask.
\newblock Moves on {$k$}-graphs preserving {M}orita equivalence.
\newblock {\em Canad. J. Math.}, 74(3):655--685, 2022.

\bibitem{ERRS2021}
S. Eilers, G.~Restorff, E.~Ruiz, and A.~P.~W. S{\o}rensen.
\newblock The complete classification of unital graph {$C^*$}-algebras:
  geometric and strong.
\newblock {\em Duke Math. J.}, 170(11):2421--2517, 2021.

\bibitem{ExelFellBundle}
R.~Exel.
\newblock {\em Partial dynamical systems, {F}ell bundles and applications},
  volume 224 of {\em Mathematical Surveys and Monographs}.
\newblock American Mathematical Society, Providence, RI, 2017.

\bibitem{EP2017}
R.~Exel and E.~Pardo.
\newblock Self-similar graphs, a unified treatment of {K}atsura and
  {N}ekrashevych {$\rm C^*$}-algebras.
\newblock {\em Adv. Math.}, 306:1046--1129, 2017.

\bibitem{FellBundleBook}
J.~M.~G. Fell.
\newblock {\em Induced representations and {B}anach {$\sp*$}-algebraic
  bundles}.
\newblock Lecture Notes in Mathematics, Vol. 582. Springer-Verlag, Berlin-New
  York, 1977.
\newblock With an appendix due to A. Douady and L. Dal Soglio-H\'{e}rault.

\bibitem{Goehle:thesis}
G.~Goehle.
\newblock {\em Groupoid crossed products}.
\newblock ProQuest LLC, Ann Arbor, MI, 2009.
\newblock Thesis (Ph.D.)--Dartmouth College.

\bibitem{Green1978}
P.~Green.
\newblock The local structure of twisted covariance algebras.
\newblock {\em Acta Math.}, 140(3-4):191--250, 1978.

\bibitem{Grigorchuk1984}
R.~I. Grigorchuk.
\newblock Degrees of growth of finitely generated groups and the theory of
  invariant means.
\newblock {\em Izv. Akad. Nauk SSSR Ser. Mat.}, 48(5):939--985, 1984.

\bibitem{Grigorchuk1980}
R.~I. Grigor\v{c}uk.
\newblock On {B}urnside's problem on periodic groups.
\newblock {\em Funktsional. Anal. i Prilozhen.}, 14(1):53--54, 1980.

\bibitem{HKQW2021}
L.~Hall, S.~Kaliszewski, J.~Quigg, and D.~P. Williams.
\newblock Groupoid semidirect product {F}ell bundles. {I}. {A}ctions by
  isomorphism, 
  \newblock {\em J. Operator Theory}, 89(1):125–153, 2023.

\bibitem{KMQW2010}
S.~Kaliszewski, P.~S. Muhly, J.~Quigg, and D.~P. Williams.
\newblock Coactions and {F}ell bundles.
\newblock {\em New York J. Math.}, 16:315--359, 2010.

\bibitem{KMQW2013}
S.~Kaliszewski, P.~S. Muhly, J.~Quigg, and D.~P. Williams.
\newblock Fell bundles and imprimitivity theorems.
\newblock {\em M\"{u}nster J. Math.}, 6(1):53--83, 2013.

\bibitem{KWR2001:Skew}
S.~Kaliszewski, J.~Quigg, and I.~Raeburn.
\newblock Skew products and crossed products by coactions.
\newblock {\em J. Operator Theory}, 46(2):411--433, 2001.

\bibitem{Kumjian1998}
A.~Kumjian.
\newblock Fell bundles over groupoids.
\newblock {\em Proc. Amer. Math. Soc.}, 126(4):1115--1125, 1998.

\bibitem{KumjianPask2000}
A.~Kumjian and D.~Pask.
\newblock Higher rank graph {$C^\ast$}-algebras.
\newblock {\em New York J. Math.}, 6:1--20, 2000.

\bibitem{LY2019b}
H.~Li and D.~Yang.
\newblock Self-similar {$k$}-graph {${\rm C}^*$}-algebras.
\newblock {\em Int. Math. Res. Not. IMRN}, (15):11270--11305, 2021.

\bibitem{Mackey:MackeyMachine}
G.~W. Mackey.
\newblock Unitary representations of group extensions. {I}.
\newblock {\em Acta Math.}, 99:265--311, 1958.

\bibitem{MW:arxiv}
B.~Maloney and P.~N. Willis.
\newblock Examples of $\ast$-commuting maps.
\newblock {\em https://arxiv.org/abs/1101.3795}, 2011.

\bibitem{MRW:Grpd}
P.~S. Muhly, J.~N. Renault, and D.~P. Williams.
\newblock Equivalence and isomorphism for groupoid {$C^\ast$}-algebras.
\newblock {\em J. Operator Theory}, 17(1):3--22, 1987.

\bibitem{MW2008}
P.~S. Muhly and D.~P. Williams.
\newblock Equivalence and disintegration theorems for {F}ell bundles and their
  {$C^*$}-algebras.
\newblock {\em Dissertationes Math.}, 456:1--57, 2008.

\bibitem{MW:Renaults}
P.~S. Muhly and D.~P. Williams.
\newblock {\em Renault's equivalence theorem for groupoid crossed products},
  volume~3 of {\em New York Journal of Mathematics. NYJM Monographs}.
\newblock State University of New York, University at Albany, Albany, NY, 2008.

\bibitem{Munkres2000}
J.~R. Munkres.
\newblock {\em Topology}.
\newblock Prentice Hall, Inc., Upper Saddle River, NJ, 2000.
\newblock Second edition of [ MR0464128].

\bibitem{MurrayVN1936}
F.~J. Murray and J.~Von~Neumann.
\newblock On rings of operators.
\newblock {\em Ann. of Math. (2)}, 37(1):116--229, 1936.

\bibitem{Nekrashevych2009}
V.~Nekrashevych.
\newblock {$C^*$}-algebras and self-similar groups.
\newblock {\em J. Reine Angew. Math.}, 630:59--123, 2009.

\bibitem{PRRS2006}
D.~Pask, I.~Raeburn, M.~R{\o}rdam, and A.~Sims.
\newblock Rank-two graphs whose {$C^*$}-algebras are direct limits of circle
  algebras.
\newblock {\em J. Funct. Anal.}, 239(1):137--178, 2006.

\bibitem{Raeburn1988}
I.~Raeburn.
\newblock Induced {$C^*$}-algebras and a symmetric imprimitivity theorem.
\newblock {\em Math. Ann.}, 280(3):369--387, 1988.

\bibitem{Renault1987Fr}
J.~Renault.
\newblock Repr\'{e}sentation des produits crois\'{e}s d'alg\`ebres de
  groupo\"{\i}des.
\newblock {\em J. Operator Theory}, 18(1):67--97, 1987.

\bibitem{Rieffel:Mackey1}
M.~A. Rieffel.
\newblock On the uniqueness of the {H}eisenberg commutation relations.
\newblock {\em Duke Math. J.}, 39:745--752, 1972.

\bibitem{Rieffel:Mackey2}
M.~A. Rieffel.
\newblock Induced representations of {$C\sp{\ast} $}-algebras.
\newblock {\em Advances in Math.}, 13:176--257, 1974.

\bibitem{Rieffel1974}
M.~A. Rieffel.
\newblock Morita equivalence for {$C\sp{\ast} $}-algebras and {$W\sp{\ast}
  $}-algebras.
\newblock {\em J. Pure Appl. Algebra}, 5:51--96, 1974.

\bibitem{Rosenberg:Mackey}
J.~Rosenberg.
\newblock {$C^\ast$}-algebras and {M}ackey's theory of group representations.
\newblock In {\em {$C^\ast$}-algebras: 1943--1993 ({S}an {A}ntonio, {TX},
  1993)}, volume 167 of {\em Contemp. Math.}, pages 150--181. Amer. Math. Soc.,
  Providence, RI, 1994.

\bibitem{SW2016}
A.~Sims and D.~P. Williams.
\newblock The primitive ideals of some \'{e}tale groupoid {$C^*$}-algebras.
\newblock {\em Algebr. Represent. Theory}, 19(2):255--276, 2016.

\bibitem{Starling2015}
C.~Starling.
\newblock Boundary quotients of {$\rm C^*$}-algebras of right {LCM} semigroups.
\newblock {\em J. Funct. Anal.}, 268(11):3326--3356, 2015.

\bibitem{WilliamsBook}
D.~P. Williams.
\newblock {\em Crossed products of {$C{^\ast}$}-algebras}, volume 134 of {\em
  Mathematical Surveys and Monographs}.
\newblock American Mathematical Society, Providence, RI, 2007.

\bibitem{Wil:Haar}
D.~P. Williams.
\newblock Haar systems on equivalent groupoids.
\newblock {\em Proc. Amer. Math. Soc. Ser. B}, 3:1--8, 2016.

\bibitem{Wil2019}
D.~P. Williams.
\newblock {\em A tool kit for groupoid {$C^*$}-algebras}, volume 241 of {\em
  Mathematical Surveys and Monographs}.
\newblock American Mathematical Society, Providence, RI, 2019.

\bibitem{Yamagami:preprint}
S.~Yamagami.
\newblock On the ideal structure of {C*}-algebras over locally compact
  groupoids.
\newblock {\em Preprint}, 1987.

\bibitem{Yamagami1990}
S.~Yamagami.
\newblock On primitive ideal spaces of {$C^*$}-algebras over certain locally
  compact groupoids.
\newblock In {\em Mappings of operator algebras ({P}hiladelphia, {PA}, 1988)},
  volume~84 of {\em Progr. Math.}, pages 199--204. Birkh\"{a}user Boston,
  Boston, MA, 1990.

\end{thebibliography}
\end{document}